\documentclass[12pt, a4paper, reqno]{amsart}
\usepackage{longtable} 
\usepackage{hyperref}
\usepackage[T1]{fontenc}
\usepackage[utf8]{inputenc}
\usepackage[english]{babel}
\usepackage{textcomp}
\usepackage{dsfont}
\usepackage{latexsym}
\usepackage{amssymb}
\usepackage{amsthm}
\usepackage{amsmath}
\DeclareMathAlphabet{\mathpzc}{OT1}{pzc}{m}{en}
\usepackage{yfonts}
\usepackage{xfrac}
\usepackage{newlfont}
\usepackage{graphicx}
\usepackage{mathtools}
\usepackage{comment}
\usepackage{indentfirst}
\usepackage{braket}
\usepackage{mathrsfs}
\usepackage{xcolor} 

\usepackage{etoolbox}

\usepackage{scalerel}[2014/03/10]
\usepackage[usestackEOL]{stackengine}
\newcommand{\dashint}{\,\ThisStyle{\ensurestackMath{%
			\stackinset{c}{.2\LMpt}{c}{.5\LMpt}{\SavedStyle-}{\SavedStyle\phantom{\int}}}%
		\setbox0=\hbox{$\SavedStyle\int\,$}\kern-\wd0}\int}

\DeclareMathOperator{\card}{Card}

\DeclareMathOperator{\supp}{Supp}

\DeclareMathOperator{\ad}{ad}

\newcommand{\bmo}{bmo}

\newcommand{\Wp}{\mathscr{W}}
\newcommand{\Lp}{\mathscr{L}}

\newcommand{\ee}{\mathrm{e}}

\newcommand{\loc}{\mathrm{loc}}
\newcommand{\vect}[1]{\mathbf{{#1}}}
\newcommand{\dd}{\mathrm{d}}

\DeclarePairedDelimiter{\abs}{\lvert}{\rvert}

\DeclarePairedDelimiter{\norm}{\lVert}{\rVert}

\let\originalleft\left
\let\originalright\right
\renewcommand{\left}{\mathopen{}\mathclose\bgroup\originalleft}
\renewcommand{\right}{\aftergroup\egroup\originalright}

\newcommand{\grado}{\Df}
\newcommand{\N}{\mathds{N}}
\newcommand{\Z}{\mathds{Z}}
\newcommand{\Q}{\mathds{Q}}

\newcommand{\C}{\mathds{C}}

\newcommand{\R}{\mathds{R}}

\newcommand{\gf}{\mathfrak{g}}

\newcommand{\Banach}{Z}
\newcommand{\BanachSpace}{Banach space}

\newcommand{\Df}{\mathfrak{D}}

\newcommand{\Lf}{\mathfrak{L}}

\newcommand{\Dc}{\mathcal{D}}

\newcommand{\Hc}{\mathcal{H}}

\newcommand{\Kc}{\mathcal{K}}
\newcommand{\Lc}{\mathcal{L}}
\renewcommand{\Mc}{\mathcal{M}}
\newcommand{\Nc}{\mathcal{N}}

\newcommand{\Pc}{\mathcal{P}}

\newcommand{\Sc}{\mathcal{S}}

\newcommand{\Sr}{\mathscr{S}}

\newcommand{\meg}{\leqslant}
\newcommand{\Meg}{\geqslant}
\newcommand{\eps}{\varepsilon}
\renewcommand{\phi}{\varphi}
\newcommand{\mi}{\mu}

\thanks{{\em Math Subject Classification 2020}: 46E30, 22E30.}
\thanks{The authors are members of the 	Gruppo Nazionale per l'Analisi
	Matematica, la Probabilit\`a e le	loro Applicazioni (GNAMPA) of
	the Istituto Nazionale di Alta Matematica (INdAM). The authors were partially funded by the INdAM-GNAMPA Project CUP\_E5324001950001. 
	A large portion of this work was completed while the second author was a post-Doc at the Università degli Studi di Milano, +Dip Eccellenza.
}

\begin{document}
	
	\title[Local Hardy and bmo Spaces]{Local Hardy and bmo Spaces on Filtered Lie Groups}
	
	\author[M.\ Calzi]{Mattia Calzi} 
	\address{Dipartimento di Matematica, Universit\`a degli Studi di
		Milano, Via C. Saldini 50, 20133 Milano, Italy}
	\email{{\tt mattia.calzi@unimi.it}}
	
	\author[E.\ Rizzo]{Elena Rizzo}
	\address{Departamento de An\'alisis Matem\'atico, Facultad de Ciencias, Universidad de M\'alaga, Bulevar Louis Pasteur, s/n, Campus de Teatinos, 29071 Málaga, Spain}
	\email{{\tt  elena.rizzo@uma.es}}
	
	\theoremstyle{definition}
	\newtheorem{deff}{Definition}[section]
	
	\newtheorem{oss}[deff]{Remark}
	
	\newtheorem{ass}[deff]{Assumptions}
	
	\newtheorem{nott}[deff]{Notation}

	\theoremstyle{plain}
	\newtheorem{teo}[deff]{Theorem}
	
	\newtheorem{lem}[deff]{Lemma}
	
	\newtheorem{prop}[deff]{Proposition}
	
	\newtheorem{cor}[deff]{Corollary}
	
	\begin{abstract}
		 We develop a theory of local Hardy and $bmo$ spaces on a connected Lie group with polynomial growth whose Lie algebra is endowed with a filtration. Under suitable assumptions, we prove that the local Hardy spaces enjoy an appropriate atomic decomposition, are compatible with Calder\'on--Zygmund operators, and can be identified with certain Triebel--Lizorkin spaces, as it happens in the classical settings.
	\end{abstract}
	
	\maketitle

\section{Introduction}

Hardy spaces play a central role in harmonic analysis, providing, particularly in the range \(0<p\meg1\), a natural substitute for \(L^p\) spaces in the study of singular integral operators and related problems. 
Originating in the classical theory of complex functions, a real-variable theory of Hardy spaces began with the work of Stein and Weiss \cite{SteinWeiss} on systems of conjugate harmonic functions in $\mathbb{R}^n$. It was then developed by Fefferman and Stein in \cite{FeffermanStein}, where $H^p(\mathbb{R}^n)$ is characterized through maximal functions, while in \cite{Fefferman} Fefferman proved the duality between $H^1$ and $BMO$. The atomic decomposition established by Coifman \cite{Coifman} and Latter \cite{Latter} and then developed by Coifman and Weiss \cite{CoifmanWeiss2} provided a flexible framework for the study of Hardy spaces and, in particular, for proving the boundedness of Calder\'on--Zygmund operators on $H^1$.
More precisely, this decomposition represents functions belonging to the Hardy space in terms of atoms, i.e., localized building blocks satisfying suitable size and vanishing moment conditions. Although such global cancellation is crucial for the boundedness of operators such as the Riesz transforms and Calder\'on--Zygmund operators, it also imposes a strong global rigidity. In particular, the spaces $H^p(\mathbb{R}^n)$ are not stable under multiplication by smooth cutoff functions, and it turns out that the Schwartz space is not contained in the Hardy space.
To overcome these limitations, Goldberg introduced in~\cite{Goldberg} the local Hardy spaces $h^p(\mathbb{R}^n)$, in which the structure of the Hardy space is only retained at small scales. In particular, one has $H^1(\mathbb{R}^n)\subsetneq h^1(\mathbb{R}^n)\subsetneq L^1(\mathbb{R}^n)$, while $h^p(\mathbb{R}^n)=L^p(\mathbb{R}^n)$ for $1<p<\infty$.
In the case \(p=1\), the resulting space \(h^1(\mathbb{R}^n)\) admits an equivalent description in terms of local Riesz transforms, and for any $0<p\meg1$ the space $h^p(\mathbb{R}^n)$ can be characterized by maximal functions and by an atomic decomposition. In the maximal function formulation, the usual supremum over $(0,\infty)$ is replaced by a supremum over $(0,1]$, while in the atomic decomposition only atoms with small support need to enjoy all the vanishing moments. This modification has important consequences. Unlike the global space \(H^1(\mathbb{R}^n)\), the local space \(h^1(\mathbb{R}^n)\) is stable under multiplication by smooth compactly supported functions, making it particularly well suited to localization arguments. 
However, a Mihlin–Hörmander type Fourier multiplier requires stronger regularity near the origin to define an endomorphism of $h^p$, $0<p\meg1$.
Moreover, in analogy with the classical duality between $H^1$ and $BMO$, the dual space of $h^1(\mathbb{R}^n)$ is the local space $bmo(\mathbb{R}^n)$. Its defining condition controls the mean oscillation on small balls, whereas on large balls it only requires a bound on the average of $|f|$.
Finally, both the global and the local Hardy spaces fit naturally into the broader framework of Triebel--Lizorkin spaces. More precisely, they correspond to the homogeneous and non-homogeneous Triebel--Lizorkin spaces, respectively, through the identifications
\[
H^p(\mathbb R^n)=\dot F^{p,2}_0(\mathbb R^n),
\qquad
h^p(\mathbb R^n)=F^{p,2}_0(\mathbb R^n).
\]
This distinction reflects the different roles played by low frequencies in the two theories.

In recent decades, substantial effort has been devoted to extending local Hardy spaces to more general contexts. In \cite{Taylor} the author introduced local $h^1$ spaces on Riemannian manifolds, while the authors of \cite{CMIM} studied local Hardy spaces of differential forms on Riemannian manifolds for $p \ge 1$. For $p=1$, local Hardy spaces have been developed on doubling metric measure spaces in \cite{DafniYue}, and on locally doubling manifolds in \cite{MartiniMedaVallarino}. Moreover, local Hardy spaces $h^1_L$ associated with specific differential operators $L$ were established on $\mathbb{R}^n$ in \cite{JiangYangZhou} and on metric measure spaces in \cite{GongLiYan}. For $p \in (0,1]$, characterizations of $h^p(\mathbb{R}^n)$ via singular integrals were studied in \cite{PelosoSecco}.
Despite these and many other advances, the extensions of local Hardy spaces $h^p$ for $p \in (0,1)$ to more general settings has, to the best of our knowledge, remained largely unexplored.
This is both due to their homogeneous nature, which imposes strong conditions on the behavior at infinity of the underlying space, and to the difficulty of understanding how to implement a vanishing moment condition in greater generality. 
In fact, while this is not necessary for all $p<1$, it becomes essential when $p$ is sufficiently small.

Conversely, Hardy spaces (in their global version) were successfully extended by Folland and Stein in~\cite{FollandStein} to homogeneous groups, that is, to simply connected nilpotent Lie groups $G$ whose Lie algebra $\gf$ is endowed with a family $(\delta_r)$ of automorphic non-isotropic dilations. Note that the existence of a family of dilations $(\delta_r)$ on a Lie algebra $\gf$ is equivalent to requiring $\gf$ to have a graduation $(\gf_\lambda)_{\lambda>0}$, in the sense that
\[
\gf=\bigoplus_{\lambda>0} \gf_\lambda\,\text{ and }\,[\gf_\lambda, \gf_\mi]\subseteq \gf_{\lambda+\mi}\text{ for every }\lambda,\mi>0.
\]
Explicitly, $\gf_\lambda$ is simply the space of $X\in \gf$ such that $\delta_r(X)= r^\lambda X$; equivalently, $\gf_\lambda$ is the eigenspace of $\log \delta_r$ relative to the eigenvalue $\lambda \log r$. 
In the same framework of homogeneous groups, a definition of homogeneous Triebel--Lizorkin spaces was recently introduced in~\cite{Jordi}, where a suitable correspondence between $H^p$ and $\dot F^{p,2}_0$ is established.

The considerations presented above provide the starting point for the present work. On the one hand, the theory of local Hardy spaces $h^p$ for $0<p<1$ remains much less developed in general settings than in the Euclidean case. On the other hand, the correspondence established in~\cite{Jordi} between global Hardy spaces and homogeneous Triebel--Lizorkin spaces suggests the possibility of an analogous identification between local Hardy spaces and non-homogeneous Triebel--Lizorkin spaces also in more general contexts.
As we mentioned earlier, the first issue is to define local Hardy spaces in a reasonable setting. Because of the above remarks, it is relatively natural to confine ourselves within the setting of Lie groups. Indeed, Triebel--Lizorkin spaces have been extended at various steps  to general Lie groups, starting with Lie groups endowed with a left-invariant Riemannian structure (see~\cite[Chapter 7]{TriebelFS2}), proceeding with Lie groups endowed with a left-invariant sub-Riemannian structure (see~\cite{BPV,BPV2,BPV3}) and, more recently, to filtered Lie groups (see~\cite{BCP,Calzi2,Calzi3,Calzi4,Calzi5}), which include both Lie groups endowed with a left-invariant sub-Riemannian structure and homogeneous groups.
In this latter case of filtered Lie groups, instead of requiring that the Lie algebra $\gf$ be endowed with a graduation, which would impose that $\gf$ be nilpotent, we simply require that $\gf$ be endowed with a suitable increasing filtration $(\gf_\lambda)_{\lambda>0}$, that is, an increasing sequence of vector subspaces of $\gf$ with $[\gf_\lambda,\gf_\mi]\subseteq \gf_{\lambda+\mi}$. One may then associate with $\gf$ a homogeneous Lie algebra $\gf_*$ -- its `contraction' -- in a natural way and, to every left-invariant differential operator $\Lc$ on $G$, a homogeneous left-invariant differential operator $P$ -- its `principal part' -- on the homogeneous group $G_*$ with Lie algebra $\gf_*$. The operator $\Lc$ is then said to be weighted subcoercive if $P+P^*$, where $P^*$ denotes the formal adjoint of $P$, is a positive Rockland operator, i.e., a homogeneous and hypoelliptic differential operator. This connection with the very well-behaved class of Rockland operators allowed ter Elst and Robinson to derive several impressive properties for weighted subcoercive operators (see~\cite{ElstRobinson}). For our purposes, a weighted subcoercive operator $\Lc$ generates a heat semigroup $(\ee^{-t\Lc})_{t>0}$ whose convolution kernel satisfies suitable Gaussian estimates. If, in addition, $\Lc$ is formally self-adjoint, then it is also essentially self-adjoint on $L^2$ with initial domain the space of test functions, hence defines a functional calculus. If $G$ has polynomial volume growth, this functional calculus is particularly interesting, since Schwartz functions of $\Lc$ turn out to have Schwartz convolution kernels with controllable bounds. Even though, in the case of a general Lie group, the properties of the heat semigroup do not seem to be sufficient to define a space $F^{p,2}_0$ for $p\meg 1$, when $G$ has polynomial growth it is actually possible to define the full scale of non-homogeneous Triebel--Lizorkin spaces (see~\cite{Calzi4,Calzi5}). Thus, it is natural to ask whether $F^{p,2}_0$ may be interpreted as a local Hardy space $h^p$ in this context, and whether $F^{\infty,2}_0$ may, by duality, be interpreted as a $bmo$ space. 

The existing work on non-homogeneous Triebel--Lizorkin spaces on Lie groups of polynomial growth therefore leads us to choose this setting for the development of a theory of local Hardy spaces. Therefore, the aim of the present paper is to introduce and study local Hardy spaces $h^p$ on filtered Lie groups of polynomial growth, investigate their basic properties, with particular emphasis on the atomic decomposition, establish the identification of $h^p$ with $F^{p,2}_0$, and characterize the dual of $h^1$ in terms of a suitable local $bmo$ space, implying its identification with $F^{\infty,2}_0$ as well.
As one may expect, we shall define the local Hardy space $h^p$ by means of a maximal function. Namely, we shall say that a tempered distribution $f$ belongs to $h^p$ if
\[
\widetilde\Mc^{(\lambda)}f\coloneqq \sup_{t\in (0,1]} \sup_{\phi\in \widetilde B_{\lambda,t}} \abs{f*\phi}
\]
belongs to $L^p$. Observe that, since we do not have dilations at our disposal, we define for each $t>0$ a set $\widetilde B_{\lambda,t}$ of smooth functions which satisfy suitable $t$-dependent estimates, instead of considering the $t$-dilates of the functions in $\widetilde B_{\lambda,1}$. In addition to that, for technical reasons, we shall actually require the elements of $\widetilde B_{\lambda,t}$ to be supported in a ball of radius $t$, instead of a sufficiently rapid decay. We shall then introduce a notion of atoms and establish the corresponding atomic decomposition for $h^p$. 
However, we shall be able to do that only if $p$ is sufficiently large or if $G$ is nilpotent. The reason, as hinted above, lies in the problem of implementing a reasonable vanishing moments condition in more general groups. If the group is nilpotent, then one may   transfer `by force' the vanishing moments condition on small balls to the Lie algebra by means of the exponential map, where it makes sense, and then essentially proceed with the usual proof. Nevertheless, one has to be careful since, by doing so, the vanishing moment condition depends on the center of the given ball. One may, in principle, do the same also for a general group. However, at the later stages of the proof one has to perform sums of `bad terms' with vanishing moments relative to different balls, hence to different centers: if the group is nilpotent, then translations happen to be polynomial and do not completely disrupt the vanishing moments, but if the group is not nilpotent, then the translations are no longer polynomial and there is little hope of reaching the conclusion.
Nonetheless, if $p$ is sufficiently large, then the only condition to be imposed is the zero average condition, which does not create any issues.

A remark on the construction of atoms is in order. In our setting, we believe that a more natural approach to defining atoms supported on small balls would be to replace the usual vanishing moment conditions with a requirement based on the possibility of `extracting derivatives', namely, of  decomposing an atom as a sum $\sum_j X_j a_j$, where the $X_j$ are non-constant left-invariant differential operators and the $a_j$ are supported in a comparable ball and satisfy appropriate size conditions. Indeed, while vanishing moments are an effective tool in the Euclidean setting, their formulation relies on the notion of polynomials, which does not have an equally natural counterpart on general groups. By contrast, a condition expressed in terms of derivatives builds on structures that are already intrinsic to the underlying group.
In fact, every atom satisfying the usual vanishing moment conditions can be decomposed in terms of derivatives, and the resulting derivative-based atomic Hardy spaces are compatible with the action of naturally defined Calder\'on--Zygmund operators. Moreover, atoms of this form also arise in the discretization theorems for Besov and Triebel--Lizorkin spaces, where they play a somewhat analogous, though conceptually different, role (see~\cite{Calzi4}). However, we have not been able to establish an atomic decomposition of the local Hardy spaces in terms of this weaker notion of atoms.

The paper is organized as follows. In Section \ref{Section:Preliminaries}, we collect some preliminaries on filtered Lie groups, weighted subcoercive operators and the definition of Triebel--Lizorkin spaces. Section \ref{Section:HardySpaces} focuses on the definition of local Hardy spaces via a maximal function and the study of some of their properties. In Section \ref{Section:AtomicDecomposition}, we first establish a Calder\'on--Zygmund decomposition for functions in $h^p$ and then prove our first main result, namely an atomic decomposition theorem. In Section \ref{Section:Identification}, we state and prove our second main theorem, establishing the equivalence between the local Hardy spaces $h^p$ and the Triebel--Lizorkin spaces $F_0^{p,2}$. Section \ref{Section:bmoSpaces} is devoted to the definition of the local $bmo$ space and the characterization of the dual of $h^1$ in terms of this space, and hence of $F^{\infty,2}_0$. Finally, Sections \ref{sec:app2} and \ref{sec:app} contain the proofs of Lemma \ref{lem:73} and Proposition \ref{prop:31}, respectively.
 while still retaining the needed properties.

\section{Preliminaries}\label{Section:Preliminaries}
 
In this section, we collect some preliminary results and introduce the notation used throughout the paper. For further details, we refer the reader to \cite{Calzi4}.

\subsection{Relatively Invariant Measures and Convolution}

Throughout the paper, we shall denote with $G$ a connected (finite-dimensional, real) Lie group  with Lie algebra $\gf$. We shall denote with $\beta$ a left Haar measure and we shall assume that $\beta$ has polynomial volume growth, that is, there is $Q_G \Meg 0$ such that, for every compact neighbourhood $V$ of the identity $e$ in $G$, one has
\[
\beta(V^n)\asymp n^{Q_G}
\]
for $n\to \infty$. Then, $\beta$ is also a right Haar measure.

We now recall some facts about convolution. Given two convolvable\footnote{We shall not provide a precise definition of this concept, since this would drag us too far away from the main topic. } distributions $f,g$ on $G$, one has
\[
\langle f*g, \phi\rangle= \langle f\otimes g, (x,y)\mapsto \phi(xy)\rangle
\]
for every $\phi\in C^\infty_c(G)$.
We shall identify each $f\in L^1_\loc(G)$ with $f\cdot \beta$, that is, the measure with density $f$ with respect to $\beta$. Given two convolvable functions $f,g$ such that $f*g$ is absolutely continuous with respect to $\beta$, we shall generally identify $f*g$ with its density with respect to $\beta$. Thus, under very mild conditions which will be always verified in the applications,
\begin{equation}\label{eq:1b}
	(f* g)(x)=\int_G f(x y^{-1}) g(y)  \,\dd \beta(y)= \int_G f(y) g(y^{-1}x) \,\dd \beta(y).
\end{equation} 
Similar formulae apply when either $f$ or $g$ is a distribution, and the convolution is still a function.
Observe that, if $f,g$ are convolvable distributions, $X$ is a left-invariant differential operator, and $Y$ is a right-invariant differential operator, then $Yf $ and $X g$ are convolvable and
\[
YX(f*g)=(Yf)*(Xg).
\]
If $f,g,h$ are distributions then, under some reasonable conditions that will always be satisfied in the applications, 
\[
\langle f * g, h \rangle =\langle f, h*  \check g \rangle =\langle g,   \check f * h\rangle,
\]
where $\langle \check f,\phi\rangle=\langle f,\check \phi \rangle$ for every $\phi\in C^\infty_c(G)$, and $\check \phi=\phi(\,\cdot\,^{-1})$.  

We now recall Young's inequality in this context. Take $p_1,p_2,p_3\in [1,\infty]$ so that $\frac{1}{p_1'}+\frac{1}{p_2'}=\frac{1}{p_3'}$. Then,
\[
\norm{f*g}_{L^{p_3}(G)}\meg \norm{f}_{L^{p_1}(G)}\norm{g}_{L^{p_2}(G)}
\] 
for every two positive $\beta$-measurable functions $f,g$, whose convolution may be defined by means of~\eqref{eq:1b}.  

In order to use clearer notation, even if somewhat more cumbersome, given a measure $\mi$ on a measurable space $X$, $p\in (0,\infty]$, and a $\mi$-measurable function $f$ on $X$, we shall also write 
$\norm{f(x)}_{L^p_x(\mi)}$ instead of $\norm{f}_{L^p(\mi)}$.
This will be particularly useful when dealing with nested (quasi-)norms.

\subsection{Differential Operators}

	We shall generally identify the complexification of the enveloping algebra $U(\gf)$ of $\gf$ with the algebra of left-invariant differential operators. We shall fix a scalar product on $\gf$, and endow $U(\gf)$ with the corresponding scalar product, namely the quotient of the natural scalar product on the complexification of the tensor algebra over $\gf$.\footnote{The actual scalar product on $U(\gf)$ will not matter in the sequel.}  We endow $\gf$ with a scalar product and $U(\gf)$ with its canonical extension (that is, the scalar product induced by the canonical scalar product on the complexified tensor algebra over $\gf$). The chosen scalar product will in fact not matter much in the sequel, but this choice has some minor advantages.

	Let $X$ be a left-invariant differential operator. We  denote with $X^R$ the right-invariant differential operator which induces the same point distribution as $X$ at $e$. In other words, $(X f)(e)=(X^R f)(e)$ for every $f\in C^\infty (G)$. We denote with $X^\dag$ the transpose of $X$ in the enveloping algebra $U(\gf)$. In other words, the mapping $X\mapsto X^\dag$ is the unique anti-automorphism of $U(\gf)$ (that is, such that $(XY)^\dag=Y^\dag X^\dag$) which extends the automorphism $X\mapsto -X$ of $\gf$. 
	Then, $X^\dag$ is the formal transpose of $X$ (with respect to $\beta$), that is, the unique left-invariant differential operator such that
	\[
	\int_G (X f) g\,\dd \beta=\int_G f X^\dag g\,\dd \beta
	\]
	for every $f,g\in C^\infty_c(G)$. We denote with $X^*$ the formal adjoint of $X$, that is, $\overline X^\dag$ and define the formal transpose and the formal adjoint of right-invariant differential operators in a similar way. If $u$ is a distribution, we then define $X u$ so that
	\[
	\langle X u,\phi\rangle =\langle u, X^\dag \phi\rangle
	\]
	for every $\phi\in C^\infty_c(G)$. In this way, if $f\in C^\infty(G)$, then $(X f)\cdot \beta=X(f\cdot \beta)$.

	In order to simplify the notation, we write $X^{R\dag}$ instead of $(X^R)^\dag$, etc.

See~\cite[Proposition 2.2]{BCP} for a proof of the following result, which will be used without citation in the sequel.
\begin{prop}
	The following hold:
	\begin{enumerate}
		\item[\textnormal{(1)}]   $X^{R\dag}  =  X^{\dag R} $ for every $X\in U(\gf)$;
		
		\item[\textnormal{(2)}] $X\delta_e= X^R\delta_e=  (X^\dag \delta_e)\check{\;}$   for every $X\in U(\gf)$;
		
		\item[\textnormal{(3)}] $X^\dag f=(X^R \check f) \check{\,}$ for every $X\in U(\gf)$ and for every $f\in C^\infty(G)$.
	\end{enumerate}
\end{prop}

Notice that, by (2), 
\[
\begin{split}
	(Xf)*g=(f*X\delta_e)*g=f*(X\delta_e*g)=f*(X^{R} \delta_e*g)=  f*(X^R g)
\end{split}
\]
under some reasonable conditions on $f$ and $g$ that are needed to grant associativity of convolution.

In addition, observe that, by our choice of the scalar product on $U(\gf)$, one has $\abs{X}=\abs{X^\dag}=\abs{\overline X}$, hence also $\abs{X}=\abs{X^*}$ for every $X\in U(\gf)$, where by $|X|$ we denote the norm of an element $X\in U(\gf)$.

\subsection{Filtrations}

Throughout the paper, $(\gf_\lambda)_{\lambda\Meg 0}$ will denote a filtration of $\gf$ such that
\begin{enumerate}
    \item it is increasing, i.e. so that $[\gf_\lambda, \gf_\mi]\subseteq \gf_{\lambda+\mi}$ for every $\lambda, \mi\Meg 0$,
    \item $\gf_\lambda=\Set{0}$ for every $\lambda<1$,
    \item $\bigcup_{\lambda\Meg 0} \gf_\lambda=\gf$,
    \item $\bigcap_{\mi>\lambda} \gf_\mi=\gf_\lambda$ for every $\lambda\Meg 0$.
\end{enumerate}
Moreover, we assume the filtration to be \textit{admissible}, namely there exists $\dd_0\ge 1$ so that, if $\lambda$ is such that $\gf_\lambda\ne\bigcup_{\mi<\lambda} \gf_\mi$, then $\lambda\in\dd_0\Q$. In other words, the set $\Lambda\coloneqq \{\lambda\Meg 1\colon \gf_\lambda \neq \bigcup_{\mi<\lambda} \gf_\mi\}$ generates a $\Q$-vector space of dimension $1$ and $\Lambda \subseteq \Q \dd_0$. The reason behind this requirement will be explained in Remark \ref{rem:1}.

For every $X\in \gf$, we define  $\deg X\coloneqq \min\Set{\lambda\Meg0\colon X\in \gf_\lambda}$ and we call it the degree of $X$.
Now set, for every $\lambda>0$,
\[
\gf_{\lambda^-}\coloneqq \bigcup_{\mi<\lambda} \gf_\mi, \quad \gf_{*,\lambda}\coloneqq\gf_\lambda/\gf_{\lambda^-}, \quad 
\gf_*\coloneqq \bigoplus_{\lambda>0} \gf_{*,\lambda}.
\]
One can define a Lie algebra structure on $\gf_*$ as follows: if $X=\sum_{\lambda>0} (X_\lambda+ \gf_{\lambda^-})$ and $Y= \sum_{\lambda>0} (Y_\lambda+ \gf_{\lambda^-})$ for some $(X_\lambda),(Y_\lambda)\in \prod_{\lambda>0} \gf_\lambda$, then
\[
[X,Y]\coloneqq \sum_{\lambda,\mi>0}\left(  [X_{\lambda},Y_{\mi}]+\gf_{(\lambda+\mi)^-}\right) .
\]
It is easily seen that $(\gf_{*,\lambda})_{\lambda>0}$ is a graduation of type $((0,+\infty),+)$ of $\gf_*$, i.e. such that 
\[
[\gf_{*,\lambda}, \gf_{*,\mi}]\subseteq \gf_{*,\lambda+\mi} \text{ for every }\lambda,\mi>0.
\]
One may then endow $\gf_*$ with the automorphic dilations $(\delta_r)_{r>0}$ defined so that $\delta_r(X)=r^\lambda X$ for every $X\in \gf_{*,\lambda}$ and for every $r,\lambda>0$. Thus, $\gf_*$ is the Lie algebra of some homogeneous group $G_*$, with homogeneous dimension $Q_*\coloneqq \sum_{\lambda>0} \dim \gf_{*,\lambda}$.

\begin{oss}
Note that we consider the full range $\lambda\Meg 0$  for notational convenience, in analogy with the filtration $(U_\lambda)$ that we introduce below, for which $U_\lambda\neq \Set{0}$ for every $\lambda\Meg 0$.
Nevertheless, we required $\gf_\lambda=\Set{0}$ for $\lambda<1$ in order for the control modulus $\abs{\,\cdot\,}_*$ (see~Definition~\ref{def:3} below) to induce a left-invariant \emph{distance} on $G$ rather than a quasi-distance. There may be also good reasons to require $\gf_1\neq \Set{0}$. We preferred to avoid imposing this condition in order to keep a natural comparison with graded (or, more generally, homogeneous) groups: if $\gf$ has a graduation $(\tilde \gf_j)_{j\in \Z_+^*}$ (with integer degrees), then it is natural to set
\[
\gf_\lambda=\bigoplus_{j=1}^{[\lambda]} \tilde \gf_j \text{ for every }\lambda \Meg 0, 
\]
but there is no guarantee that $\tilde \gf_1$ should be non-trivial.
\end{oss}

We now extend this filtration to the  enveloping algebra $U(\gf)$.	
For every $\lambda\Meg 0$, let
\[
U_\lambda=\Bigl\{X_1\cdots X_k\colon k\Meg 0, X_1,\dots, X_k\in \gf, \deg X_1+\cdots+ \deg X_k\meg \lambda\Bigr\}.
\]
Thus, the identity operator, corresponding to the case $k=0$, belongs to all $U_\lambda$. 
Then, $(U_\lambda)$ is an increasing filtration of $U(\gf)$, that is, $U_\lambda, U_\mi\subseteq U_\lambda U_\mi\subseteq U_{\lambda+\mi}$ for every $\lambda, \mi\Meg 0$.
For every $X\in U(\gf)$, we define the degree of $X$ as
$\deg X\coloneqq \min\Set{\lambda\Meg 0\colon X\in U_\lambda}$.
Notice that $\gf\cap U_\lambda=\gf_\lambda$ for every $\lambda\Meg 0$ by means of Proposition~\ref{prop:8} below, so that this definition is consistent with the previous one. In analogy with what we did before, we set
\[
U_{\lambda^-}\coloneqq \bigcup_{\mi<\lambda} U_\mi \text{ for }\lambda>0,\quad U_{0^-}\coloneqq\Set{0}, \quad U_{*,\lambda}\coloneqq U_\lambda/U_{\lambda^-} \text{ for every }\lambda\Meg 0
\]
and 
\[
U_*\coloneqq \bigoplus_{\lambda\Meg 0} U_{*,\lambda} .
\] 
We define an algebra structure on $U_*$ as follows:  if $X=\sum_{\lambda\Meg 0} (X_\lambda+ U_{\lambda^-})$ and $Y= \sum_{\lambda\Meg 0} (Y_\lambda+ U_{\lambda^-})$ for some $(X_\lambda),(Y_\lambda)\in \prod_{\lambda\Meg 0} U_\lambda$, then
\[
XY\coloneqq \sum_{\lambda,\mi\Meg 0}\left(  X_{\lambda}Y_{\mi}+U_{(\lambda+\mi)^-}\right) .
\]
It is easily seen that $U_*$ becomes a graded algebra of type $([0,+\infty),+)$ with this structure.

See~\cite[Proposition 4.2]{BCP} for a proof of the following result.

\begin{prop}\label{prop:8}
	The canonical inclusions $\gf_\lambda \subseteq U_\lambda$, $\lambda> 0$, induce a linear mapping $\pi \colon \gf_*\to U_*$. The canonical extension $U(\pi)\colon U(\gf_*)\to U_*$ of $\pi$ is an  isomorphism of graded algebras.
\end{prop}

From now on, we shall identify $U_*$ and $U(\gf_*)$ by means of $U(\pi)$.  

\subsection{Weighted Subcoercive Operators, Control Modulus, and the Schwartz Space}

\begin{deff}	
	We say that a family $(X_j)_{j\in J}$ of elements of $\gf$ is a minimal basis if, setting $Y_j\coloneqq X_j+ \gf_{(\deg X_j)^-}$, the family $(Y_j)$ induces a homogeneous basis of $\gf_*/[\gf_*,\gf_*]$.
	We denote with $\dd$ the least common multiple of the degrees of the $X_j$, $j\in J$.
\end{deff}

Observe that $\dd$ does \emph{not} depend on the choice of $(X_j)$, since it is the least common multiple of the degrees of the non-zero elements of $\gf_*/[\gf_*,\gf_*]$.
Moreover, if $(X_j)$ is a minimal basis of $\gf$, then $(X_j)$ generates $\gf$ as a Lie algebra, and also generates the filtrations $(\gf_\lambda)$ and $U_\lambda$ (see~\cite[Proposition 4.4]{BCP}). In other words, $\gf_\lambda$  is the vector space generated by the vector fields of the form $\ad(X_{j_1})\cdots \ad(X_{j_{k-1}})X_{j_k}$,
where $k\Meg 1$, $j_1,\dots, j_k\in J$, and $\deg(X_{j_1})+\cdots +\deg(X_{j_k})\meg \lambda$. Analogously, $U_\lambda$ is the vector space generated by the differential operators for the form $X_{j_1}\cdots X_{j_k}$, where $k\Meg 0$, $j_1,\dots, j_k\in J$, and $\deg(X_{j_1})+\cdots +\deg(X_{j_k})\meg \lambda$.
In particular, $(X_j)$ is a `reduced weighted algebraic basis' of $\gf$, in the terminology of~\cite{ElstRobinson}.

\begin{deff}
    An operator $\Lc \in U(\gf)$ is weighted subcoercive if $ \Lc+\Lc^*+U_{\deg(\Lc)^-}$ is a positive Rockland operator on $G_*$.
\end{deff}
As shown in~\cite[Theorem 4.7]{BCP}, the above   definition of weighted subcoercive operators is consistent with the one given in~\cite{ElstRobinson}. In particular, if $\Lc$ is weighted subcoercive, then $\deg(\Lc)/\dd$ is an even integer. 

\begin{oss}\label{rem:1}
    Notice that $\dd$ is well defined since we assumed that the set $\Lambda$ of `non-trivial degrees'  is contained in a one-dimensional $\Q$-vector space. For the same reason, weighted subcoercive operators actually exist on $G$ (see~Proposition~\ref{prop:4} below). In fact, without this assumption, there would be no weighted subcoercive operators on $G$.
\end{oss}

We recall the following example from~\cite[Proposition 4.9]{BCP} for the construction of weighted subcoercive operators.
\begin{prop}\label{prop:4}
	Let $(X_j)_{j\in J}$ be a minimal basis of $\gf$. Then, $\Lc\coloneqq \sum_{j\in J} (X_j^{\dd/d_j})^\dag X^{\dd/d_j}_j$, where $d_j=\deg X_j$ for every $j\in J$, is a real,  positive and formally self-adjoint  weighted subcoercive operator.
\end{prop}
Here, by `positive' we mean that $\int \Lc f \overline f \,\dd \beta\Meg 0$ for every $f\in C^\infty_c(G)$.
In particular, if $(\gf_\lambda)$ is the filtration generated by $\gf_1$, then $\Lc$ is a sub-Laplacian (a Laplacian, if $\gf_1=\gf$).

\begin{nott}
	From now on, we shall fix a \emph{formally self-adjoint} weighted subcoercive operator $\Lc$ with degree $\grado$.  
\end{nott}

\vskip 0.3 cm
We shall now define a control modulus on $G$. 
\begin{deff}\label{def:3}
	Given an absolutely continuous curve $\gamma\colon [0,1]\to G$, we  define the content of $\gamma$ as 
    \[
    C(\gamma)=\inf\Set{\eps>0\colon \abs{P_\lambda \dd L_{\gamma(t)}^{-1}\gamma'(t)} \meg \min(\eps, \eps^{\lambda}) \text{ for a.e. } t\in [0,1] \text{ and for all } \lambda>0},
    \]
    where $P_\lambda$ is the orthogonal projector of $\gf$ onto $\gf_\lambda \ominus \gf_{\lambda^-}$ (this is non-zero only for finitely many $\lambda>0$) and $L_{\gamma(t)}$ is the left translation by $\gamma(t)$. 
	
	Given $x\in G$, we define
    \[
    \abs{x}_*\coloneqq\inf\Set{C(\gamma)\colon \gamma\colon [0,1]\to G \text{ absolutely continuous, } \gamma(0)=e\text{ and }\gamma(1)=x}.
    \]
\end{deff} 
It turns out that $\abs{\,\cdot\,}_*$ is a control modulus and we  endow $G$ with the induced left-invariant distance $d(x,y)\coloneqq \abs{y^{-1}x}_*$.

Choosing a different scalar product on $\gf$   gives rise to bi-Lipschitz equivalent control moduli. Equivalence near $e$ follows from~\cite[Corollary 6.5]{ElstRobinson}  whereas equivalence at infinity follows from the fact that all these control moduli are `connected moduli': in fact, every $x\in G$ may be written as $x_1\cdots x_k$, where $\abs{x_j}_*\meg 1$ for $j=1,\dots, k$, and $k\meg \abs{x}_*+1$.
Notice that here we are not requiring $\gamma$ to be `horizontal'. One may also require $\gamma$ to be horizontal with respect to some fixed weighted algebraic basis compatible with the filtration $(\gf_\lambda)$, in the terminology of~\cite{ElstRobinson}, and still get an equivalent control modulus. This would provide a better mean value theorem for the corresponding `horizontal gradient,' but we shall not need this kind of more precise estimates.

\vskip 0.3 cm
It is a well-known fact that in this context
\[
\beta(B(e,r))\asymp r^{Q_*} \qquad \text{for }r\to 0^+,
\]
while
\[
\beta(B(e,r))\asymp r^{Q_G}  \qquad \text{for }r\to+\infty,
\]
so that it makes sense to define the Schwartz space as follows.
\begin{deff}
	Let $\Banach$ be a \BanachSpace. We define $\Sr(G;\Banach)$ as the space of $f\in C^\infty(G;\Banach)$ such that the seminorms
    \[
    \norm{ (1+ \abs{\,\cdot\,}_*)^k X^R f}_{L^1(G;\Banach)}, \, k\in \N, \, X\in U(\gf),
    \]
    are finite, endowed with the corresponding topology. We will simply write $\Sr(G)$ whenever $\Banach=\C$.
	
	We denote with $\Sr'(G;\Banach)$ the space of continuous linear mappings from $\Sr(G)$ into $\Banach$, endowed with the topology of uniform convergence on the bounded subsets of $\Sr(G)$.
\end{deff}

Notice that, as in~\cite{Calzi4,Calzi5}, we denote the Schwartz space as $\Sr(G)$ instead of $\Sc(G)$ since we wish to avoid any confusion with the completely different `Schwartz space' $\Sc(G)$ which was used in~\cite{BCP, Calzi2, Calzi3}, where the polynomial decay was replaced with an exponential one in order to deal with the possibly exponential growth of the volume.

See~\cite[Theorem 2.15 and Corollary 2.16]{Calzi4} for more information on $\Sr(G)$ and $\Sr'(G)$, and~\cite{Schweitzer} for a much more detailed exposition of general Schwartz spaces in various settings.

\subsection{Estimates}

We shall now consider some results on the functional calculus associated with $\Lc$.
Observe first that, by\cite[Theorem 2.9]{Calzi4}, the operator $\Lc$, with initial domain $C^\infty_c(G)$, is essentially self-adjoint on $L^2(G)$, so that we may consider the associated functional calculus. In particular,  for every Borel function $m$ on $\R$ with polynomial growth, that is, such that $(1+\abs{\,\cdot\,})^{-k} m$ is bounded for some $k\in\N$, there is a unique  distribution $\Kc_\Lc(m)$ on $G$ such that
\[
m(\Lc)\phi=\phi*\Kc_\Lc(m)
\]
for every $\phi\in C^\infty_c(G)$. In fact, $\Kc_\Lc(m)$ is necessarily a tempered distribution, so that the previous equality may be extended to every $\phi$ in the Schwartz space.

\begin{deff}
	We define $D$ as the maximum of $Q_*$ and $Q_G$.
\end{deff}
Note that that there exists c > 1 such that
\[
\frac1cr^{Q_*}(1+r)^{Q_G-Q_*}\meg \beta(B(e,r))\meg cr^{Q_*}(1+r)^{Q_G-Q_*}.
\]
for all $r>0$. It then follows that
\[
\beta(B(e,tr))\meg c^2(1+t)^D \beta(B(e,r))
\]
for all $t, r > 0$, which implies that $\beta$ is doubling.

See~\cite[Lemmas 2.24  and 2.25]{Calzi4} for a proof of the following simple results.

\begin{lem}\label{lem:37b}
	Take $ \rho>0$. Then, there is a constant $C>0$ such that
	\[
	\norm{\chi_{B(e,\rho)\setminus B(e,t)}\abs{\,\cdot\,}_*^{-Q_*/p-\alpha}  }_{L^p(G)}\meg \left( \frac{C}{1-2^{-p\alpha}}\right) ^{1/p} t^{-\alpha}
	\]
	while
	\[
	\norm{\chi_{G\setminus B(e,s)}\abs{\,\cdot\,}_*^{-Q_G/p-\alpha}  }_{L^p(G)}\meg \left( \frac{C}{1-2^{-p\alpha}}\right) ^{1/p} s^{-\alpha}
	\]
	for every $t\in (0,\rho]$, $s\in [\rho,+\infty)$, $\alpha>0$, and $p\in (0,\infty]$.
\end{lem}

\begin{lem}\label{lem:37}
	There is a constant $C>0$ such that
	\[
	\norm{(1+\abs{\,\cdot\,}_*/t)^{-\alpha-D/p}  }_{L^p(G)}\meg 3 \left( \frac{C}{1-2^{-p\alpha}}\right) ^{1/p} t^{Q_*/p}\max(1,t)^{(Q_G-Q*)/p}
	\]
	for every $t,\alpha>0$ and $p\in (0,\infty]$.
\end{lem}

See~\cite[Corollary 2.27]{Calzi4} for a proof of the following result.

\begin{prop}\label{cor:12} 
	Take $\lambda,c\Meg 0$, $\rho>0$ and $s>c+D(1/p-1/2)_+$. There is a constant $C>0$ such that for every $p\in (0,\infty]$, $m\in B^{\infty,\infty}_s(\R)$ supported in $[-\rho,\rho]$, $r\in (0,\rho]$, and $X,Y\in U_\lambda$,
	\[
	\norm{XY^R\Kc_{r^{\grado}\Lc } (m ) (1+\abs{\,\cdot\,}_*/r)^c}_{L^p(G)} \meg C r^{-\deg(X)-\deg(Y)+(1/p-1)Q_*}\abs{X}\abs{Y}\norm{m}_{B^{\infty,\infty}_s(\R)}.
	\]
\end{prop}

We refer to ~\cite[Lemma 2.35 and Corollary 2.36]{Calzi4} to better understand how one may  use the estimates provided in Proposition~\ref{cor:12} to deal with convolution kernels associated with Schwartz multipliers, and also how one may control the convolution of convolution kernels associated with the functional calculi of two distinct formally self-adjoint weighted subcoercive operators.

\subsection{Triebel--Lizorkin Spaces}\label{subsec:TLspaces}

	Consider the set $\Psi_{2^{-\grado}}$ of $(\psi_j)\in \Sr(G)^\N$ such that there is some bounded family $(\phi_j)\in C^\infty_c(\R)^\N$ such that 
	\[
	0\not \in \overline{\bigcup_{j\Meg 1}\supp \phi_j }, \qquad 
	\inf_{\lambda\in \R} \sum_{j\in\N} \abs{\phi_j(2^{-j\grado}\lambda)}>0,
	\]
	and
	\[
	\psi_j= \Kc_{2^{-j\grado}\Lc}(\phi_j )
	\]
	for every $j\in\N$.

\begin{deff}
	Take $p,q\in (0,\infty]$, with $p<\infty$, a  \BanachSpace\ $\Banach$, $\alpha\in \R$, and $(\psi_{j})\in \Psi_{2^{-\grado}}$. We define  $F^{p,q}_\alpha(G;\Banach)$ as the space of $f\in \Sr'(G;\Banach)$ such that
	\[
	\norm*{ \norm{2^{j\grado\alpha/d} (f*\psi_{j})(x)}_{\ell^q_j(\N;\Banach)}}_{L^p_x(G)}<\infty,
	\]
	endowed with the corresponding quasi-norm.
    If $q>1$, then we define $F^{\infty,q}_\alpha(G;\Banach)$ as the elements of $\Sr'(G;\Banach)$ which are in the dual of $F^{1,q'}_{-\alpha}(G;Z')$, canonically embedded in $\Sr'(G;Z')$ by means of the transpose of the dense inclusion $\Sr(G;\Banach)\subseteq F^{1,q'}_{-\alpha}(G;\Banach)$. This is consistent with~\cite[Definition 3.11]{Calzi4} thanks to~\cite[Theorem 5.2]{Calzi4}.
\end{deff}

\begin{oss}
    It was proved in~\cite[Corollary 3.10]{Calzi4} that we can replace $2^{-\grado}$ by any $\eps\in(0,1)$, yielding an equivalent definition of Triebel--Lizorkin spaces. Moreover, the resulting definition is independent of the choice of the operator involved.
\end{oss}

\section{Local Hardy Spaces: definition and properties}\label{Section:HardySpaces}
In this section, we introduce local Hardy spaces on groups of polynomial growth via a maximal function. We then show that these spaces admit several equivalent characterizations in terms of different maximal functions.

	\begin{deff}
		Take $\lambda,a\Meg 0$ and $t>0$. Define $B_{\lambda,a,t}$ as the set of $\phi\in \Sr(G;\Banach)$ such that
		\[
		\abs{(X^R \phi)(x)}\meg\abs{X} t^{-Q_*-\deg(X)}(1+\abs{x}_*/t)^{-a}
		\]
		for every $X\in U_\lambda$  and for every $x\in G$.   Then, we  define
		\[
		\Mc^{(\lambda,a)} f\coloneqq \sup_{t\in (0,1]} \sup_{\phi\in B_{\lambda,a,t}}\abs{f*\phi}
		\]
		for every $f\in \Sr'(G;\Banach)$.
	
		In addition, for every $\lambda\Meg 0$ and $t>0$ we define $\widetilde B_{\lambda,t}$ as the set of $\phi\in \Sr(G;\Banach)$ supported in $\overline B(e,t)$ and such that
		\[
		\norm{X^R \phi}_{L^\infty(G)}\meg \abs{X}t^{-Q_*-\deg(X)}
		\]
		for every $X\in U_\lambda$. Then, for every $\kappa>0$ we define
		\[
		\widetilde\Mc^{(\lambda,\kappa)} f\coloneqq \sup_{t\in (0,\kappa]} \sup_{\phi\in B'_{\lambda,t}}\abs{f*\phi}
		\]
		for every $f\in \Sr'(G;\Banach)$.
	\end{deff}
	
	Observe that $\Mc^{(\lambda,a)}f$ and $\widetilde\Mc^{(\lambda,\kappa)} f$ are lower semi-continuous, since they are pointwise least upper bounds of the continuous functions
    \[
    \abs{f*\phi}, \text{ for }\phi\in \bigcup_{t\in (0,1]} B_{\lambda,a,t}\text{ and }\phi \in \bigcup_{t\in (0,\kappa]} \widetilde B_{\lambda,t},
    \]
    respectively.
	In addition, it is clear that, if $\phi\in \widetilde B_{\lambda,t}$, then $(1+\max(1,t))^{-a}\phi\in B_{\lambda,a,\min(1,t)}$, so that
	\[
	\widetilde \Mc^{(\lambda,\kappa)} f \meg (1+\max(1,\kappa))^a  \Mc^{(\lambda,a)}f
	\]
	for every $f\in \Sr'(G;\Banach)$ and every $a>0$.
	
	The next results show that, as long as one is only interested in the $L^p$ norm of the maximal functions, the parameter $\kappa$ in $\widetilde \Mc^{(\lambda, \kappa)}$ is irrelevant, and the maximal function $\Mc^{(\lambda, a)}$ is equivalent to a simpler one if $\lambda $ and $a$ are sufficiently large.
	We shall therefore only consider  $\widetilde \Mc^{(\lambda, 1)}$ in several results; it will nonetheless be convenient to have $\widetilde \Mc^{(\lambda, \kappa)}$ at our disposal in some situations.
	
	\begin{lem}\label{lem:85}
		Take $p\in (0,\infty]$, $\lambda \Meg 0$, and $\kappa, \kappa'>0$. There is a constant $C>0$ such that
		\[
		\norm{\widetilde \Mc^{(\lambda, \kappa)}f}_{L^p(G)}\meg C\norm{\widetilde \Mc^{(\lambda, \kappa')}f}_{L^p(G)}
		\]
		for every \BanachSpace\ $\Banach$ and for every $f\in \Sr'(G;\Banach)$.
	\end{lem}
	Observe that, if $\kappa\meg \kappa'$, then it is clear that $\widetilde \Mc^{(\lambda, \kappa)}f\meg \widetilde \Mc^{(\lambda, \kappa')}f$ for every $f$ and $\lambda$.
	
	\begin{proof}
		We only need to prove the assertion for $\kappa> \kappa'$. Take $\tau\in C^\infty_c(G)$ so that $\chi_{B(e,\kappa'/2)}\meg \tau \meg \chi_{B(e,\kappa')}$, and a maximal family $(x_j)_{j\in J}$ of elements of $G$ such that $B(e,\kappa'/4)x_j\cap B(e,\kappa'/4) x_{j'}=\emptyset$ for every $j,j'\in J$ with $j\neq j'$. Then,   there is $N\in\N$ such that
\[
1\meg \sum_{j\in J} \chi_{B(e,\kappa')x_j}\meg N
\]
on $G$ (see~\cite[Lemma 4.3]{Calzi2}, applied to the opposite group of $G$). Let $J'$ be the (finite) set of $j\in J$ such that $B(e,\kappa)\cap B(e,\kappa')x_j\neq \emptyset$. In addition, define $\widetilde \tau \coloneqq \sum_{j\in J} \tau(\,\cdot\, x_j^{-1})$, so that $\widetilde \tau \in C^\infty(G)$ and $1\meg \widetilde \tau \meg N$, and define $\tau_j\coloneqq \tau(\,\cdot\,x_j^{-1})/\widetilde \tau$ for every $j\in J$, so that
\[
\chi_{B(e,\kappa)} \meg \sum_{j\in J'} \tau_j \meg 1.
\]
		
		Now, take $t\in (0,\kappa]$ and $\phi\in \widetilde B_{\lambda,t}$. If $t\meg \kappa'$, then clearly
		\[
		\abs{f*\phi}\meg \widetilde \Mc^{(\lambda, \kappa')} f
		\]
		for every \BanachSpace\ $\Banach$ and for every $f\in \Sr'(G;\Banach)$. If, otherwise, $t>\kappa'$, note that 
		\[
		\phi=\sum_{j\in J'} \tau_j \phi
		\]
		and
		\[
		(f* \phi)(x)= \sum_{j\in J'} \langle f, \phi(\,\cdot\,^{-1}x) \tau_j(\,\cdot\,^{-1}x)\rangle =\sum_{j\in J'} \bigg(f*\frac{\phi(\,\cdot\,x_j)\tau}{\widetilde \tau(\,\cdot\, x_j)}\bigg)(x x_j^{-1})
		\]
		for every \BanachSpace\ $\Banach$, $x\in G$, and $f\in \Sr'(G;\Banach)$. Now, observe that $\phi(\,\cdot\,x_j)\tau/\widetilde \tau(\,\cdot\, x_j)$ is of class $C^\infty$ on $G$ and supported in $\overline B(e,\kappa')$ by construction. In addition, it is clear that there is a constant $C_1>0$ such that 
		\[
		\norm*{X \bigg(\frac{\phi(\,\cdot\,x_j)\tau}{\widetilde \tau(\,\cdot\, x_j)}\bigg)}_{L^\infty(G)}\meg C_1 \abs{X} \kappa'^{-Q_*-\deg(X)}
		\]
		for every $X\in U_\lambda$. Consequently,
		\[
		\abs{(f*\phi)(x)}\meg C_1 \sum_{j\in J'} (\widetilde \Mc^{(\lambda, \kappa')} f)(x x_j^{-1})
		\]
		for every \BanachSpace\ $\Banach$, $x\in G$, and $f\in \Sr'(G;\Banach)$. The assertion follows by the translation-invariance of $\beta$.		
	\end{proof}

    \begin{deff}
	For every $p\in (0,+\infty]$, for every $r>0$, and for every $\beta$-measurable function $f$ from $G$ into $\Banach$, we define
	\[
	(\Nc_{p,a,r} f)(x)\coloneqq r^{-Q_*/p}\norm{ f(y) (1+d(x,y)/r)^{-a}  }_{L^p_y(G;\Banach)}
	\]
	for every $x\in G$.
\end{deff}
	
	\begin{lem}\label{lem:71}
		Take $p\in (0,\infty]$, $a>D/p$, $k\in \N$  with $k\grado>a$, and $\phi\in C^\infty_c(\R)$ so that $\chi_{[-1,1]}\meg \phi\meg \chi_{[-2,2]}$. Set $\psi_j\coloneqq \Kc_{2^{-j\grado} \Lc}(\phi)$ for every $j\in\N$. There is a constant $C>0$ such that
		\[
		2^{-a}\widetilde\Mc^{(k\grado,1)}f\meg \Mc^{(k\grado,a)}f\meg C \sup_{j\in \N} \Nc_{\infty,a,2^{-j}}(f*\psi_j)
		\]
		and
		\[
		\Bigl\lVert\sup_{j\in \N} \Nc_{\infty,a,2^{-j}}(f*\psi_j)\Bigr\rVert_{L^p(G)}\meg C\Bigl\lVert\sup_{j\in \N} \abs{f*\psi_j} \Bigr\rVert_{L^p(G)}
		\]
		for every \BanachSpace\ $\Banach$, $x\in G$, and $f\in \Sr'(G;\Banach)$. If, in addition, $\Lc=\sum_{j\in J} (-1)^{\dd/d_j} X_j^{2\dd/d_j}$ for some minimal basis $(X_j)_{j\in J}$, with $d_j\coloneqq \deg(X_j)$ (see~Proposition~\ref{prop:4}), then one may replace $\grado$ with $\dd$ in the above formulae.
	\end{lem}

	\begin{proof}
		The first inequality in the first assertion of the statement is clear, so that we only have to show the second one. Take $t\in (0,1]$ and $\phi\in B_{k\grado,a,t}$. Take $j_0\in \N$ so that $t\in (2^{-(j_0+1)},2^{-j_0}]$ and observe that, setting $\psi_j'\coloneqq \psi_{j+1}-\psi_{j}$ for every $j\Meg 1$, one has 
		\[
		\begin{split}
			f*\phi&=f*\psi_{j_0}*\phi+\sum_{j>j_0} f*\psi'_j*\phi\\
				&=(f*\psi_{j_0+1})*(\psi_{j_0}*\phi)+\sum_{j>j_0} (f*\psi_{j+2})*(\psi'_j*\phi)
		\end{split}
		\] 
		in $\Sr'(G;\Banach)$ and in $C^\infty(G;\Banach)$, thanks to~\cite[Lemma 3.7]{Calzi4}.
		Next, observe that, using~\cite[Lemma 2.35]{Calzi4} (with $\lambda=\mi=0$), \cite[Lemma 2.34]{Calzi4} and~\cite[Corollary 2.36]{Calzi4} (with $\lambda=k\grado-a$)  we see that there is a constant $C_1>0$ such that
		\[
		\norm{(\psi_{j_0}*\phi) (1+2^{j_0}\abs{\,\cdot\,}_*)^a}_{L^1(G)}\meg C_1\meg C_1 (2^{j_0}t)^{a-k\grado}
		\]
		and such that 
		\[
		\norm{(\psi'_{j}*\phi) (1+2^{j}\abs{\,\cdot\,}_*)^a}_{L^1(G)}\meg C_1(2^{j}t)^{a-k\grado} 
		\]
		for every $j>j_0$.		
		Consequently,
		\[
		\abs{f*\phi}\meg C_1\sum_{j\Meg  j_0}(2^{j}t)^{a-k\grado} \Nc_{\infty,a,2^{-j}}(f*\psi_j)\meg \frac{C_1 2^{a-k\grado}}{1-2^{a-k\grado}} \sup_{j\in \N} \Nc_{\infty,a,2^{-j}}(f*\psi_j).
		\]
		This proves the first assertion by the arbitrariness of $\phi$ and $t$. The possibility of replacing $\grado$ with $\dd$ when $\Lc$ is as in the last part of the statement follows using the second assertion of~\cite[Lemma 2.34]{Calzi4}.
		
		We turn now to the proof of the second assertion. Take $p_0\in (0,p)$ so that $a>D/p_0$. By~\cite[Lemma 2.31]{Calzi4}, there is a constant $C_2>0$ such that
		\[
		\Nc_{\infty,a,2^{-j}}(f*\psi_j)\meg C_2 \Nc_{p_0,a,2^{-j}}(f*\psi_j)
		\]
		for every $j\in\N$. The assertion then follows from~\cite[Lemma 2.29]{Calzi4} and~\cite{Singular}.
	\end{proof}
	 
Now we can present the local Hardy spaces, whose definition is well posed thanks to Lemma~\ref{lem:85}.

	\begin{deff}
		Define $\dd_M\coloneqq \max_{X\in \gf} \deg(X)$ and $\dd_0\coloneqq \min_{X\in \gf\setminus \Set{0}} \deg(X)$ and take $p\in (0,\infty]$, $\lambda\Meg 0$, and $\kappa>0$. We define \[
        h^p_\lambda(G;\Banach):=\Bigl\{f\in \Sr'(G;\Banach)\colon\, \widetilde\Mc^{(\lambda,\kappa)} f\in L^p(G)\Bigr\},
        \]
        endowed with the corresponding quasi-norm.
	\end{deff}
One can observe that the local nature of these Hardy spaces arises from the definition of the maximal function $\widetilde\Mc^{(\lambda,\kappa)}$, although it is independent of $\kappa$. Indeed, the supremum in its definition is taken over $t\in(0,\kappa]$, and thus, regardless of how large $\kappa$ is, $t$ is always restricted to vary on a bounded range.
	We shall see later that $h^p_\lambda(G;\Banach)$ does not depend on the parameter $\lambda$, provided that 
\[
\lambda> \dd_M([Q_*(1/p-1)/\dd_0]+1),\dd([D/(p \dd)]+1)
\]
and that $G$ satisfies some additional assumptions (or $p>Q_*/(Q_*+\dd_0)$).
	
	We now discuss an unfortunate issue of the generality of the above definition. One would like that $h^p_0(G;\Banach)=L^p(G;\Banach)\cdot \beta$ for $p\in (1,\infty)$ for every \BanachSpace\ $\Banach$. Unfortunately, this is not the case, in general. In fact, it may happen that a measure $\mi\in \Mc(G;\Banach)$ such that $\abs{\mi}\meg h\cdot \beta$ for some positive $h\in L^p(G)$ has no density with respect to $\beta$ (see~\cite[Example 1 of Chapter III, No.\ 1]{DiestelUhl} for the case in which $G$ is the torus and $h=1$). 
	We shall therefore suitably `enlarge' the spaces $L^p(G;\Banach)$ so as to hide this difficulty. This will be also necessary when dealing with atomic decomposition, since we cannot require that the atoms belong to $L^q(G;\Banach)$, in general. 
	Notice, however, that this issues does \emph{not} arise when $\Banach$ has the Radon--Nikod\'ym property, for example when $\Banach$ is reflexive (see~\cite[Corollary 4 to Theorem 2 of Chapter III, No.\ 3]{DiestelUhl}).

To address this issue, we introduce suitably defined $L^p$ spaces, which turn out to be the natural ones to identify with our local Hardy spaces.
	
	\begin{deff}
		Take $p\in [1,\infty]$ and a \BanachSpace\ $\Banach$. We define
        \[
        \Lp^p(G;\Banach):=\Bigl\{f\in \Mc(G;\Banach)=\Lf(C_c(G);\Banach)\colon \abs{f}\in L^p(G)\cdot \beta\Bigr\}
        \]
        endowed with the norm
        \[
        f\mapsto \norm{\abs{f}}_{L^p(G)},
        \]
        where we identified $\abs{f}$ with an element of $L^p(G)$, as we shall always do in the sequel.
		We define the space $\Lp^p_\loc(G;\Banach)$ in a similar way.
	\end{deff}

	\begin{prop}\label{prop:42}
		Take a \BanachSpace\ $\Banach$,  $a>D$, and   $\lambda\Meg 0$. There is a constant $C>0$ such that the following hold:
		\begin{enumerate}
			\item[\textnormal{(1)}] if $f\in \Sr'(G;\Banach)$ and $ \widetilde \Mc^{(\lambda,\kappa)} f\in L^1_\loc(G)$ for some $\kappa>0$, then $f\in \Lp^1_\loc(G;\Banach)$ and $\abs{f}\meg C \widetilde \Mc^{(\lambda,\kappa)} f$;
			
			\item[\textnormal{(2)}] for every $\alpha>0$
			\[
			\beta\Big(\Set{x\in G\colon ( \widetilde\Mc^{(\lambda,1)} f)(x)>2^a\alpha}\Big)\meg \beta\Big(\Set{x\in G\colon ( \Mc^{(\lambda,a)} f)(x)>\alpha}\Big)\meg C\frac{\norm{ f}_{\Lp^1(G;\Banach)}}{\alpha},
			\]
			for every $f\in \Lp^1(G;\Banach)$,
			and
			\[
			2^{-a}\norm{\widetilde \Mc^{(\lambda,1)} f}_{L^p(G)}\meg \norm{ \Mc^{(\lambda,a)} f}_{L^p(G)}\meg C p' \norm{f}_{\Lp^p(G;\Banach)}
			\]
			for every $f\in \Lp^p(G;\Banach)$ and for every $p\in (1,\infty]$;

			\item[\textnormal{(3)}] the inclusion $h^1_\lambda(G;\Banach)\subseteq \Lp^1(G;\Banach)$ holds and the embedding is continuous, moreover for every $p\in (1,\infty]$ we have that  $h^p_\lambda(G;\Banach)= \Lp^p(G;\Banach)$ with the same topology. 
		\end{enumerate}
	\end{prop}
	
Before we pass to the proof, we need some technical lemmas. We begin with a definition that will be useful later on, as we will often move to the setting of the algebra $\gf$ and then go back to the group $G$, where some arguments cannot be handled directly.

	\begin{deff}
		Let $\mathrm{inj}(G)$ be the least upper bound of the set of $r>0$ such that $ \exp_G$  induces a diffeomorphism of $B_\gf(0,r)$ onto an open neighbourhood of $e$ in $G$, so that $\exp_G$ induces a diffeomorphism of $B_\gf(0,\mathrm{inj}(G))$ (interpreted as $\gf$ when $\mathrm{inj}(G)=+\infty$) onto an open neighbourhood $\mathbf U$ of $e$ in $G$. 	
	\end{deff}
	
From now on we will consider an \emph{ordered} orthonormal basis $(X_j)_{j\in J}$ of $\gf$ which induces a homogeneous basis $(Y_j)_{j\in J}$ of $\gf_*$, so that $Y_j=X_j+\gf_{\deg(X_j)^-}$ for every $j\in J$.  We set $d_j\coloneq \deg(X_j)$ for every $j\in J$,
\[
    \vect X^\alpha=\prod_{j\in J} X_j^{\alpha_j}, \, \vect Y^\alpha=\prod_{j\in J} Y_j^{\alpha_j}\text{ and } d_\alpha\coloneqq \sum_j \alpha_j \deg(X_j)\text{ for every }\alpha\in \N^J.
\]
We can identify $G_*$ with $\gf_*$ by means of $\exp_{G_*}$, and $\gf$ with $\gf_*$ (as vector spaces) by means of the orthonormal bases $(X_j)$ and $(Y_j)$. In addition, we identify $B(0,\mathrm{inj}(G))$ with $\mathbf U$ by means of $\exp_G$.  
	\begin{lem}\label{lem:72}
		Take $m\in\N$ and a compact neighbourhood $K$ of $e$ in $\mathbf U$. There are families $(a_{\ell,\alpha,\gamma})_{\abs{\gamma}\meg \abs{\alpha}\meg m}$, $\ell=1,2,3,4$, of real analytic functions on $\mathbf U$ and a constant $C>0$ such that
		\[
		\abs{a_{\ell,\alpha,\gamma}(x)}\meg C \abs{x}_*^{(d_\gamma-d_\alpha)_+}
		\]
		for every $x\in K$, for every $\ell=1,2,3,4$, and for every $\abs{\gamma}\meg\abs{\alpha}\meg m$,
		\begin{equation}\label{eq:16}
		\begin{aligned}
			\vect X^\alpha&=\sum_{\abs{\gamma}\meg \abs{\alpha}} a_{1,\alpha,\gamma} \vect Y^\gamma, & \vect Y^\alpha&=\sum_{\abs{\gamma}\meg \abs{\alpha}} a_{2,\alpha,\gamma} \vect X^\gamma,
		\end{aligned}
		\end{equation}
		and
		\begin{equation}\label{eq:17}
		\begin{aligned}
			\vect X^\alpha&=\sum_{\abs{\gamma}\meg \abs{\alpha}} a_{3,\alpha,\gamma} \vect (\vect X^\gamma)^R, & \vect (\vect X^\alpha)^R&=\sum_{\abs{\gamma}\meg \abs{\alpha}} a_{4,\alpha,\gamma} \vect X^\gamma
		\end{aligned}	
		\end{equation}
		for every $\abs{\alpha}\meg m$. In addition, $a_{\ell,\alpha,0}\neq 0$ if and only if $\alpha=0$.
	\end{lem}

	\begin{proof}
		We shall endow $\gf$, and then $\mathbf U$, with the dilations induced by $\gf_*$. 
		The existence (and uniqueness) of the families $(a_{\ell,\alpha,\gamma})$ so that~\eqref{eq:16} and~\eqref{eq:17} hold is clear. It is also clear that $a_{\ell,\alpha,0}\neq 0$ if and only if $\alpha=0$, since no $X_j$ and no $Y_j$ has constant terms. It will therefore suffice to prove that they satisfy the desired estimates on $K$. In fact, it will suffice to show that such estimates hold in some neighbourhood of $e$ in $G$.
		 
		Then,  observe that~\cite[Corollary 3.7]{ElstRobinson}  shows that
		\[
		\lim_{t\to 0^+} t^{d_\alpha} (t\,\cdot\,)^*\vect X^\alpha= \vect Y^\alpha 
		\] 
		on $B_\gf(0,\eta)$ for some $\eta\in (0,\mathrm{inj}(G)]$, that is,
		\begin{equation}\label{eq:1}
		\lim_{t\to 0^+}t^{d_\alpha-d_\gamma}a_{1,\alpha,\gamma} (t\,\cdot\,)= \delta_{\alpha,\gamma}
		\end{equation}
		for every $\alpha,\gamma$. Notice that $a_{1,\alpha,\gamma}$ is real analytic, so that it is equal to the sum of its homogeneous components in a neighbourhood of $0$  (where homogeneity is defined using the dilations of $\gf_*$). Therefore, \eqref{eq:1} proves that for every $\alpha$ there is a constant $c_\alpha>0$ such that
\[
\abs{a_{1,\alpha,\gamma}}\meg c_{\alpha} \abs{\,\cdot\,}_*^{(d_\gamma-d_\alpha)_+}
\]
on $B_\gf(0,\eta)$ for every $ \gamma$ (see~\cite[Proposition 6.1]{ElstRobinson}). This concludes the proof of the estimates  for the $a_{1,\alpha,\gamma}$.
		
		Notice that, clearly, $(a_{2,\alpha,\gamma}(x))$ (considered as a square matrix with $a_{2,\alpha,\gamma}=0$ when $\abs{\gamma}>\abs{\alpha}$) is the inverse of  $(a_{1,\alpha,\gamma}(x))$ for every $x\in \mathbf U$. Moreover, $x\mapsto\det(a_{1,\alpha,\gamma}(x))$ is a nowhere vanishing real analytic function on $\mathbf U$, and $a_{2,\gamma,\alpha}/\det(a_{1, \alpha,\gamma})$ is a linear combination of $\prod_{\alpha'\neq \alpha} a_{1,\alpha',\sigma(\alpha')}$, where
\[
\sigma\colon \Set{\alpha'\in \N^J\colon \abs{\alpha'}\meg m, \alpha'\neq \alpha}\to\Set{\gamma'\in \N^J\colon \abs{\gamma'}\meg m, \gamma'\neq \gamma}
\]
is a bijection. Observe that
		\[
		\abs*{\prod_{\alpha'\neq \alpha} a_{1,\alpha',\sigma(\alpha')}(x)}\meg \prod_{\alpha'\neq \alpha}c_{\alpha'}\abs{x}_*^{(\dd_{\sigma(\alpha')}-\dd_{\alpha'})_+}\meg \prod_{\alpha'\neq \alpha} c_{\alpha'} \abs{x}_*^{[\sum_{\alpha'\neq \alpha}(\dd_{\sigma(\alpha')}-\dd_{\alpha'})]_+}
		\]
		for every $x\in B_\gf(0,\eta)$ (assuming that $\eta$ is so small that $\abs{x}_*\meg 1$ for every $x\in B_\gf(0,\eta)$).
		Since $\sum_{\alpha'\neq \alpha}(\dd_{\sigma(\alpha')}-\dd_{\alpha'})= \dd_\alpha-\dd_{\gamma}$, the estimates for $a_{2,\gamma,\alpha}$ follow.
		
		 The last estimates follow from the previous ones and from the fact that, by homogeneity, there are two families of polynomials $(p_{\ell,\alpha,\gamma})_{\abs{\gamma}\meg \abs{\alpha}}$, $\ell=1,2$, with $p_{\ell,\alpha,\gamma}$ homogeneous of degree $d_\gamma-d_\alpha$ (hence $0$ if $d_\gamma<d_\alpha$), such that
		\[
		\vect Y^\alpha=\sum_\gamma p_{1,\alpha,\gamma} (\vect Y^\gamma)^R\qquad \text{and} \qquad (\vect Y^\alpha)^R =\sum_\gamma p_{2,\alpha,\gamma} \vect Y^\gamma
		\]
		on $\gf_*$.
	\end{proof} 
	
	\begin{lem}\label{lem:55}
		There is a family $(\phi_t)_{t\in (0,1]}$ of positive  elements of $C^\infty_c(G)$ such that the following hold:
		\begin{enumerate}
			\item[\textnormal{(1)}] there are $\kappa>0$ and $C>1$ such that
\[
\frac{t^{-Q_*}}{C}\chi_{B(e,\kappa t)}\meg \phi_t\meg C t^{-Q_*}\chi_{B(e,t)}
\]
 for every $t\in (0,1]$;
			
			\item[\textnormal{(2)}] $\int_G \phi_t\,\dd \beta=1$ for every $t\in (0,1]$;

			\item[\textnormal{(3)}] for every $\lambda>0$ there is a constant $C_\lambda>0$ such that
\[
\abs{X\phi_t}\meg C_\lambda \abs{X}t^{-Q_*-\deg(X)}
\]
 for every $t\in (0,1]$ and $X\in U_\lambda$.
		\end{enumerate}
	\end{lem}

	\begin{proof}
		By~\cite[Proposition 6.1]{ElstRobinson}, we may take $\eta\in (0,\mathrm{inj}(G)]$ and a constant $C_1>1$ such that 
		\[
		\frac{1}{C_1}\abs{\exp_G(x)}_*\meg \max_{j\in J}\abs{\langle x,X_j\rangle}^{1/d_j} \meg C_1 \abs{\exp_G(x)}_*
		\]
		for every $x\in B_\gf(0,\eta)$. We denote with $\theta$ the bijection from $B_\gf(0,\mathrm{inj}(G))$ onto $ \mathbf U$ induced by $\exp_G$. For every $r >0$, let
        \[
        B_r\coloneqq \Bigl\{x\in \gf\colon  \max_{j\in J}\abs{\langle x,X_j\rangle}^{1/d_j}\meg r\Bigr\}
        \]
        and take $\eta'>0$ so that $B_{\eta'}\subseteq B_\gf(0,\eta)$.
		Take a positive non-zero $\phi\in C^\infty_c(B_{\eta'})$ such that $\phi=1$ on $B_{\eta'/2}$, and define $\phi_t\in C^\infty_c(G)$ so that 
\[
\phi_t(\exp_G(x))=a_t t^{-Q_*}\phi((ct)^{-1} \cdot x )
\]
 for every $x\in B_{ct\eta'}$ (where $\cdot$ refers to the dilations of $\gf_*$, identified with $\gf$ by means of the basis $(X_j)$ and the corresponding homogeneous basis of $\gf_*$), while $\phi_t=0$ on $G\setminus \exp_G (B_{ct \eta'})$, and where $a_t$ is chosen so that $\int_G \phi_t \,\dd \beta=1$ for every $t\in (0,1]$, while $c\in (0,1]$ is a constant to be chosen later. 
		Observe that $a_t$ depends continuously on $t$ and has a finite non-zero limit for $t\to 0^+$ since $\theta^*(\beta)$ has a nowhere vanishing real analytic density with respect to the Lebesgue measure on $\gf$, so that $0<\inf_t a_t\meg \sup_t a_t <+\infty$. 
		Define $b_t\coloneqq \sup_{\phi_t(x)\neq 0} \abs{x}_*$, and observe that  $b_t\meg C_1 \eta' t$ for every $t\in (0,1]$, provided that $c$ is sufficiently small (so that $B(e,b_1)\subseteq \exp_G(B_{\eta'})$).   Thus, (1) and (2) hold. Finally, (3) is a consequence of Lemma~\ref{lem:72} and the previous remarks. 
	\end{proof}
	
	We also state the following simple corollary which will be useful later on.
	
	\begin{cor}\label{cor:21}
		Take $\kappa_2>\kappa_1>0$. There is a family $(\tau_t)_{t>0}$ of elements of $C^\infty_c(G)$ such that the following hold:
		\begin{enumerate}
			\item[\textnormal{(1)}] $\chi_{B(e,\kappa_1 t)}\meg \tau_t\meg \chi_{B(e,\kappa_2t)}$ for every $t\in (0,1]$;
			
			\item[\textnormal{(2)}] for every $\lambda>0$ there is a constant $C_\lambda>0$ such that 
\[
\abs{X Y^R \tau_t}\meg C_\lambda\abs{X}\abs{Y} \min(1,t)^{-\deg(X)-\deg(Y)}
\]
for every $X,Y\in U_\lambda$ and $t>0$.
		\end{enumerate}
	\end{cor}
	
	\begin{proof}
		It suffices to take $(\phi_t)$ as in Lemma~\ref{lem:55} and to set 
		\[
		\tau_t\coloneqq \phi_{\min((\kappa_2-\kappa_1)t/2 ,1)}*\chi_{B(e, (\kappa_1+\kappa_2) t/2)}*\phi_{\min((\kappa_2-\kappa_1)t/2 ,1)}
		\]
		for every $t>0$.
	\end{proof}

    We are now ready to prove Proposition~\ref{prop:42}.
	\begin{proof}[Proof of Proposition~\ref{prop:42}.]
		(1) Take $(\phi_t)_{t\in (0,1]}$ as in Lemma~\ref{lem:55}, so that there is a constant $C_1>0$ such that $C_1\phi_t \in \widetilde B_{\lambda,t}$ for every $t\in (0,1]$. Consequently, if $f\in \Sr'(G;\Banach)$ and $\widetilde \Mc^{(\lambda,\kappa)} f\in L^1_\loc(G;\Banach)$ for some $\kappa>0$, then
        \[
        \abs{f*\phi_t}\meg C_1^{-1} \widetilde \Mc^{(\lambda,\kappa)} f\text{ for every }t\in (0,1].
        \]
         Since $f*\phi_t$ converges to $f$ in $\Sr'(G;\Banach)$ and since $ \widetilde \Mc^{(\lambda,\kappa)} f\in L^1_\loc(G;\Banach)$, this proves that $f\in \Lp^1_\loc(G;\Banach)$ and that $\abs{f}\meg C_1^{-1} \widetilde \Mc^{(\lambda,\kappa)} f$.
		
		(2) Take $C_2>1$ such that 
\[
C_2^{-1} r^{Q_*}(1+r)^{Q_G-Q_*}\meg \beta(B(e,r))\meg C_2 r^{Q_*}(1+r)^{Q_G-Q_*}
\]
 for every $r>0$.
		Take $t\in (0,1]$, $\phi\in B_{\lambda,a, t}$, and $x\in G$, and observe that, if $f\in \Sr'(G;\Banach)\cap \Lp^1_\loc(G;\Banach)$, then 
		\[
		\begin{split}
		\abs{(f*\phi)(x)}&\meg \int_G \abs{f(y)} \abs{\phi(y^{-1}x)}\,\dd \beta(y)\\
			&\meg t^{-Q_*} \sum_{j\in\N} \int_{B'_{x,t,j}} \frac{\abs{f(y)}}{(1+d(y,x)/t)^{a}}\,\dd \beta(y) \\ 
			&\meg C_2 \sum_{j\in \N}  2^{jQ_*}(1+ t 2^j)^{Q_G-Q_*} (1+2^{j-1})^{-a}\dashint_{B(x,2^j t)} \abs{f(y)}\,\dd \beta(y)  \\
			&\meg  C_2 \sum_{j\in \N }  2^{jQ_*}(1+   2^j)^{D-Q_*} (1+2^{j-1})^{-a} (\Mc_{HL} \abs{f})(x) 
		\end{split}
		\]
		where $B'_{x,t,0}=B(x,t)$ and $B'_{x,t,j}=B(x,2^j t)\setminus B(x,2^{j-1}t)$ for $j\Meg 1$, and where $\Mc_{HL}$ denotes the Hardy--Littlewood maximal function.
		Consequently, by the arbitrariness of $\phi$ there is a constant $C_3>0$ such that
		\[
		\Mc^{(\lambda,a)} f\meg C_3 \Mc_{HL} \abs{f}.
		\]
		By the properties of the Hardy--Littlewood maximal function $\Mc_{HL}$ (see, e.g.,~\cite[Theorem 2.1 of Chapter III]{CoifmanWeiss}), the assertion follows.\footnote{Notice that we are applying the Hardy--Littlewood maximal function to $\abs{f}$, and not to $f$, so that the usual `scalar' theory applies.}
		
		(3) By (1), it is clear that $h^p_\lambda(G;\Banach)\subseteq \Lp^p(G;\Banach)$ continuously for every $p\in [1,\infty]$. By (2), it is also clear that $\Lp^p(G;\Banach)\subseteq h^p_\lambda(G;\Banach)$ continuously for every $p\in (1,\infty]$, whence the conclusion.
	\end{proof}

The following proposition collects some basic and classical properties of the spaces $h^p_\lambda(G;\Banach)$, such as completeness, the embedding into the tempered distributions, and a useful summability result.
	\begin{prop}\label{prop:41}
		Let $\Banach$ be a \BanachSpace, $\lambda \Meg 0$ and $p\in (0,1]$. The following hold,
        \begin{enumerate}
            \item $h^p_\lambda(G;\Banach)$ is complete and embeds continuously into $\Sr'(G;\Banach)$;
            \item if $(f_j)$ is a sequence of elements of $h^p_\lambda(G;\Banach)$ such that the sum $\sum_{j\in \N} \norm{f_j}_{h^p_\lambda(G;\Banach)}^p$ is finite, then the sum $\sum_{j\in \N} f_j$ converges in $h^p_\lambda(G;\Banach)$.	
        \end{enumerate}
	\end{prop}
	
	\begin{proof}
		(1) Let us first show that $h^p_\lambda(G;\Banach)$ embeds continuously into $\Sr'(G;\Banach)$. To this aim, let $B$ be a bounded subset of $\Sr(G)$, and observe that
		\[
		\langle f,\phi\rangle=[f*(\phi(x\,\cdot\,^{-1}))](x)
		\]
		for every $x\in G$, $f\in \Sr'(G;\Banach)$ and $\phi\in B$. Fix a maximally $2$-separated family $(x_j)_{j\in J}$ of elements of $G$, so that the $B(x_j,2)$ cover $G$, while the $B(x_j,1)$ are pairwise disjoint.  Take $\tau\in C^\infty_c(G)$ so that $\chi_{B(e,2)}\meg \tau \meg \chi_{B(e,3)}$. Observe that, by~\cite[Lemma 4.3]{Calzi2}, there is $N\in\N$ so that $ \sum_{j\in J} \chi_{B(x_j,3)}\meg N$, so that
\[
1\meg \widetilde \tau\coloneqq \sum_{j\in J} \tau(x_j^{-1}\,\cdot\,)\meg N.
\]
 Consequently, if we set $\tau_j\coloneqq\tau(x_j^{-1}\,\cdot\,)/\widetilde \tau$, then $\sum_{j\in J} \tau_j=1$ and the $\tau_j(x_j\,\cdot\,)$ are uniformly bounded in $C^\infty_c(G)$. 
		Take $a>D/p$ and observe that there is a constant $C_1>0$ such that
		\[
		\abs{(X \phi)(x)}\meg \frac{C_1}{(1+\abs{x}_*)^a}
		\]
		for every $\phi\in B$, for every $x\in G$, and for every $X\in U_{\lambda}$ with $\abs{X}\meg 1$. Consequently, there is a constant $C_2>0$ such that
		\[
		\begin{split}
		\abs{X^R((\phi\tau_j)(x\,\cdot\,^{-1}))(y)}&\meg  C_2 \chi_{B(x_j,3)}(x y^{-1})\sup_{ Y\in U_\lambda, \abs{Y}\meg 1 } \abs{(Y \phi)(x y^{-1})}\\
			&\meg C_1C_2 \chi_{B(x_j,3)}(x y^{-1}) \frac{1}{(1+\abs{x y^{-1}}_*)^a}\\
			&\meg   \frac{4^aC_1C_2}{(1+\abs{x_j}_*)^a}
		\end{split}
		\]
		for every $x,y\in G$, $j\in J$, and $\phi \in B$. In addition, $(\phi\tau_j)(x\,\cdot\,^{-1})$ is supported in $\overline B(e,3) x_j^{-1}x \subseteq B(e,4)$ as long as $x\in B(x_j,1)$. It then follows that
		\begin{align*}
		\abs{\langle f,\phi\rangle}
\meg &\sum_{j\in J}  \inf_{y\in B(e,1)}\abs{[f*((\phi \tau_j)(x_j y\,\cdot\,^{-1}))](x_j y)}\\
\meg& \sum_{j\in J} \frac{4^aC_1C_2}{(1+\abs{x_j}_*)^a}\inf_{y\in B(e,1)}(\widetilde \Mc^{(\lambda,4)} f)(x_j y),
		\end{align*}
		so that
		\[
		\abs{\langle f,\phi\rangle}^p \meg \sum_{j\in J}\frac{(4^aC_1C_2)^p}{(1+\abs{x_j}_*)^{ap}} \beta(B(e,1))^{-1}\norm{\widetilde \Mc^{(\lambda,4)} f}_{L^p(G)}^p
		\]
		for every $f\in \Sr'(G;\Banach)$ and $\phi\in B$.  The first assertion follows since $\sum_{j\in J}\frac{1}{(1+\abs{x_j}_*)^{ap}} $ is finite (see~Lemma~\ref{lem:37}).
		
        Next, let $(g_j)$ be a Cauchy sequence in $h^p_\lambda(G;\Banach)$. Then, $(g_j)$ is a Cauchy sequence in $\Sr'(G;\Banach)$, so that it converges to some $g$ in $\Sr'(G;\Banach)$, thanks to~\cite[Proposition 12 of Chapter III, \S\ 3, No.\ 8]{BourbakiTVS}. Now, clearly the mapping $f\mapsto (\widetilde \Mc^{(\lambda,4)}f)(x)$ is lower semi-continuous on $\Sr'(G;\Banach)$ for every $x\in G$, so that by Fatou's theorem also the mapping $f\mapsto \norm{\widetilde \Mc^{(\lambda,4)}f}_{L^p(G)}$ is \emph{sequentially} lower semi-continuous on $\Sr'(G;\Banach)$. As a consequence, $g\in h^p_\lambda(G;\Banach)$ and $(g_j)$ also converges to $g$ in $h^p_\lambda(G;\Banach)$, so that $h^p_\lambda(G;\Banach)$ is complete. 
		
        (2) Let $(f_j)$ be as in the statement. By the subadditivity of the mapping $f\mapsto \norm{f}_{h^p_\lambda(G;\Banach)}^p$, we then see that the section filter of the family of the finite sums $\sum_{j\in J} f_j$ is a Cauchy filter, so that it converges in $h^p_\lambda(G;\Banach)$.
	\end{proof}

    \begin{prop}\label{prop:25}
		Let $\Banach$ be a \BanachSpace, $p\in (0,1]$, $q\in (1,\infty]$, $\lambda \Meg 0$  and $a> Q_G(1/p-1/q)$. The following inclusion holds
        \[
        (1+\abs{\,\cdot\,}_*)^{-a}\Lp^q(G;\Banach)\subseteq h^p_\lambda(G;\Banach)
        \]
        and the embedding is continuous. In particular, $\Sr(G;\Banach)\subseteq h^p_\lambda(G;\Banach)$ continuously.
	\end{prop}
	
	\begin{proof}
		Let $\eps>0$ be such that $B(e,\eps)\subseteq \mathbf U$, $\rho\in (0,\min(\eps,1))$,   and $f\in (1+\abs{\,\cdot\,}_*)^{-a}\Lp^q(G;\Banach)$ with $\norm{(1+\abs{\,\cdot\,}_*)^{a} f}_{\Lp^q(G;\Banach)}=1$. Define $f_j\coloneqq \chi_{B_j} f$ for every $j\in\N$, where
        \[
        B_0\coloneqq B(e,1)\text{ and }B_j\coloneqq B(e,2^{j})\setminus B(e,2^{j-1})\text{ for every }j\Meg 1.
        \]
        Observe that there is a constant $C_1>0$ such that $\beta(B(e,r))\meg C_1 r^{Q_G}$ for every $r\Meg 1$, so that, setting $b\coloneqq a-Q_G(1/p-1/q)$,
		\[
		\norm{f_j}_{\Lp^q(G;\Banach)}\meg 2^{-a(j-1)}\meg 2^a C_1^{1/p-1/q} 2^{-j b}\beta(B(e,2^j))^{1/q-1/p}
		\]
		for every $j\in\N$. Now, observe that clearly $\widetilde \Mc^{(\lambda,1)} f_j$ is supported in $\overline B(e,2^j+1)\subseteq \overline B(e, 2^{j+1})$, so that by Proposition~\ref{prop:42} there is a constant $C_2>0$ such that
		\[
		\begin{split}
			\norm*{ \widetilde \Mc^{(\lambda,1)} f_j}_{L^p(G)} &\meg \norm*{ \widetilde \Mc^{(\lambda,1)} f_j}_{L^q(G)} \beta(B(e,2^{j+1}))^{1/p-1/q}\\
			&\meg C_{2} \norm{f_j}_{\Lp^q(G;\Banach)} \beta(B(e,2^{j+1}))^{1/p-1/q}\\
			&\meg   C_{2}2^a C_1^{1/p-1/q} 2^{-j b} \beta(B(e,2^j))^{1/q-1/p}\beta(B(e,2^{j+1}))^{1/p-1/q}\\
			&\meg C_{2}^2 2^{-j b}
		\end{split}
		\]
		since $q>1$ and $\beta $ is doubling.
		Therefore,
		\[
		\norm{f}_{h^p_\lambda(G;\Banach)}^p\meg \sum_{j\in \N} \norm{f_j}_{h^p_\lambda(G;\Banach)}^p\meg \frac{C_2^2}{1-2^{-b}}.
		\]
		The first assertion follows from Proposition~\ref{prop:41}. The second assertion is a simple consequence of the first one.
	\end{proof}

\section{Atomic decomposition for Local Hardy Spaces}\label{Section:AtomicDecomposition}
    In order to introduce atoms, we need some analog of the vanishing moments condition that has to be imposed in the classical situation and on homogeneous groups. One possible way to do this relies on providing a suitable notion of polynomials on small balls and then defining moments in terms of these polynomials. As we shall see later, this approach leads to the desired results when $G$ is nilpotent, whereas in the general case we only get partial results. We shall briefly discuss later a weaker notion of atoms which may in principle lead to better results for arbitrary groups.
    That being said, we shall define polynomials on small balls simply by tranferring them from the Lie algebra by means of the exponential map. Even though this approach does not lead to complete results, it still allows us to prove the density of $h^p_\lambda \cap h^1_\lambda$ in $h^p_\lambda$ in full generality.
    
	\begin{deff}
		We endow the dual  $\gf^*$ of $\gf$ with the increasing filtration $(\gf_\lambda^{ \perp\circ})_{\lambda\Meg 0 }$, where $\gf_\lambda^{\perp\circ}$ denotes the polar in $\gf^*$ of the orthogonal complement of $\gf_\lambda$ in $\gf$, and then extend it to a filtration $(\Pc_\lambda)_{\lambda \Meg 0}$ of the space $\Pc$ of polynomials on $\gf$ (identified with the symmetric algebra over $\gf^*$, see~\cite[Remark 2 of Chapter IV, \S\ 5, No.\ 11]{BourbakiA2}). We define the degree of $P\in \Pc$ as
        \[
        \deg P\coloneqq \min\Set{\lambda \Meg 0\colon P\in \Pc_\lambda}.
        \]
		
		For every $x\in G$, let $\Pc^{(x)}$ be the space of functions defined on $x\mathbf U$ given by
        \[
        \Pc^{(x)}=\Bigl\{(P\circ \exp_G^{-1})(x^{-1}\,\cdot\,) : P\in\Pc\Bigr\}.
        \]
        We then endow $\Pc^{(x)}$ 
		with the filtration $(\Pc^{(x)}_\lambda)$ induced by $(\Pc_\lambda)$, and define a degree accordingly.
		We also define $\Pc_{\lambda^-}\coloneqq \sum_{\lambda'<\lambda}\Pc_{\lambda'}$ and then $\Pc^{(x)}_{\lambda^-}$ accordingly.
	\end{deff}
	Note that  $\Pc^{(x)}=L_x \Pc^{(e)}$, where $L_x$ denotes left translation by $x$.
	
	Observe that we will often work with $\Pc^{(x)}_{\lambda^-}$ instead of $\Pc^{(x)}_\lambda$ since it will allow us to avoid technical issues related to the fact that there are only countably many $\lambda$ for which $\Pc^{(x)}_{\lambda^-}\neq\Pc^{(x)}_\lambda$ (issues  which force several unpleasant references to the set of those $\lambda$ in~\cite{FollandStein}).
	
	\begin{oss}
        If $(x_j)$ is the  basis of $\gf^*$ corresponding to $(X_j)$, then $\gf_\lambda^{\perp \circ}$ is the vector space generated by the $x_j$ with $d_j\meg \lambda$. In other words, if we identify $\gf$ with $\gf_*$ by means of the basis $(X_j)$ and the corresponding homogeneous basis $(Y_j)$ of $\gf_*$, then $\Pc_\lambda$ becomes the space of polynomials of homogeneous degree $\meg \lambda$ with respect to the dilation structure of $\gf_*$. The above definition, albeit more abstract, has the advantage of being `invariant', that is, of depending only on the filtration  and the fixed scalar product on~$\gf$.  In particular, the spaces $\Pc^{(x)}$ and their filtrations do not depend on the chosen orthonormal basis $(X_j)$ of $\gf$.
	\end{oss}

	\begin{deff}
		Take $\rho>0$ so that $\overline B(e,\rho)\subseteq \mathbf U$, a \BanachSpace\ $\Banach$, $p\in (0,1]$, $q\in (1,\infty]$, and $\lambda \Meg0$.   We say that an $a\in \Lp^1_\loc(G;\Banach)$ is a ($\Banach$-valued) $(p,q,\rho,\lambda)$-atom if the following hold:
		\begin{enumerate}
			\item $a$ is concentrated in $B(x,r)$ for some $x\in G$ and $r>0$;\footnote{This means that $\abs{a}$ is concentrated in $B(x,r)$.}
			
			\item $\norm{a}_{\Lp^q(G;\Banach)}\meg \beta(B(x,r))^{1/q-1/p}$;
			
			\item if $r\meg \rho$, then $\langle a, P\rangle=0$ for every $P\in \Pc^{(x)}_{\lambda^-}$.
		\end{enumerate}
	\end{deff}

    From now on we set $q(\lambda)\coloneqq [\lambda/\dd_0]+1$.
	The following results shows that we may ``extract'' derivatives from an atom supported in a small ball. This links the usual theory, which is based on vanishing moments, with the approach more commonly used in the current setting, which is based on extracting derivatives. We postpone its proof to Section~\ref{sec:app2}.

	\begin{lem}\label{lem:73} 
		Take $p\in [1,\infty]$ and $\lambda > 0$.  There are $\eps,c,C>0$ such that the following hold:
		\begin{itemize}
			\item $B(e,\eps)\subseteq \mathbf U$;
			
			\item for every $x\in G$, $\rho\in (0,\eps]$, \BanachSpace\ $\Banach$, and  $f\in \Lp^p(G;\Banach)$ with $\supp f \subseteq B(x,\rho)$ such that
			\[
			\langle f, P\rangle =0
			\]
			for every $P\in \Pc^{(x)}_{\lambda^-}$,  there is a family $(f_\alpha)_{0<\abs{\alpha}\meg q(\lambda)}$ of elements of $\Lp^p(G;\Banach)$, supported in $B(x,c\rho)$, such that
			\[
			f= \sum_{0<\abs{\alpha}\meg q(\lambda)} \vect X^\alpha f_\alpha
			\]
			and such that
			\[
			\norm{f_\alpha  }_{\Lp^p(G;\Banach)}\meg C\rho^{\max(\lambda,d_\alpha)} \norm{f}_{\Lp^p(G;\Banach)}
			\]
			for every $\alpha$ such that $0<\abs{\alpha}\meg q(\lambda)$, where $d_\alpha=\sum_{j\in J} \alpha_j \deg(X_j)$. 
		\end{itemize}
	\end{lem}

	This lemma is a convenient  version of~\cite[Lemma 1.60]{FollandStein}. The main differences are related to the fact the we work in a non-homogeneous setting (which imposes less precise decompositions) and that we wish to deal with compactly supported $L^p$  functions (or, rather, elements of $\Lp^p(G;\Banach)$) instead of Schwartz functions. The proof is therefore based on an elementary one-dimensional result instead of the Fourier transform. 
	\\Notice that, when $G$ is a homogeneous group endowed with the induced filtration (so that $G\cong G_*$ canonically), then the possibility of writing $f$ as $\sum_{d_\alpha> \lambda} \vect X^\alpha f_\alpha$ implies that $\langle f , P\rangle=0$ for every polynomial $P\in \Pc_\lambda$, and conversely (the direct implication is clear, even if the $f_\alpha$ are arbitrary distributions, whereas the converse implication follows from the proof of the above result). In this more general setting, there does not seem to be a converse to the above result.

    This result suggests a notion of weak $(p,q,\rho, \lambda)$-atoms, namely the  $a\in\Lp^1_\loc(G;\Banach)$ such that:
		\begin{enumerate}
			\item $a$ is concentrated in $B(x,r)$ for some $x\in G$ and $r>0$; 
			
			\item $\norm{a}_{\Lp^q(G;\Banach)}\meg \beta(B(x,r))^{1/q-1/p}$;
			
			\item if $r\meg \rho$, then there is a family $(a_\alpha)_{0<\abs{\alpha}\meg q(\lambda)}$ of elements of $\Lp^p(G;\Banach)$, supported in $B(x,r)$, such that
			\[
			a= \sum_{0<\abs{\alpha}\meg q(\lambda)} \vect X^\alpha a_\alpha
			\]
			and such that
			\[
			\norm{a_\alpha  }_{\Lp^p(G;\Banach)}\meg  \rho^{\max(\lambda,d_\alpha)}\beta(B(x,r))^{1/q-1/p}.
			\]
		\end{enumerate}
        In fact, if $a$ is a $(p,q,\rho,\lambda)$-atom (associated with the ball $B(x,r)$), then $ a/C$ is a weak $(p,q,c\rho, \lambda)$-atom (associated with the ball $B(x,cr)$) for  suitable $c,C>0$ independent of $a$. The proof of Proposition~\ref{prop:31} below shows that Calder\'on--Zygmund operators interact nicely with these weak atoms. In fact, weak atoms appear to be a little more natural, since they do not require the polynomials -- locally transplanted from the Lie algebra -- to be defined, and are somewhat similar in spirit to the atoms used in the atomic decomposition of Besov and Triebel--Lizorkin spaces (see, e.g.~\cite[Section 4]{Calzi4}).

\subsection{Calder\'on--Zygmund decomposition}
The proof of the atomic decomposition theorem relies on a number of preliminary results, in particular on a suitable Calder\'on--Zygmund decomposition, which we now establish following the approach of~\cite[Chapter 3]{FollandStein}.
    
	\begin{lem}\label{lem:80}
		Take $\rho>0$ so that $\overline B(e,\rho)\subseteq \mathbf U$, and take $\lambda,\lambda'\Meg 0$ and $\kappa>1$. There is a constant $C>0$ such that
		\[
		\norm{X  P}_{L^p(B(x,t),\beta)}\meg C \abs{X}t^{Q_*(1/p-1/q)-\deg(X)} \norm{P}_{L^q(B(x,t'),\beta)}
		\]
		for every $p,q\in [1,\infty]$, $X\in U_\lambda$, $t,t'\in (0,\rho]$ such that $t\meg \kappa t'$, $P\in \Pc^{(x)}_{\lambda'}$, and $x\in G$.
	\end{lem}
	
	Notice that this lemma would be trivial if  $t,t'$ were fixed, essentially because $\Pc^{(e)}_{\lambda'}$ is finite-dimensional. 
	
	\begin{proof}
		By left-invariance, we may reduce to proving the assertion for $x=e$. In addition, by H\"older's inequality, we may assume that $p=\infty$ and $q=1$.  By~\cite[Proposition 6.1]{ElstRobinson}, we may take $\rho'\in (0,\min(1,\rho)]$  and $C_1>1$ so that
		\[
		\frac{1}{C_1}\abs{\exp_G(x)}_*\meg\max_{j\in J} \abs{\langle x , X_j\rangle}^{1/\deg(X_j)}\meg C_1 \abs{\exp_G(x)}_*
		\] 
		for every  $x\in \gf$ with $\max_{j\in J} \abs{\langle x , X_j\rangle}^{1/\deg(X_j)}\meg\rho'$ or $\abs{\exp_G(x)}_*\meg \rho'$. 
		In addition, observe that, since $\Pc^{(e)}_{\lambda'}$ is finite-dimensional (and consists of real analytic functions on the connected set $\mathbf U$), we may find a constant $C_2>0$ such that
		\[
		\norm{X  P}_{L^\infty(B(e,\rho),\beta)}\meg C_2 \norm{X P}_{L^\infty(\exp_G( B_{\rho'/C_1}),\beta  )}
		\]
		for every $P\in \Pc^{(e)}_{\lambda'}$ and for every $X\in U_\lambda$, 
		where for every $r>0$
        \[
        B_r\coloneqq \Bigl\{x\in \gf\colon \max_{j\in J} \abs{\langle x , X_j\rangle}^{1/\deg(X_j)}<r\Bigr\}.
        \]
		We may then assume that $t,t'\meg \rho'/C_1$, transfer the problem to $\gf$ by means of the exponential map, and replace $\beta$ with the Lebesgue measure $\beta_\gf$ on $\gf$, since the density of $\beta$ with respect to $\beta_\gf$ is a nowhere-vanishing real analytic function.  
		Now, take $k\in \N$ so that, if $d_\alpha\meg \lambda$, then $\abs{\alpha}\meg k$.
		Since $\Pc_{\lambda'}$ is finite-dimensional (and consists of real analytic functions), we then may find a constant $C_3>0$ such that
		\[
		\norm{\vect Y^\alpha P}_{L^\infty( B_{C_1} ,\beta_\gf)}\meg C_3\norm{P}_{L^1(B_{1/(\kappa C_1)},\beta_\gf)}
		\]
		for every $P\in \Pc_{\lambda'}$ and for every $\alpha$ with length $\meg k$. Applying the previous inequality to $P(t\,\cdot\,)$ we then see that
		\[
		\begin{split}
		\norm{\vect Y^\alpha P}_{L^\infty( B(e,t) ,\beta)}&\meg \norm{\vect Y^\alpha P}_{L^\infty(B_{C_1t} ,\beta_\gf)}\\
			&\meg C_3 t^{ -d_\alpha-Q_*}\norm{P}_{L^1( B_{t/(\kappa C_1)},\beta_\gf)}\\
			&\meg C_3 t^{ -d_\alpha-Q_*}\norm{P}_{L^1(B(e,t'),\beta_\gf)}
		\end{split}
		\]
		since $t\meg \kappa t'$.  By Lemma~\ref{lem:72} we may find a  constant $C_4>0$ and a family $(a_{\alpha,\gamma})_{\abs{\gamma}\meg \abs{\alpha}\meg k}$ so that $\vect X^\alpha=\sum_{\abs{\gamma}\meg \abs{\alpha}} a_{\alpha,\gamma} \vect Y^\gamma$ and $\abs{a_{\alpha,\gamma}(x)}\meg C_4 \abs{x}_*^{(d_\gamma-d_\alpha)_+}$ for every $x\in B(e,\rho)$. Then,
		\[
		\begin{split}
			t^{d_\alpha}\norm{\vect X^\alpha P}_{L^\infty( B(e,t) ,\beta)}&\meg C_4 t^{d_\alpha}\sum_{\abs{\gamma}\meg \abs{\alpha}} t^{(d_\gamma-d_\alpha)_+ }\norm{\vect Y^\gamma P}_{L^\infty( B(e,t) ,\beta)}\\
				&\meg C_3C_4 \sum_{\abs{\gamma}\meg \abs{\alpha}} t^{-Q_* }\norm{P}_{L^1(B(e,t'),\beta_\gf)}
		\end{split}
		\]
		since $t\meg \rho'/C_1\meg 1$. The assertion follows.
	\end{proof}
	
	\begin{cor}\label{cor:22}
		Take $\rho>0$ so that $\overline B(e,\rho)\subseteq \mathbf U$, and take $\lambda,\lambda'\Meg 0$ and $\kappa_1,\kappa_2>0$. 
		Take a family $(\tau_{x,t})_{x\in G, t>0}$ of elements of $C^\infty(G)$ such that $\tau_{x,t}$ is supported in $\overline B(x,\kappa_1 t)$ and
		\[
		\max_{X\in U_{\lambda'}, \abs{X}\meg 1} \norm{ X \tau_{x,t} }_{L^\infty(G)}\meg 1
		\]
		for every $x\in G$ and $t>0$.
		Then, for every \BanachSpace\ $\Banach$, $f\in \Sr'(G;\Banach)$, $x\in G$, and $t\in (0,\rho/\kappa_1]$, there is a unique $P_{f,x,t}\in \Pc^{(x)}_{\lambda^-}\otimes \Banach$ so that  \[
        \langle f, Q\tau_{x,t}\rangle= \langle P_{f,x,t}, Q\tau_{x,t}\rangle
        \]
        for every $Q\in \Pc^{(x)}_{\lambda^-}$.

		In addition, there is a constant $C>0$ such that 
		\[
		\norm{P_{f,x,t}\tau_{x,t}}_{L^\infty(G)}\meg C (\widetilde\Mc^{(\lambda',\kappa_1+\kappa_2)} f)(y)
		\]
		for every $x,y\in G$ with $d(x,y)\meg \kappa_2 t$, \BanachSpace\ $\Banach$, $f\in \Sr'(G;\Banach)$, and $t>0$ such that $t\in(0,\rho/\kappa_1]$.
	\end{cor}
	
	This result replaces~\cite[Lemma 3.8]{FollandStein} in our setting. The proof is slightly simpler.
	
	\begin{proof}
		Observe first that $P_{f,x,t}$ is unique and well defined since the bilinear mapping
        \[
        (P,Q)\mapsto \langle P\vert Q \tau_{x,t} \rangle
        \]
        is non-degenerate on $\Pc^{(x)}_{\lambda^-}$. In fact, if $(p_j)$ is an orthonormal basis of $\Pc_{\lambda^-}^{(x)}$, endowed with the scalar product induced by $L^2(\tau_{x,t}\cdot \beta)$, then $P_{f,x,t}=\sum_j  \langle f, p_j\tau_{x,t}\rangle p_j$. In addition, observe that, if we prove the assertion for $\langle f, v'\rangle$ for every $v'\in \overline B_{\Banach'}(0,1)$ (so that $P_{\langle f, v'\rangle,x,t}=\langle P_{f,x,t},v'\rangle$), then we get the assertion for $f$. Consequently, we may reduce to the case $\Banach=\C$ (or $\R$). 
		Then, observe that
		\[
		\langle f, \overline{P_{f,x,t}}\tau_{x,t}\rangle=\langle P_{f,x,t},\overline{ P_{f,x,t}}\tau_{x,t}\rangle=\norm{P_{f,x,t} }_{L^2(\tau_{x,t}\cdot \beta)}^2.
		\]
		Moreover, 
		\[
		\langle f,\overline{P_{f,x,t}}\tau_{x,t}\rangle= (f* (\overline{P_{f,x,t}}\tau_{x,t})(y\,\cdot\,^{-1}))(y)
		\]
		and by Lemma~\ref{lem:80}   there is a constant $C_1>0$ such that
		\[
		\norm{  P_{f,x,t}\tau_{x,t}}_{L^\infty(G)}\meg C_1 t^{-Q_*/2}\norm{P_{f,x,t}  }_{L^2(\tau_{x,t}\cdot \beta)}
		\]
		and such that, for every $X\in U_{\lambda'}$ with $\abs{X}\meg 1$ and for every $y,z\in G$,
		\[
		\begin{split}
			\abs{ X^R[(P_{f,x,t}\tau_{x,t})(y\,\cdot\,^{-1})](z)}&= \abs{ [X^\dag(P_{f,x,t}\tau_{x,t})](y z^{-1})}\\
				&\meg  \norm{ X^\dag(P_{f,x,t}\tau_{x,t})}_{L^\infty(G)}\\
				&\meg C_1 t^{-\deg(X)} \norm{  P_{f,x,t}\tau_{x,t}}_{L^\infty(G)}.
		\end{split}
		\]
		In addition, observe that $\abs{z}_*\meg \abs{y^{-1}x}_*+\kappa_1 t\meg (\kappa_1+\kappa_2)t$ for every $z\in G$ such that $y z^{-1}\in \supp \tau_{x,t}\subseteq \overline B(x,\kappa_1 t)$.   Consequently, by the definition of $\widetilde \Mc^{(\lambda',\kappa_1+\kappa_2)}$ we see that
		\[
		\abs{(f* (\overline{P_{f,x,t}}\tau_{x,t})(y\,\cdot\,^{-1}))(y)}\meg  C_1(\kappa_1+\kappa_2)^{Q_*+\deg(X)} t^{Q_*} \norm{  P_{f,x,t}\tau_{x,t}}_{L^\infty(G)}(\widetilde \Mc^{(\lambda',\kappa_1+\kappa_2)} f)(y),
		\]
		so that there is a constant $C_2>0$ such that
		\[
		\begin{split}
		\norm{  P_{f,x,t}\tau_{x,t}}_{L^\infty(G)}^2&\meg C_1^2 t^{-Q_*}\norm{P_{f,x,t}  }_{L^2(\tau_{x,t}\cdot \beta)}^2\\
			&\meg C_2 \norm{  P_{f,x,t}\tau_{x,t}}_{L^\infty(G)}(\widetilde \Mc^{(\lambda',\kappa_1+\kappa_2)} f)(y),
		\end{split}
		\]
		whence the conclusion since $\norm{  P_{f,x,t}\tau_{x,t}}_{L^\infty(G)}$ is finite.
	\end{proof}
	
	\begin{lem}\label{lem:82}
		Take $\kappa>0$ and $\lambda\Meg 0$, and set $k\coloneqq \max\Set{\abs{\alpha}\colon \alpha \in \N^J, d_\alpha\meg \lambda}$.
		 Then,  
		\[
		\widetilde \Mc^{(\lambda,\kappa)}(fg)(x) \meg 2^k  (\widetilde \Mc^{(\lambda,\kappa)}f)(x)  \max_{d_\alpha \meg \lambda} t^{d_\alpha}\norm{\vect X^\alpha g}_{L^\infty(G)} 
		\]
		for every $t\in (0, \kappa ]$, $x\in G$, \BanachSpace\ $\Banach$, 
		$f\in \Sr'(G;\Banach)$, and $g\in C^\infty(G)$, provided that   $g$ is supported in $\overline B(x , t)$ when $t<  \kappa$. 
	\end{lem}

	\begin{proof}
		For simplicity, we may reduce to the case in which $ \max_{d_\alpha \meg \lambda}t^{d_\alpha}\norm{\vect X^\alpha g}_{L^\infty(G)}= 1$. 
		Take $t'\in (0,\kappa]$ and $\phi\in \widetilde B_{\lambda, t'}$.
		Observe first that, for every $x\in G$, 
		\[
		[(fg)*\phi](x)=\langle f, g \phi(\,\cdot\,^{-1}x)\rangle= [f*(g(x\,\cdot\,^{-1}) \phi)](x).
		\]
		Set $B\coloneqq \overline B(e, t)$ if $t< \kappa $, and $B\coloneqq G$ otherwise, so that $g$ is supported in $x B$ and $B$ is symmetric. In particular, $g(x\,\cdot\,^{-1})\phi$ is supported in $\overline B(e,t')\cap B\subseteq \overline B(e,\min( t,t'))$.
		In addition, given $\alpha$ with $d_\alpha\meg \lambda$ and $z\in G$, we have
		\[
		\begin{split}
			& \min( t,t')^{Q_*+d_\alpha} \abs{(\vect X^\alpha)^R(g(x\,\cdot\,^{-1}) \phi)(z)}\\
				&\qquad \meg  \min( t,t')^{Q_*+d_\alpha} \sum_{\alpha_1+\alpha_2=\alpha} \frac{\alpha!}{\alpha_1!\alpha_2!} \abs{(\vect X^{\alpha_1} g)(x z^{-1}) (\vect X^{\alpha_2})^R\phi(z) }\\
				&\qquad \meg \min( t,t')^{Q_*+d_\alpha} \sum_{\alpha_1+\alpha_2=\alpha} \frac{\alpha!}{\alpha_1!\alpha_2!} t^{-d_{\alpha_1}} t'^{-Q_*-d_{\alpha_2}}\\
				&\qquad\meg    2^{\abs{\alpha}}.  
		\end{split}
		\]
		The assertion follows.
	\end{proof}
	
	\begin{cor}\label{cor:23}
		Keep the hypotheses and the notation of Corollary~\ref{cor:22}, take $\kappa>0$, and define $P_{f,x,t}\coloneqq 0$ whenever $\kappa_1 t> \rho$.  Then, there is a constant $C>0$ such that
		\[
		\widetilde \Mc^{(\lambda',\kappa)} ((f-P_{f,x,t})\tau_{x,t})(y)\meg C(\widetilde \Mc^{(\lambda', \max(\kappa,\kappa_1+\kappa_2))} f)(y)
		\]
		for every \BanachSpace\ $\Banach$, $f\in \Sr'(G;\Banach)$, $x,y\in G$ with $d(x,y)\meg \kappa_2 t$, and $t>0$.
	\end{cor}
	
	This result replaces~\cite[Lemma 3.12]{FollandStein} in our setting.

	\begin{proof}
		Take $t'\in (0,1]$ and $\phi\in \widetilde B_{\lambda',t'}$, we wish to estimate $ \abs{[((f-P_{f,x,t})\tau_{x,t})*\phi](y)}$. The term $\abs{[( f \tau_{x,t})*\phi](y)}$ may be estimated by means of Lemma~\ref{lem:82}, applied with $t=\min(\kappa_1 t, \kappa)$  and $g=\tau_{x,t}$. 
		Then, observe that, if  $P_{f,x,t}\neq 0$ (so that $\kappa_1 t\meg \rho$), by Corollary~\ref{cor:22} we may find a constant $C_1>0$ such that
		\[
		\begin{split}
			\abs*{\int_G P_{f,x,t}(z) \tau_{x,t}(z) \phi(z^{-1}y)\,\dd \beta(z)}&\meg C_1(\widetilde \Mc^{(\lambda',\kappa_1+\kappa_2)} f)(y)\int_{B(x,\kappa_1 t)} \phi(z)\,\dd \beta(z)\\
				&\meg C_1^2 (\widetilde \Mc^{(\lambda',\max(\kappa,\kappa_1+\kappa_2))} f)(y).
		\end{split}
		\]
		The assertion follows.
	\end{proof}

Recall that $q(\lambda)\coloneqq [\lambda/\dd_0]+1$.
	\begin{lem}\label{lem:77}
		Take $\rho>0$  such that $\overline B(e,\rho)\subseteq \mathbf U$, and take $\lambda \Meg 0$. 
		There are  $c,C>0$ such that, for every $x\in G$ and $\phi\in C^{q(\lambda)}(G)$ supported in $\overline B(x,\rho)$, there is $Q\in \Pc^{(x)}_{\lambda^-}$ such that
		\[
		\abs{ \vect X^\alpha(\phi-Q)(y)  }\meg C\max_{ \abs{\gamma}\meg q(\lambda) } d(x,y)^{[\max(\lambda,d_\gamma)-d_\alpha]_+} \sup_{z\in B(x,cd(x,y))}\abs{(\vect X^\gamma   \phi)(z)}
		\]
		for every $y\in B(x,\rho)$ and for every $\alpha$ with length $\meg q(\lambda)$, while
		\[
		\sup_{y\in B(x,t)} \abs{Q(y)}\meg C \max_{d_\alpha<\lambda} t^{d_\alpha}\abs{(\vect X^\alpha   \phi)(x)}
		\]
		for every $t\in (0,\rho]$.
	\end{lem}

	\begin{proof}
		We may reduce to the case $x=e$ and transfer the problem to $\gf$ by means of $\exp_G$. In addition,  identify $\gf$ with $\gf_*$ by means of the bases $(X_j)$ and $(Y_j)$, and identify  $\gf_*$ with $G_*$ by means of $\exp_{G_*}$.  
		Furthermore, note that $\max_{j\in J} \abs{\langle \,\cdot\, , X_j\rangle}^{1/\deg(X_j)}$ is a homogeneous norm on $\gf_*$ and, by~\cite[Proposition 6.1]{ElstRobinson}, we may find a constant $C_1>1$ such that
		\[
		\frac{1}{C_1}\abs{x}_*\meg \max_{j\in J} \abs{\langle x , X_j\rangle}^{1/\deg(X_j)}\meg C_1 \abs{x}_*
		\]
		for every $x\in \overline B(e,\rho)$.\footnote{Notice that the cited reference only ensures the existence of such a $C_1$ in some neighbourhood of $e$. The assertion may then be extended to the whole compact set $\overline B(e,\rho)$ by continuity, possibly for a larger $C_1$.}
		Set  $\lambda'\coloneqq \max \Set{\mi< \lambda\colon U_{\mi}\neq U_{\mi^-}}$ and let $Q$ be the unique element of $\Pc_{\lambda^-}=\Pc_{\lambda'}$ such that
        \[
        (\vect Y^\alpha Q)(0)=(\vect Y^\alpha \phi)(0)
        \]
        for every $\alpha$ with $d_\alpha <\lambda$, that is, $d_\alpha\meg \lambda'$.  Observe that,   if $d_\alpha \meg \lambda'$, then $\abs{\alpha}\meg q(\lambda')-1$.
		Let us prove that, if $\alpha\in \N^J$ and $d_\alpha \meg\lambda'$, then
        \[
        (\vect Y^\gamma \vect Y^\alpha Q)(0)=(\vect Y^\gamma \vect Y^\alpha \phi)(0)
        \]
        for every $\gamma$ with $d_\gamma\meg \lambda'- d_\alpha$. 
		Observe that $\vect Y^\gamma \vect Y^\alpha =\sum_{d_{\gamma'}=d_\gamma+d_\alpha} a_{\gamma'} \vect Y^{\gamma'}$ for some family $(a_\gamma')$ of real numbers. Consequently,
		\[
		(\vect Y^\gamma \vect Y^\alpha Q)(0)=\sum_{\dd_{\gamma'}=d_\gamma+d_\alpha} a_{\gamma'}  (\vect Y^{\gamma'}Q)(0)=\sum_{\dd_{\gamma'}=d_\gamma+d_\alpha} a_{\gamma'}  (\vect Y^{\gamma'}\phi)(0)=(\vect Y^\gamma \vect Y^\alpha \phi)(0),
		\]
		whence our claim.   Observe also that, if $d_\gamma\meg \lambda'-d_\alpha$, then $\abs{\gamma}\meg q(\lambda'-d_\alpha)-1$. 
		Since clearly $\abs{\alpha}\meg d_\alpha/\dd_0$, we have
        \[
        \frac{\lambda'-d_\alpha}{\dd_0} \meg \frac{\lambda'}{\dd_0}-\abs{\alpha},
        \text{ whence }
        q(\lambda'-d_\alpha)-1\meg q(\lambda')-1-\abs{\alpha},
        \]
        so that $\abs{\gamma}+\abs{\alpha}\meg q(\lambda')-1$.
		Consequently,~\cite[Theorem 1.37]{FollandStein}, applied to the opposite group of $G_*$, shows that there are $c',C_2>0$ such that
		\[
		\abs{  \vect Y^\alpha (\phi-Q)(y)}\meg C_2 \max_{\abs{\gamma}\meg q(\lambda'-d_\alpha),d_\gamma>\lambda'-d_\alpha  }\norm{y}^{d_\gamma} \sup_{\norm{z}\meg c'\norm{y}}\abs{(\vect Y^\gamma \vect Y^\alpha \phi)(z)}
		\]
		for every $\alpha$ with $d_\alpha\meg \lambda'$, that is, $d_\alpha< \lambda$. Observe that, if $d_\gamma> \lambda'-d_\alpha$, then $d_\gamma\Meg \lambda-d_\alpha$. 		
		Applying Lemma~\ref{lem:72} we then see that there is a constant $C_3>0$ such that
		\[
		\abs{  \vect Y^\alpha (\phi-Q)(y)}\meg C_3 \max_{\abs{\gamma}\meg q(\lambda')}\abs{y}_*^{\max(\lambda,d_\gamma)-d_\alpha} \sup_{z\in B(e, C_1c'\abs{y}_*)}\abs{(\vect X^\gamma   \phi)(z)}
		\]
		for every $\alpha$ with $d_\alpha < \lambda$. If, otherwise, $d_\alpha>\lambda$ (but $\abs{\alpha}\meg q(\lambda)$), then arguing as before we see that
		\[
		\abs{  \vect Y^\alpha (\phi-Q)(y)}=\abs{  (\vect Y^\alpha \phi)(y)}\meg  \max_{\abs{\gamma}\meg q(\lambda)}\abs{y}_*^{[\max(\lambda,d_\gamma)-d_\alpha]_+} \sup_{z\in B(e,  \abs{y}_*)}\abs{(\vect X^\gamma   \phi)(z)}.
		\]
		Another application of Lemma~\ref{lem:72} leads to the first assertion. For what concerns the second assertion, observe that there is a constant $C_4>0$ such that
		\[
		\sup_{y\in B(e,\rho)} \abs{Q(y)}\meg C_4 \max_{d_\alpha \meg \lambda'} \abs{(\vect Y^\alpha Q)(0)}=C_4 \max_{d_\alpha \meg \lambda'} \abs{(\vect Y^\alpha \phi)(0)},
		\]
		so that the conclusion follows arguing as in the proof of Lemma~\ref{lem:80}.
	\end{proof}

	\begin{lem}\label{lem:84}
	Keep the hypotheses and the notation of Corollary~\ref{cor:22} and define $P_{f,x,t}\coloneqq 0$ whenever $\kappa_1t> \rho$. Take $a\Meg 0$ and assume that  $\lambda \Meg a-Q_*$ and that  $\lambda' \Meg \dd_M ([\lambda/\dd_0]+1)$, where $\dd_M=\max\Set{\deg(X)\colon X\in \gf}$. 
	Take   $\kappa_3>\kappa_1 $ and set  $\kappa\coloneqq \max\big(  1,\frac{\kappa_1+\kappa_2 }{\kappa_3-\kappa_1},\kappa_1+\kappa_2\big)$. There is a constant $C>0$ such that
	\[
	\widetilde \Mc^{(\lambda',1)} ((f-P_{f,x,t})\tau_{x,t})(y)\meg C\chi_{B(x,\kappa_1 t+1)}(y)\frac{(\widetilde \Mc^{(\lambda',\kappa)} f)(z)}{(1+\abs{y^{-1}x}/t)^a}
	\]
	for every \BanachSpace\ $\Banach$, $f\in \Sr'(G;\Banach)$, $x,y,z\in G$  with $d(x,y)\Meg \kappa_3 t$ and $d(x,z)\meg \kappa_2 t$, and $t>0$.
\end{lem}	

This replaces~\cite[Lemma 3.13]{FollandStein} in our setting. 

\begin{proof}
	\textsc{Step I} Take $t'\in (0,1]$ with $t'\meg t\kappa_1/\rho$  and $\phi\in \widetilde B_{\lambda',t'}$.  
	Arguing as in the proof of Corollary~\ref{cor:23}, we see that, for every $y,z\in G$,
	\[
	[((f-P_{f,x,t})\tau_{x,t})*\phi](y)=[f*( \tau_{x,t}(z\,\cdot\,^{-1}) \phi(\,\cdot\, z^{-1} y)  )](z) - \int_G P_{f,x,t}(z') \tau_{x,t}(z') \phi(z'^{-1}y)\,\dd \beta(z').
	\]
	In addition, by Corollary~\ref{cor:22} there is a constant $C_1>0$ such that, if $\kappa_1 t\meg \rho$, then
	\[
	\norm{P_{f,x,t}  \tau_{x,t}}_{L^\infty(G)}\meg C_1 (\widetilde \Mc^{(\lambda',\kappa_1+\kappa_2)} f)(z)
	\]
	for every $x\in G$ and for every $z\in B(x,\kappa_2 t)$. Consequently,  there is a constant $C_2>0$ such that, 
	\[
	\begin{split}
		\abs*{ \int_G P_{f,x,t}(z') \tau_{x,t}(z') \phi(z'^{-1}y)\,\dd \beta(z') }&\meg C_1 (\widetilde \Mc^{(\lambda',\kappa_1+\kappa_2)} f)(z) t'^{-Q_*}\beta(B(x,\kappa_1 t)\cap B(y,t'))\\
			&\meg C_2 (\widetilde \Mc^{(\lambda',\kappa_1+\kappa_2)} f)(z)\chi_{[0,\kappa_1 t+t')}(d(x,y))\\
			&\meg  C_2(1+\kappa_1(1+1/\rho))^a \frac{(\widetilde \Mc^{(\lambda',\kappa_1+\kappa_2)} f)(z)}{(1+\abs{y^{-1}x}_*/t)^a}
	\end{split}
	\]
	for every $x,y,z\in G$ such that $d(x,y)\Meg \kappa_3 t$ and $d(x,z)\meg \kappa_2 t$. 
	Observe that, if $d(x,z)\meg \kappa_2 t$, $d(x,y)\Meg \kappa_3 t$ and $\tau_{x,t}(z z'^{-1}) \phi(z' z^{-1}y)\neq 0$, then $z z'^{-1}\in \overline B(x,\kappa_1 t)\cap \overline B(y,t')$ so that
	\[
	\abs{z'}_*\meg \abs{z' z^{-1}x}_* + \abs{z^{-1} x}_* \meg \kappa_1 t+ \kappa_2 t
	\]
	and  
	\[
	t'\Meg \abs{z' z^{-1} y}_*\Meg \abs{x^{-1}y}_*-\abs{z' z^{-1}x}_* \Meg (1-\kappa_1/\kappa_3)\abs{y^{-1}x}_*.
	\]
	This shows that $t\kappa_1/\rho > t'\Meg (\kappa_3-\kappa_1) t$. Consequently, $\tau_{x,t}(z\,\cdot\,^{-1}) \phi(\,\cdot\, z^{-1} y)$ is supported in
    \[
    \overline B(e, (\kappa_1+\kappa_2)t)\subseteq \overline B\big(e, \frac{\kappa_1+\kappa_2 }{\kappa_3-\kappa_1} t'\big).
    \]
	Furthermore,  by Corollary~\ref{cor:21} there is a constant $C_3>0$ such that, for every $\alpha$ with $d_\alpha \meg \lambda'$,
	\[
	\begin{split}
		& \abs{(\vect {X}^\alpha)^R  [\tau_{x,t}(z\,\cdot\,^{-1}) \phi(\,\cdot\, z^{-1} y)  ](z') }\\
			&\qquad\meg  \sum_{\alpha_1+\alpha_2=\alpha} \frac{\alpha!}{\alpha_1!\alpha_2!} \abs{(\vect {X}^{\alpha_1} \tau_{x,t})(zz'^{-1}) (\vect X^{\alpha_2})^R\phi(z' z^{-1} y)  }\\
			&\qquad\meg C_3   \sum_{\alpha_1+\alpha_2=\alpha} \frac{\alpha!}{\alpha_1!\alpha_2!} t^{-d_{\alpha_1}} t'^{-Q_*-d_{\alpha_2}}  \\
			&\qquad \meg C_3^2 \bigg(\frac{\kappa_1+\kappa_2 }{\kappa_3-\kappa_1} t'\bigg)^{-Q_*-d_\alpha}.
	\end{split}
	\]
    Thus,  $\tau_{x,t}(z\,\cdot\,^{-1}) \phi(\,\cdot\, z^{-1} y)\in \widetilde B_{\lambda', (\kappa_1+\kappa_2 )/(\kappa_3-\kappa_1)}$.
	Using that if $\tau_{x,t}(z\,\cdot\,^{-1}) \phi(\,\cdot\, z^{-1} y)  \neq 0$, then
    \[
    t\kappa_1/\rho\Meg t'\Meg (1-\kappa_1/\kappa_3) \abs{y^{-1}x}_*\Meg (\kappa_3-\kappa_1) t,
    \]
    we see that there is a constant $C_4>0$ such that
	\[
	\abs{[((f-P_{f,x,t})\tau_{x,t})*\phi](y)}\meg C_4 \frac{(\widetilde \Mc^{(\lambda',\kappa)} f)(z)}{(1+\abs{y^{-1}x}/t)^a}
	\]
	for every $x,y,z\in G$ such that $d(x,y)\Meg \kappa_3 t$ and $d(x,z)\meg \kappa_2 t$.
	
	\textsc{Step II} Take $t'\in (0,1]$ with $t'> t\kappa_1/\rho  $ and $\phi\in \widetilde B_{\lambda',t'}$. In particular, $\kappa_1 t< \rho$.  
	Take $(\tau'_t)_{t>0}$ as in Corollary~\ref{cor:21} with $\kappa_1=\rho$ and $\kappa_2=\kappa_4$, where $\kappa_4>\rho$ is chosen so that $\overline B(e, \kappa_4)\subseteq \mathbf U$. Consequently, if we set $\tau'_{x,t}\coloneqq \tau_t'(x^{-1}\,\cdot\,)$, then
	\[
	[((f-P_{f,x,t})\tau_{x,t})*\phi](y)=\langle (f-P_{f,x,t})\tau_{x,t}, \phi(\,\cdot\,^{-1}y)\rangle= \langle (f-P_{f,x,t})\tau_{x,t}, \tau'_{x,1}\phi(\,\cdot\,^{-1}y)\rangle.
	\]	 
	By Lemma~\ref{lem:77} (and Lemma~\ref{lem:80}), we may find  $c,C_5>0$  and $Q\in \Pc_{\lambda^-}^{(x)}$ so that
	\[
	\begin{split}
		\abs{\vect X^\alpha (\tau_{x,1}' \phi(\,\cdot\,^{-1}y)-Q)(y')}&\meg C_5 \max_{\abs{\gamma}\meg q(\lambda)} d(x,y')^{[\max(\lambda,d_\gamma)-d_\alpha]_+} \sup_{z'\in B(x, c d(x,y'))}\abs{\vect X^\gamma (\tau_{x,1}' \phi(\,\cdot\,^{-1}y))(z')}\\
			&\meg C_5^2\max_{\abs{\gamma}\meg q(\lambda)} t^{[\max(\lambda,d_\gamma)-d_\alpha]_+}  t'^{-Q_*-d_{\gamma}} 
	\end{split}
	\]
	for every $y'\in B(x, \kappa_1 t)$ and for every $\alpha$ with length $\meg q(\lambda)$, and such that 
	\[
	\sup_{B(e,t'')} \abs{\vect X^\alpha Q}\meg C_5 \max_{d_\gamma<\lambda} t''^{d_\gamma-d_\alpha} \abs{\vect X^\gamma (\tau_{x,1}'\phi(\,\cdot\,^{-1}y))(x)}=C_5 \max_{d_\gamma<\lambda} t''^{d_\gamma-d_\alpha} \abs{(\vect X^\gamma)^R \phi(x^{-1}y)}
	\]
	for every $t''\in (0,  \kappa_4]$ and for every $\alpha$ with $d_\alpha \meg \lambda'$. In particular, combining the cases $t''=\kappa_1 t$ and $t''=\kappa_4$, we see that 
	\[
	\sup_{B(e,\kappa_1 t)} \abs{\vect X^\alpha Q}\meg C_5\max(\kappa_1^\lambda, \kappa_1^{-d_\alpha},\kappa_4^\lambda, \kappa_4^{-d_\alpha}) \max_{d_\gamma<\lambda} t^{(d_\gamma-d_\alpha)_+} \abs{(\vect X^\gamma)^R \phi(x^{-1}y)},
	\]
	for every $\alpha$ with $d_\alpha \meg \lambda'$. Consequently, there is a constant $C_6>0$ such that
	\[
	\begin{split}
	\abs{\vect X^\alpha (\tau_{x,1}' \phi(\,\cdot\,^{-1}y)-Q)(y')}&\meg \sum_{\alpha_1+\alpha_2=\alpha} \abs{(\vect X^{\alpha_1} \tau_{x,1}')(y')}\abs{(\vect X^{\alpha_2})^R \phi(y'^{-1}y)}+ \abs{(\vect X^{\alpha} Q)(y') }\\
		&\meg C_6  t'^{-Q_*-d_\alpha} + C_6   \max_{d_\gamma< \lambda} t^{(d_\gamma-d_\alpha)_+}  t'^{-Q_*-d_\gamma} 
	\end{split}
	\]
	for every $\alpha$ with $d_\alpha \meg \lambda'$ and $\abs{\alpha}>q(\lambda )$, and for every $y'\in B(e, \kappa_1 t )$. 
	Taking into account the fact that, if $\abs{\gamma}\meg q(\lambda )$, then $d_\gamma\meg \dd_M([\lambda /\dd_0]+1)\meg \lambda'$, we see that  the previous estimates imply that there is a constant $C_7>0$ such that
	\[
	\abs{\vect X^\alpha (\tau_{x,1}' \phi(\,\cdot\,^{-1}y)-Q)(y')}\meg C_7 \max_{d_\gamma\meg \lambda'} t^{[\max(\lambda ,d_\gamma)-d_\alpha]_+}   t'^{-Q_*-d_{\gamma}} 
	\]
	for every $\alpha$ with $d_\alpha\meg \lambda'$ and for every $y'\in B(x,\kappa_1 t)$.
	
	Set $R\coloneqq \tau_{x,1}' (\phi(\,\cdot\,^{-1}y)-Q)$, so that
	\[
	[((f-P_{f,x,t})\tau_{x,t})*\phi](y)= \langle (f-P_{f,x,t})\tau_{x,t}, R\rangle=[f*((\tau_{x,t}R)(z\,\cdot\,^{-1}) )](z)-\int_G P_{f,x,t}\tau_{x,t} R\,\dd \beta(z').
	\]
	Now, observe that $\tau_{x,t}R$ is supported in $\overline B(x,\kappa_1 t)$ and that, by Corollary~\ref{cor:21} and the previous computations, there is a constant $C_8>0$ such that
	\[
	\begin{split}
		\abs{ \vect X^\alpha (\tau_{x,t}R)(z')} &\meg  \sum_{\alpha_1+\alpha_2=\alpha} \frac{\alpha!}{\alpha_1!\alpha_2!} \abs{(\vect X^{\alpha_1}  \tau_{x,t} )(z')} \abs{(\vect X^{\alpha_2} R)(z')}\\
			&\meg C_8 \max_{ \alpha_1+\alpha_2=\alpha }   t^{-d_{\alpha_1}} \max_{d_\gamma\meg \lambda'} t^{[\max(\lambda ,d_\gamma)-d_{\alpha_2}]_+} t'^{-Q_*-d_\gamma}
	\end{split}
	\]
	for every $z'\in G$ and for every $\alpha$ with $d_\alpha \meg \lambda'$. In addition, since $1\Meg t'\Meg t \kappa_1/\rho $ and $\lambda \Meg a-Q_*$,
	\[
	t'^{a-Q_*-d_\gamma}\meg t'^{a-Q_*-\max(\lambda ,d_{\alpha_2},d_\gamma)}\meg (\kappa_1/\rho)^{a-Q_*-\max(\lambda ,d_{\alpha_2},d_\gamma)} t^{a-Q_*-\max(\lambda ,d_{\alpha_2},d_\gamma)}
	\]
	for every $\gamma$ and $\alpha_2$, so that 
	\[
	t^{-d_{\alpha_1}+(\max(\lambda ,d_\gamma)-d_{\alpha_2})_+} t'^{a-Q_*-d_\gamma}\meg (\kappa_1/\rho)^{a-Q_*-\max(\lambda ,d_{\alpha_2},d_\gamma)} t^{a-Q_*-d_{\alpha_1}-d_{\alpha_2}}
	\]
	for every $\gamma$ and $\alpha_1,\alpha_2$. Moreover, observe that, if $\abs{y^{-1}x}_*> \kappa_1 t+ t'$, then $[((f-P_{f,x,t})\tau_{x,t})* \phi](y)=0$, so that we may assume that $\kappa_3 < \abs{y^{-1}x}_*/t\meg \kappa_1+ t'/t$.
	Consequently, there is a constant $C_9>0$ such that
	\[
	\abs{ \vect X^\alpha (\tau_{x,t}R)(z')}\meg C_9   \frac{ t^{-Q_*-d_\alpha} }{(\abs{x^{-1}y}_*/t)^a} \meg C_9^2   \frac{ t^{ -Q_*-d_\alpha} }{(1+\abs{x^{-1}y}_*/t)^a}
	\]
	for every $z'\in G$ and for every $\alpha$ with $d_\alpha \meg \lambda'$. Using Corollary~\ref{cor:22}, we then see that there is a constant $C_{10}>0$ such that
	\[
    \begin{split}
	\abs*{\int_G P_{f,x,t}\tau_{x,t} R\,\dd \beta(z')}&\meg C_{10} (\widetilde \Mc^{(\lambda',\kappa)} f)(z) \frac{t^{-Q_*}\beta(B(x,\kappa_1 t))}{(1+\abs{x^{-1}y}_*/t)^a}\\
        &\meg C_{10}^2 (\widetilde \Mc^{(\lambda',\kappa)} f)(z) \frac{1}{(1+\abs{x^{-1}y}_*/t)^a}.
    \end{split}
	\]
	Analogously,  
	\[
	\begin{split}
		\abs{(\vect{X}^\alpha)^R ((\tau_{x,t}R)(z\,\cdot\,^{-1}))(z')}&=\abs{\vect{X}^\alpha (\tau_{x,t}R)(z z'^{-1})}\\
			&\meg  \chi_{B(x,\kappa_1 t)}(z z'^{-1}) C_9^2   \frac{ t^{ -Q_*-d_\alpha} }{(1+\abs{x^{-1}y}_*/t)^a}\\
			&\meg \chi_{B(e,(\kappa_1+\kappa_2)t)} C_9^2   \frac{ t^{ -Q_*-d_\alpha} }{(1+\abs{x^{-1}y}_*/t)^a}
	\end{split}
	\]
	for every $z'\in G$ and $\alpha$ with $d_\alpha \meg \lambda'$, since $\abs{z'}_*\meg \abs{z' z^{-1}x}_*+\abs{z^{-1}x}_*\meg \kappa_1 t+ \kappa_2 t$ for every $z'\in G$ such that $z z'^{-1} \in B(x,\kappa_1 t)$.
	Consequently,
	\[
	\abs*{[f*((\tau_{x,t}R)(z\,\cdot\,^{-1}) )](z)}\meg C_9^2  \frac{  (\widetilde \Mc^{(\lambda',\kappa)} f)(z)  }{(1+\abs{x^{-1}y}_*/t)^a}   ,
	\]
	so that the assertion follows.	
\end{proof}

	The next result is a convenient version of Whitney's covering lemma.
	\begin{lem}\label{lem:81}
		Let $\kappa_1,\kappa_2,\kappa_3>0$ be so that $2< \kappa_1\meg \kappa_2<\kappa_3$. There are $N\in\N$ and $c>1$ such that for every proper open subset $\Omega$ of $G$ with $\beta(\Omega)<\infty$ there is a countable family $(x_j,r_j)_{j\in J}$ of elements of $G\times (0,+\infty)$ such that the following hold:
		\begin{itemize}
			\item the $B(x_j,r_j)$, $j\in J$, are pairwise disjoint;
			
			\item the $B(x_j, \kappa_1 r_j)$, $j\in J$, cover $\Omega$;
			
			\item the $B(x_j,\kappa_2 r_j)$, $j\in J$, are contained in $\Omega$;
			
			\item $B(x_j,\kappa_3 r_j)\setminus \Omega \neq \emptyset$ for every $j\in J$;
			
			\item $\sum_{j\in J} \chi_{B(x_j,\kappa_2 r_j)}\meg N\chi_\Omega$ on $G$;
			
			\item for every $j,j'\in J$ such that $B(x_j,\kappa_2 r_j)\cap B(x_{j'},\kappa_2 r_{j'})\neq \emptyset$ one has $r_j/c \meg r_{j'}\meg c r_j$.
		\end{itemize}
	\end{lem}
	
	 Notice that, when $G$ is not compact, the assumption that $\beta(\Omega)<\infty$ implies that $\Omega \neq G$, so that in this case it is not necessary to assume $\Omega$ to be a proper subset of $G$.  
	
	\begin{proof}
		For every $x\in \Omega$, define
        \[
        r'(x)\coloneqq d(x,G\setminus \Omega),
        \]
        so that $r'(x)\in (0,+\infty)$ for every $x\in \Omega$ since $G\setminus \Omega$ is a non-empty closed subset of $G$. In addition, define $R$ as the diameter of $G$ when $G$ is compact, and as $\sup \Set{r>0\colon \beta(B(e,r))\meg \beta(\Omega)}$ if $G$ is not compact,  and observe that $r'(x)\meg R$ for every $x\in G$. Take $\kappa\in (2,\kappa_1)$ and $\eta,\eta'>0$ so that   $\kappa_1 \eta'= \kappa \eta$, and $\eta'\kappa_2< 1<\eta'\kappa_3$. 
		Then, by Vitali's covering lemma (see, e.g.,~\cite[Lemma 4]{Singular}), we may find a countable family $(x_j)_{j\in J}$ of elements of $\Omega$ such that the $B(x_j,\eta r'(x_j))$ are pairwise disjoint, while the $B(x_j, \kappa \eta r'(x_j))$ cover $\Omega$. Set $r_j\coloneqq \eta' r'(x_j)$ for every $j\in J$. Since $r_j\meg \eta r'(x_j)$ for every $j\in J$, it is clear that the $B(x_j, r_j)$, $j \in J$, are pairwise disjoint, while since $\kappa_1r_j= \kappa \eta r'(x_j)$ for every $j\in J$, it is clear that the $B(x_j, \kappa_1 r_j)$, $j\in J$, cover $\Omega$. 
		Furthermore, since $\kappa_2 r_j\meg r'(x_j)$ for every $j\in J$, it is  clear that $B(x_j,\kappa_2 r_j)\subseteq \Omega$ for every $j\in J$, while since $\kappa_3 r_j> r'(x_j)$, it is also clear that  $B(x_j,\kappa_3 r_j)\setminus \Omega \neq \emptyset$ for every $j\in J$. Let us then prove that there is $N\in\N$ such that $\sum_{j\in J} \chi_{B(x_j,\kappa_2 r_j)} \meg N$.
		Observe that, if $x\in \Omega$ and  $J_x\coloneqq \Set{j\in J\colon d(x,x_j)<\kappa_2 r_j}$, then for every $j\in J_x$
		\[
		r_j/\eta'= d(x_j,G\setminus \Omega)\meg d(x,G\setminus \Omega)+d(x,x_j)<r'(x)+ \kappa_2 r_j
		\]
		and, analogously,
		\[
		r'(x)=d(x,G\setminus \Omega)\meg d(x,x_j)+ d(x_j,G\setminus \Omega)<\kappa_2 r_j+  r_j/\eta',
		\]
		so that
		\[
		(1/\eta'-\kappa_2) r_j<r'(x)<(1/\eta'+\kappa_2) r_j.
		\] 
		This proves, by the way, the last assertion.
		Consequently,
		\[
		\bigcup_{j\in J_x}B(x_j,r_j)\subseteq \bigcup_{j\in J_x}B(x,(\kappa_2+1)r_j)\subseteq B(x,c_1 r'(x))
		\]
		and
		\[
		\bigcap_{j\in J_x} B(x_j, c_2 r_j)\supseteq \bigcap_{j\in J_x}B(x, (c_2-\kappa_2)r_j)\supseteq B(x,c_1 r'(x))
		\]
		where $c_1\coloneqq \frac{\kappa_2+1}{1/\eta'-\kappa_2}$ and $c_2\coloneqq \kappa_2+ c_1(1/\eta'+\kappa_2)$.
		Now, since $\beta$ is doubling, there are $C>0$ and $k\in \N$ so that $\beta(B(e, tr))\meg C t^k \beta(B(e,r))$ for every $r>0$ and for every $t\Meg 1$. Therefore,
		\[
		\begin{split}
			\beta(B(x,c_1 r'(x))&\Meg \sum_{j\in J_x} \beta(B(x_j, r_j))\\ 
				&\Meg \frac{1}{C c_2^k}\sum_{j\in J_x} \beta(B(x_j, c_2 r_j )  )\\
				&\Meg \card(J_x) \frac{1}{C c_2^k} \beta(B(x, c_1 r'(x) )),
		\end{split}
		\]
		so that $\card(J_x)\meg C c_2^k$ for every $x\in \Omega$.
	\end{proof}
 
    The next result provides the Calder\'on--Zygmund decomposition necessary to deal with atomic decomposition in this context.
    Recall that $\dd_M= \max\{\deg(X)\colon X\in \gf\}$ and $\dd_0=\min\{\deg(X)\colon X\in \gf, X\neq 0\}$.
     
	\begin{teo}\label{prop:50}
		Take $\rho>0$ so that $\overline B(e,\rho)\subseteq \mathbf U$ and take $\lambda, \lambda'\Meg 0$, $p\in (0,1]$, $a>Q_*/p$, and $\kappa_1,\kappa_2,\kappa_3>0$  so that  $\lambda'\Meg \dd_M([\lambda/\dd_0]+1)$,   $\lambda \Meg a-Q_* $, $\kappa_1>2$, and $3 \kappa_1< \kappa_2<\kappa_3$. Set 
        \[
        \kappa=\max\bigg(1, \frac32 \kappa_1+\kappa_2, \frac{3 \kappa_1+(3c+1) \kappa_3}{\kappa_2-3\kappa_1}, \frac3 2\kappa_1+[(3c+1)/2]\kappa_3\bigg),
        \]
        where $c$ is as in Lemma~\ref{lem:81}.
		Take a \BanachSpace\ $\Banach$, $f\in h^p_{\lambda'}(G;\Banach)$, and $\alpha>0$ such that the open set $\Omega_\alpha\coloneqq \Set{x\in G\colon (\widetilde \Mc^{(\lambda',\kappa)} f)(x)>\alpha}$ is different from $G$. Let $(x_{\alpha,j},r_{\alpha,j})_{j\in \N}$ be a sequence of elements of $G\times (0,+\infty)$ satisfying the properties of Lemma~\ref{lem:81}, and take $(\tau_\lambda)_{\lambda>0}$ as in Corollary~\ref{cor:21} corresponding to $\kappa_1 = \kappa_1$ and $\kappa_2=(3/2)\kappa_1 $; set $\tau\coloneqq \sum_{j\in \N} \tau_{r_{\alpha,j}}(x_{\alpha,j}^{-1}\,\cdot\,)$, so that $\tau\in C^\infty(\Omega_\alpha)$ and $1\meg \tau \meg N$ on $\Omega_\alpha$, with the notation of Lemma~\ref{lem:81}. Set $\zeta_{\alpha,j}\coloneqq \tau_{r_{\alpha,j}}(x_{\alpha,j}^{-1}\,\cdot\,)/\tau$, and define  
		\[
		b_{\alpha,j}\coloneqq (f-P_{\alpha,j})\zeta_{\alpha,j},
		\] 
		where $P_{\alpha,j}=0$ if $\kappa_2 r_{\alpha,j} >\rho$, while  $P_{\alpha,j}$ is the unique element of $\Pc^{(x_{\alpha,j})}_{\lambda^-}\otimes \Banach$ such that $\langle f, Q \zeta_{\alpha,j}\rangle=\langle P_{\alpha,j}, Q \zeta_{\alpha,j}\rangle$ for every $Q\in  \Pc^{(x_{\alpha,j})}_{\lambda^-}$ otherwise. There is a constant $C>0$ such that the following hold
\begin{enumerate}
\item the sum $\sum_{j\in \N} b_{\alpha,j}$ converges in $h^p_{\lambda'}(G;\Banach)$ and
		\[
		\sum_{j\in \N} \norm{\widetilde \Mc^{(\lambda',1)}b_{\alpha,j}}_{L^p(G)}^p\meg C\norm{\chi_{\Omega_\alpha}\widetilde \Mc^{(\lambda',\kappa)} f}_{L^p(G)} ^p;
		\]
\item  if we set $g_\alpha\coloneqq f-\sum_{j\in J} b_{\alpha,j}$, then
		\[
		(\widetilde \Mc^{(\lambda',1)} g_\alpha)(x)\meg C \chi_{G\setminus \Omega_\alpha}(x)(\widetilde \Mc^{(\lambda',\kappa)} f)(x)+C\alpha\sum_{j\in \N} \frac{\chi_{B(x_{\alpha,j},1+(3/2)\kappa_1 r_{\alpha,j})}(x)}{(1+\abs{x^{-1}x_{\alpha,j}}_*/r_{\alpha,j})^a}
		\]
		for every $x\in G$;
\item 	\[
		\norm{\widetilde \Mc^{(\lambda',1)} g_\alpha}_{L^1(G)}\meg C \alpha^{1-p} \norm{\widetilde \Mc^{(\lambda',\kappa)} f}_{L^p(G)}^p.
		\]
\end{enumerate}
	\end{teo}
	
	This replaces~\cite[Theorems 3.17 and 3.20 (i), and Lemma 3.19]{FollandStein} in our setting.
	Observe, by the way, that $\alpha^p\beta(\Omega_\alpha)\meg \norm{\widetilde \Mc^{(\lambda',\kappa)} f}_{L^p(G)}^p$ by Chebychev's inequality, so that $\beta(\Omega_\alpha)<\infty$; hence, Lemma~\ref{lem:81} may be applied. 
    Before we pass to the proof, we need a lemma.

	\begin{lem}\label{lem:37c}
		Take $\alpha>0$, $\kappa>1$, and $p\in (0,+\infty]$. There is a constant $C>0$ such that
		\[
		\norm{ \chi_{B(e,1+\kappa t)}(1+\abs{\,\cdot\,}_*/t)^{-Q_*/p-\alpha}  }_{L^p(G)}\meg C t^{Q_*/p}(1+t)^{(Q_G-Q_*)/p}
		\]
		for every   $t>0$.
	\end{lem}
	\begin{proof}
		The case $p=\infty$ is trivial, so that we may assume that $p<\infty$. In addition, the assertion follows from Lemma~\ref{lem:37} when $Q_G\meg Q_*$ (so that $D=Q_*$), so that we may assume that $Q_G>Q_*$.
		Then,  take a constant $C_1>0$ such that  $\beta(B(e,r))\meg C_1 r^{Q_*}$ for every $r\meg 2$ and $\beta(B(e,r))\meg C_1 r^{Q_G}$ for every $r\Meg 1/2$ and		
		define $B_{t,j}\coloneqq B(e,2^{j+1} t)\setminus B(e, 2^j t)$ for every $t>0$ and for every $j\in\Z$. In addition, set $N_t\coloneqq [\log_2(1/t+\kappa)]$, so that $2^{N_t}\meg1/t+\kappa\meg 2^{N_t+1}$. Therefore,
		\[
		\begin{split}
			&\norm{ \chi_{B(e,1+\kappa t)}(1+\abs{\,\cdot\,}_*/t)^{-Q_*/p-\alpha}  }_{L^p(G)}^p \meg \sum_{j\meg N_t}\int_{B_{t,j}} \frac{1}{(1+\abs{x}_*/t)^{Q_*+p\alpha}}\,\dd \beta(x)\\
				&\qquad\qquad\qquad\meg  C_1 \sum_{j\meg [\log_2(1/t)]} \frac{( 2^{j+1} t)^{Q_*}}{(1+2^j)^{Q_*+p\alpha}} + C_1 \sum_{[\log_2(1/t)]< j \meg N_t} \frac{( 2^{j+1} t)^{Q_G}}{(1+2^j)^{Q_*+p\alpha}} .
		\end{split}
		\]
		If $t\Meg 1$, then
		\[
		\sum_{j\meg [\log_2(1/t)]} \frac{( 2^{j+1} t)^{Q_*}}{(1+2^j)^{Q_*+p\alpha}}\meg  (2t)^{Q_*} \sum_{j\meg [\log_2(1/t)]} 2^{jQ_*}\meg \frac{2^{Q_*}}{1-2^{-Q_*}}\meg \frac{2^{Q_*}}{1-2^{-Q_*}}t^{Q_G}
		\]
		and
		\[
		 \sum_{[\log_2(1/t)]< j \meg N_t} \frac{( 2^{j+1} t)^{Q_G}}{(1+2^j)^{Q_*+p\alpha}} \meg t^{Q_G} \sum_{j\meg N_t+1} 2^{jQ_G}\meg t^{Q_G}\frac{(2+2\kappa)^{Q_G}}{1-2^{-Q_G}} .
		\]
		If, otherwise, $t<1$, then
		\[
		\sum_{j\meg [\log_2(1/t)]} \frac{( 2^{j+1} t)^{Q_*}}{(1+2^j)^{Q_*+p\alpha}}\meg t^{Q_*}\sum_{j\in \Z} \frac{ 2^{(j+1)Q_*}  }{(1+2^j)^{Q_*+p\alpha}},
		\]
		and
		\[
		\sum_{[\log_2(1/t)]< j \meg N_t} \frac{( 2^{j+1} t)^{Q_G}}{(1+2^j)^{Q_*+p\alpha}}\meg (2t)^{Q_G}\sum_{0\meg j \meg N_t}  2^{j(Q_G-Q_*)}\meg 2^{Q_G} t^{Q_G}\frac{(1/t+\kappa)^{Q_G-Q_*}}{1-2^{Q_*-Q_G}},
		\]
		which is bounded by a constant  times $t^{Q_*}$. The assertion follows.
	\end{proof}
    
	\begin{proof}[Proof of Theorem~\ref{prop:50}.]
	(1) Observe first that by Corollary~\ref{cor:21} there is a constant $C_1>0$ such that
		\[
		\norm{\vect X^\alpha \zeta_{\alpha,j}}_{L^\infty(G)}\meg C_1 \min(1,r_{\alpha,j})^{-d_\alpha}
		\]
		for every $\alpha$ with $d_\alpha \meg \lambda'$.
		Consequently, by Corollary~\ref{cor:23} and Lemma~\ref{lem:84},  there is a constant $C_2>0$ such that
		\[
		\widetilde \Mc^{(\lambda',1)}b_{\alpha,j}\meg C_2 \widetilde \Mc^{(\lambda',\kappa)} f
		\]
		on $B(x_{\alpha,j},\kappa_2 r_{\alpha,j})$, while
		\[
		(\widetilde \Mc^{(\lambda',1)}b_{\alpha,j})(x)\meg C_2\chi_{B(x_{\alpha,j},1+(3/2)\kappa_1 r_{\alpha,j})} (x)\frac{(\widetilde \Mc^{(\lambda',\kappa)} f)(y)}{(1+\abs{x^{-1}x_{\alpha,j}}_*/r_{\alpha,j})^a}
		\]
		for every $x\in G\setminus B(x_{\alpha,j}, (\kappa_2/2) r_{\alpha,j})$  and for every $y\in B(x_{\alpha,j},[(3c+1)/2]\kappa_3 r_{\alpha,j})$. Since we may choose $y\in B(x_{\alpha,j},\kappa_3 r_{\alpha,j})\setminus \Omega_\alpha$, we then see that
		\[
		\begin{split}
		(\widetilde \Mc^{(\lambda',1)}b_{\alpha,j})(x)&\meg C_2\chi_{B(x_{\alpha,j},1+(3/2)\kappa_1 r_{\alpha,j})} (x)\frac{\alpha}{(1+\abs{x^{-1}x_{\alpha,j}}_*/r_{\alpha,j})^a}\\
			&\meg C_2 \chi_{B(x_{\alpha,j},1+(3/2)\kappa_1 r_{\alpha,j})} (x)\frac{1}{(1+\abs{x^{-1}x_{\alpha,j}}_*/r_{\alpha,j})^a} \inf_{  B(x_{\alpha,j}, \kappa_2 r_{\alpha,j})} \widetilde \Mc^{(\lambda',\kappa)} f,
		\end{split}
		\]
		so that by Lemma~\ref{lem:37c}  there is a constant $C_2'>0$ such that
		\[
		\begin{split}
			&\norm{\chi_{G\setminus B(x_{\alpha,j}, \kappa_2 r_{\alpha,j})}\widetilde \Mc^{(\lambda',1)}b_{\alpha,j}}_{L^p(G)}\\
				&\qquad\meg  C_2' r_{\alpha,j}^{Q_*/p}(1+r_{\alpha,j})^{(Q_G-Q_*)/p} \beta(B(e,\kappa_2 r_{\alpha,j}))^{-1/p} \norm{\chi_{B(x_{\alpha,j}, \kappa_2 r_{\alpha,j})}  \widetilde \Mc^{(\lambda',\kappa)} f}_{L^p(G)}  ,
		\end{split}
		\]
		and then there is a constant $C_3>0$ such that
		\[
		\norm{\widetilde \Mc^{(\lambda',1)}b_{\alpha,j}}_{L^p(G)}\meg C_3\norm{\chi_{B(x_{\alpha,j}, \kappa_2 r_{\alpha,j})}  \widetilde \Mc^{(\lambda',\kappa)} f}_{L^p(G)} .
		\]
		Since there is $N\in\N$ such that $\sum_{ j\in\N}\chi_{B(x_{\alpha,j}, \kappa_2 r_{\alpha,j})}\meg N\chi_{\Omega_\alpha}$, we then infer that 
		\[
		\sum_{j\in \N} \norm{\widetilde \Mc^{(\lambda',1)}b_{\alpha,j}}_{L^p(G)}^p\meg C_3^p N \norm{\chi_{\Omega_\alpha}\widetilde \Mc^{(\lambda',\kappa)} f}_{L^p(G)} ^p,
		\]
		whence the first assertion by Proposition~\ref{prop:41}.
		
		(2) If $x\in G\setminus \Omega_\alpha$, then the above estimates show that
		\[
		\begin{split}
			(\widetilde \Mc^{(\lambda',1)}g_\alpha)(x)&\meg (\widetilde \Mc^{(\lambda',1)}f)(x)+ \sum_{j\in \N} (\widetilde \Mc^{(\lambda',1)}b_{\alpha,j})(x)\\
				&\meg  (\widetilde \Mc^{(\lambda',\kappa)}f)(x)+ C_2\alpha \sum_{j\in\N}\chi_{B(x_{\alpha,j},1+(3/2)\kappa_1 r_{\alpha,j})}(x) \frac{1}{(1+\abs{x^{-1}x_{\alpha,j}}_*/r_{\alpha,j})^a}
		\end{split}
		\]
		since $x\in G\setminus B(x_{\alpha,j},\kappa_2 r_{\alpha,j})$ for every $j\in\N$. Then, take $x\in \Omega_\alpha$, and take $j_x\in \N$ so that $x\in B(x_{\alpha,j_x},\kappa_1 r_{\alpha,j_x})$. Set
        \[
        J_x\coloneqq \Set{j\in \N\colon B(x_{\alpha,j_x},(\kappa_2/2) r_{\alpha,j_x})\cap B(x_{\alpha,j},(\kappa_2/2) r_{\alpha,j})\neq \emptyset},
        \]
        so that $\card(J_x)\meg N$ and
		\[
		\sum_{j\in \N\setminus J_x} (\widetilde \Mc^{(\lambda',1)}b_{\alpha,j})(x)\meg   C_2\alpha \sum_{j\in\N\setminus J_x}\chi_{B(x_{\alpha,j},1+(3/2)\kappa_1 r_{\alpha,j})}(x)  \frac{1}{(1+\abs{x^{-1}x_{\alpha,j}}_*/r_{\alpha,j})^a}
		\]
		by the previous estimates. Here we use the fact that
        \[
        d(x,x_{\alpha,j})\Meg d(x_{\alpha,j},x_{\alpha,j_x})-d(x_{\alpha,j_x},x)>\Big( \frac{\kappa_2}{2}-\kappa_1\Big) r_{\alpha,j_x}+ \frac{\kappa_2}{2}r_{\alpha,j}\Meg \frac{\kappa_2}{2}r_{\alpha,j}
        \]
        for every $j\in N\setminus J_x$.
        Consequently, it only remains  to estimate $\widetilde \Mc^{(\lambda',1)}(f-\sum_{j\in J_x} b_{\alpha,j})(x)$. Observe that   $1/c\meg r_{\alpha,j}/r_{\alpha,j'}\meg c$ for $j,j'$ such that $B(x_j,r_{\alpha,j})\cap B(x_{j'},r_{\alpha,j'})\neq \emptyset$, so that
        \[
            d(x,x_{\alpha,j})\meg d(x,x_{\alpha,j_x})+d(x_{\alpha,j_x},x_{\alpha,j})< \kappa_1 r_{\alpha,j_x}+\frac{\kappa_2}{2}(r_{\alpha,j_x}+r_{\alpha,j})  \meg \Big( c\kappa_1+\frac{\kappa}{2}(c+1)\Big)r_{\alpha,j}      
        \]
        and then 
        \[
		1 \meg   \frac{(1+c\kappa_1+\kappa_2(1+c)/2)^a}{(1+\abs{x^{-1}x_{\alpha,j}}_*/r_{\alpha,j})^a }
		\]
		for every $j\in J_x$. Consequently, it will suffice to show that  $\widetilde \Mc^{(\lambda',1)}(f-\sum_{j\in J_x} b_{\alpha,j})(x)$ is controlled by 
        \[
            \alpha \sum_{j\in J_x}\chi_{B(x_{\alpha,j},1+(3/2)\kappa_1 r_{\alpha,j})}(x) .
        \]
		Let $\zeta_{\alpha,x}\coloneqq 1-\sum_{j\in J_x } \zeta_{\alpha,j}$, so that $\zeta_{\alpha,x}=0$ on $B(x_{\alpha,j_x},(\kappa_2/2)r_{\alpha,j_x})$ and
        \[
        f-\sum_{j\in J_x} b_{\alpha,j}= \zeta_{\alpha,x} f+\sum_{j\in J_x} \zeta_{\alpha,j} P_{\alpha,j}.
        \]
		Thus, take $t\in (0,1]$ and $\phi\in \widetilde B_{\lambda',t}$ and observe that, if we take $j\in J_x$ and  $z_{j}\in B(x_{\alpha,j},\kappa_3 r_{\alpha,j})\setminus \Omega_\alpha$, then Corollary~\ref{cor:22}  shows that there is a constant $C_4>0$ such that
		\[
		\begin{split}			
			\abs*{ [(\zeta_{\alpha,j} P_{\alpha,j})*\phi](x)}&\meg \norm{ \zeta_{\alpha,j} P_{\alpha,j} }_{L^\infty(G)} t^{-Q_*}\beta(B(x_{\alpha,j}, (3/2) \kappa_1 r_{\alpha,j})\cap B(x,t))\\
				&\meg C_4\chi_{B(x_{\alpha,j},1+(3/2)\kappa_1 r_{\alpha,j})}(x)(\widetilde \Mc^{(\lambda',\kappa)} f)(z_{j}) \\
				&\meg C_4 \alpha\chi_{B(x_{\alpha,j},1+(3/2)\kappa_1 r_{\alpha,j})}(x).
		\end{split} 
		\]
		In addition, if $t\meg r_{\alpha,j_x}$, then
		\[
		((f\zeta_{\alpha,x})*\phi)(x)= \langle f, \zeta_{\alpha,x} \phi(\,\cdot\,^{-1}x)\rangle =0
		\]
		since $\zeta_{\alpha,x} \phi(\,\cdot\,^{-1}x)$ is supported in $G\setminus B(x_{\alpha,j_x}, (\kappa_2/2) r_{\alpha,j})\cap \overline B(x,t)$, which is empty. In fact, if $y \in \overline B(x,t)\setminus B(x_{\alpha,j_x}, (\kappa_2/2) r_{\alpha,j_x})$, then
		\[
		t\Meg \abs{y^{-1} x}_* \Meg \abs{y ^{-1}x_{\alpha,j_x}}_*- \abs{x_{\alpha,j_x}^{-1}x}_*\Meg  (\kappa_2/2)r_{\alpha,j_x}- \kappa_1 r_{\alpha,j_x} > (\kappa_1/2) r_{\alpha,j_x}>r_{\alpha,j_x},
		\]
		which is absurd. 
		Therefore, 
		\[
        \begin{split}
		\abs*{\Big(f-\sum_{j\in J_x} b_{\alpha,j'}\Big)*\phi)(x)}&
            \meg C_4 \alpha \sum_{j\in J_x}\chi_{B(x_{\alpha,j},1+(3/2)\kappa_1 r_{\alpha,j})}(x)
        \end{split}
		\]
		provided that $t\meg r_{\alpha,j_x}$.
		If, otherwise, $t>r_{\alpha,j_x}$, then observe that
		\[
		\bigg(\Big(f-\sum_{j\in J_x} b_{\alpha,j}\Big)*\phi\bigg)(x)= [f*\phi(\,\cdot\,z_{j_x}^{-1}x)](z_{j_x})+ \sum_{j\in J_x} (b_{\alpha,j}*\phi)(x).
		\]
		On the one hand, $\phi(\,\cdot\,z_{j_x}^{-1}x)$ is supported in $\overline B(e,t) x^{-1}z_{j_x} \subseteq \overline B(e, t+ \kappa_3 r_{\alpha,j_x})\subseteq B(e,\kappa t)$, so that clearly 
		\[
		\abs{[f*\phi(\,\cdot\,z_{j_x}^{-1}x)](z_{j_x})}\meg (\widetilde \Mc^{(\lambda',\kappa)} f)(z_{j_x})\meg   \alpha= \alpha \chi_{B(x_{\alpha,j_x}, 1+(3/2)\kappa_1 r_{\alpha,j_x})}(x).
		\]
		In addition, for every $j\in J_x$ one has
        \[
        d(z_{j_x},x_{\alpha,j})\meg d(z_{j_x}, x_{\alpha,j_x})+d(x_{\alpha,j_x},x_{\alpha,j})\meg \kappa_3 r_{\alpha,j_x}+(\kappa_2/2) (r_{\alpha,j_x}+r_{\alpha,j})\meg [(3c+1)/2]\kappa_3 r_{\alpha,j},
        \]
        so that the previous estimates show that
		\[
		\begin{split}
		 (b_{\alpha,j}*\phi)(x)&\meg (\widetilde  \Mc^{(\lambda',1)}b_{\alpha,j})(x)\\
		 	&\meg C_1\chi_{B(x_{\alpha,j}, 1+(3/2)\kappa_1 r_{\alpha,j})}(x) \frac{(\widetilde  \Mc^{(\lambda',\kappa)}f)(z_{j_x})}{(1+\abs{x^{-1}x_{\alpha,j}}_*/r_{\alpha,j})^a}\\
		 	&\meg C_1 \alpha \chi_{B(x_{\alpha,j}, 1+(3/2)\kappa_1 r_{\alpha,j})}(x).
		 \end{split}
		\]
		This completes the proof of the second assertion. 

	(3) By~(2), there is a constant $C_6>0$ such that
	\[
	(\widetilde \Mc^{(\lambda',1)} g_\alpha)(x)\meg C_6 \chi_{G\setminus \Omega_\alpha}(x)(\widetilde \Mc^{(\lambda',\kappa)} f)(x)+C_6\alpha\sum_{j\in \N} \frac{\chi_{B(x_{\alpha,j}, 1+(3/2)\kappa_1 r_{\alpha,j})}(x)}{(1+\abs{x^{-1}x_{\alpha,j}}_*/r_{\alpha,j})^a}
	\]
	for every $x\in G$. Observe that
	\[
	\int_{G\setminus \Omega_\alpha} (\widetilde \Mc^{(\lambda',\kappa)} f)(x)\,\dd \beta(x)\meg \alpha^{1-p} \int_{G\setminus \Omega_\alpha} (\widetilde \Mc^{(\lambda',\kappa)} f)(x)^p\,\dd \beta(x)\meg \alpha^{1-p} \norm{\widetilde \Mc^{(\lambda,\kappa)} f}_{L^p(G)}^p
	\]
	since $\widetilde \Mc^{(\lambda',\kappa)} f\meg \alpha$ on $G\setminus \Omega_\alpha$. In addition, using  Lemma~\ref{lem:37c}, we see that  there is a constant $C_7>0$ such that
	\[
	\begin{split}
		\int_G \sum_{j\in \N}  \chi_{B(x_{\alpha,j}, 1+(3/2)\kappa_1 r_{\alpha,j})}(x)\frac{1}{(1+\abs{x^{-1}x_{\alpha,j}}_*/r_{\alpha,j})^a}\,\dd \beta(x) &\meg  C_7\sum_{j\in \N} r_{\alpha,j}^{Q_*}(1+r_{\alpha_j})^{Q_G-Q_*}\\
			&\meg C_7^2\sum_{j\in \N} \beta(B(x_{\alpha,j},r_{\alpha,j}))\\
			&\meg C_7^2 \beta(\Omega_\alpha)\\
			&\meg \frac{C_7^2}{\alpha^p}\norm{\widetilde \Mc^{(\lambda',\kappa)} f}_{L^p(G)}^p,
	\end{split}
	\]
	whence the conclusion.
	\end{proof}

	\begin{cor}\label{cor:24}
		Take $p\in (0,1]$, $\lambda'\Meg \dd_M([\lambda/\dd_0]+1)$ for some $\lambda>Q_*(1/p-1)$ and a \BanachSpace\ $\Banach$. The intersection $h^1_{\lambda'}(G;\Banach)\cap h^p_{\lambda'}(G;\Banach)$ is dense in $h^p_{\lambda'}(G;\Banach)$.
	\end{cor}
	
	This replaces~\cite[Corollary 3.22]{FollandStein} in our setting.
	
	\begin{proof}
		Take $f\in h^p_{\lambda'}(G;\Banach)$. By Theorem~\ref{prop:50}, there are $\kappa\Meg 1$ and a constant $C>0$ such that for every $\alpha>\norm{ \widetilde \Mc^{(\lambda',\kappa)} f }_{L^p(G)}/\beta(G)^{1/p}$ we may find $g_\alpha\in  h^p_{\lambda'}(G;\Banach)$ such that $\widetilde \Mc^{(\lambda',\kappa)} g_\alpha\in L^1(G)$ and
		\[
		\norm{\widetilde \Mc^{(\lambda',1)}(f-g_\alpha)}_{L^p(G)}\meg C  \norm{ \chi_{\Omega_\alpha} \widetilde \Mc^{(\lambda',\kappa)} f }_{L^p(G)},
		\]
		where $\Omega_\alpha=\big\{x\in G\colon (\widetilde \Mc^{(\lambda',\kappa)} f)(x)>\alpha\big\}$. Letting $\alpha \to +\infty$, by dominated convergence we see that $g_\alpha$ converges to $f$ in $h^p_{\lambda'}(G;\Banach)$. Since $\widetilde \Mc^{(\lambda',1)} g_\alpha\in L^1(G)$, we obtain that $g_\alpha\in h^1_{\lambda'}(G;\Banach) $ for every $\alpha$, whence the conclusion.
	\end{proof}

    The following result provides some additional information on the previous Calder\'on--Zygmund decomposition when considering an integrable function (or, more precisely, with a measure whose total variation is absolutely continuous).
    
	\begin{teo}\label{lem:83}
		Keep the hypotheses and the notation of Theorem~\ref{prop:50}, but assume only that $f\in \Lp^1(G;\Banach)$. There is a constant $C>0$ such that
		\[
		\sum_{j\in \N} \norm{b_{\alpha,j}}_{\Lp^1(G;\Banach)}\meg  C \norm{ f}_{\Lp^1(G;\Banach)}
		\]
		and
		\[
		\norm{g_\alpha}_{\Lp^\infty(G;\Banach)}\meg C\alpha.
		\]
	\end{teo}
	
	This replaces~\cite[Theorems 3.18 and 3.20 (ii)]{FollandStein} in our setting.
	
	\begin{proof}
		Observe that, by Corollary~\ref{cor:22}, there is a constant $C_1>0$ such that
		\[
		\norm{P_{\alpha,j} \zeta_{\alpha,j}}_{L^\infty(G;\Banach)}\meg C_1 (\widetilde \Mc^{(\lambda',\kappa)} f)(x)
		\]
		for every $x\in B(x_{\alpha,j},\kappa_3 r_{\alpha,j})$, whence
		\[
		\norm{P_{\alpha,j} \zeta_{\alpha,j}}_{L^\infty(G;\Banach)}\meg C_1 \alpha
		\]
		taking $x\in B(x_{\alpha,j},\kappa_3 r_{\alpha,j})\setminus \Omega_\alpha$. In addition, by Lemma~\ref{lem:81} there is $N\in\N$ such that
		\[
		\sum_{j\in \N} \chi_{B(x_{\alpha,j},(3/2)\kappa_1)} \meg N\chi_{\Omega_\alpha}.
		\]
		Consequently,
		\[
		\begin{split}
			\sum_{j\in \N} \norm{b_{\alpha,j}}_{\Lp^1(G;\Banach)}&\meg \sum_{j\in \N} \int_G \abs{f} \zeta_{\alpha,j}\,\dd \beta+ \sum_{j\in \N} \int_G \abs{P_{\alpha,j}} \zeta_{\alpha,j}\,\dd\beta\\
				&\meg 	 \int_{\Omega_\alpha} \abs{f}  \,\dd \beta+ C_1 \alpha \sum_{j\in \N} 	 \beta(B(x_{\alpha,j},(3/2)\kappa_1 r_{\alpha,j}))\\
				&\meg  \int_{\Omega_\alpha} \abs{f}  \,\dd \beta+ C_1 N \alpha \beta(\Omega_\alpha).
		\end{split}
		\]
		Now, by Proposition~\ref{prop:42},  there is a constant $C_2>0$ such that
		\[
		\alpha \beta(\Omega_\alpha)\meg C_2 \norm{f}_{\Lp^1(G;\Banach)},
		\]
		whence the first assertion.
		
		Concerning the second assertion, observe that
		\[
		\begin{split}
		\abs{g_\alpha }&=\abs*{f -\sum_{j\in \N} b_{\alpha,j}} \meg \chi_{G\setminus \Omega_\alpha}\abs{f}+\sum_{j\in \N} \abs{P_{\alpha,j} \zeta_{\alpha,j}}  \meg \chi_{G\setminus \Omega_\alpha}\abs{f}+ C_1 N \alpha 
		\end{split}
		\]
		since $\sum_{j\in \N} \zeta_{\alpha,j}=1-\chi_{\Omega_\alpha}=\chi_{G\setminus \Omega_\alpha}$. Then, by Proposition~\ref{prop:42}, observe that there is a constant $C_3>0$ such that 
		\[
		\abs{f(x) }\meg C_3 (\widetilde \Mc^{(\lambda',\kappa)} f)(x) \meg  C_3\alpha
		\]
		for $\beta$-almost every $x\in G\setminus \Omega_\alpha$. The second assertion follows.
	\end{proof}

    \subsection{Calder\'on--Zygmund operators}
    \begin{deff}
		Take two \BanachSpace s  $\Banach_1, \Banach_2 $,  $\lambda, a\Meg 0$, and $q\in(1,\infty]$. We say that a distribution $K\in \Sr'(G;\Lf(\Banach_1;\Banach_2))$ is a $(\lambda,a,q,\Banach_1,\Banach_2)$-standard Calder\'on--Zygmund kernel, or simply a $(\lambda,a,q,\Banach_1,\Banach_2)$-standard kernel, if the following conditions are satisfied:
		\begin{itemize}
			\item the restriction of $X^R K$ to $G\setminus \Set{e}$ belongs to $\Lp^1_\loc(G\setminus \Set{e}; \Lf(\Banach_1;\Banach_2))$ for every $X\in U(\gf)$ with order $\meg q(\lambda)$;
			
			\item for $\beta$-almost every $x\in B(e,1)\setminus \Set{e}$ and for every $X\in U(\gf)$ with order $\meg q(\lambda)$,
			\[
			\abs{ X^R K }(x)\meg \frac{\abs{X}}{\abs{x}_*^{\deg(X)+Q_*}  };
			\]
			
			\item for $\beta$-almost every $x\in G\setminus B(e,1)$ and for every $X\in U(\gf)$ with order $\meg q(\lambda)$,
			\[
			\abs{  X^R K }(x)\meg \frac{\abs{X}}{ \abs{x}_* ^{a}  };
			\]
			
			\item for every $f\in \Sr(G;\Banach_1)$,
			\[
			\norm{f*K}_{L^q(G;\Banach_2)}\meg   \norm{f}_{L^q(G;\Banach_1)}.
			\]
		\end{itemize} 
	\end{deff}

    The next result is useful for proving that Hardy spaces are preserved under the action of Calder\'on--Zygmund operators. Its proof is postponed to Section~\ref{sec:app}.
	\begin{prop}\label{prop:31}
		Take two \BanachSpace s $\Banach_1,\Banach_2$, $p\in (0,1]$, $q\in (1,\infty]$, $a,a'>D/p$, $\lambda\Meg Q_*(1/p-1)$ and $\lambda'\Meg \dd([a/\dd]+1)$. Let $\rho>0$ be so that $\overline B(e,\rho)\subseteq \mathbf U$. There exists a constant $C>0$ such that  
		\[
		\norm{\Mc^{(\lambda',a)} (f*K)}_{L^p(G)}\meg C   
		\]
		for every $\Banach_1$-valued $(p,q,\rho,\lambda)$-atom $f$, and for every $(\lambda,a',q, \Banach_1,\Banach_2)$-standard kernel $K$.
	\end{prop}

	\begin{cor}\label{cor:26}
		Take a \BanachSpace\ $\Banach$, $p\in (0,1]$, $q\in (1,\infty]$,   $a>D/p$, $\lambda \Meg Q_*(1/p-1)$, and $\lambda'\Meg \dd([a/\dd]+1)$. Let $\rho>0$ be so that $\overline B(e,\rho)\subseteq \mathbf U$.  There exists a constant $C>0$ such that
        \[
        \norm{\Mc^{(\lambda',a)}f}_{L^p(G)}\meg C
        \]
        for every $\Banach$-valued $(p,q,\rho, \lambda)$-atom  $f$.
	\end{cor}
	
	\begin{proof}
		Simply apply Proposition~\ref{prop:31} with $K\in \Sr'(G;\Lf(\Banach;\Banach))$ defined so that $\langle K,\phi\rangle= \phi(e) I$ for every $\phi\in \Sr(G)$.
	\end{proof}

  \subsection{Atomic decomposition}
  We are now able to provide, under suitable assumptions, the atomic decomposition of the elements of $h^p_{\lambda'}(G;Z)$.

	\begin{teo}\label{teo:18} 
		Take $p\in (0,1]$, $q\in (1,\infty]$, $\lambda > Q_*(1/p-1)$, and a \BanachSpace\ $\Banach$.  Assume that either $\lambda \meg \dd_0$, in which case we set $\mi\coloneqq \lambda$, or that $\lambda>\dd_0$ and $G$ is  $\ell$-step nilpotent for some $\ell\in\N$, in which case we set $\mi\coloneqq \dd_M([\lambda/\dd_0]\min(1,\ell-1))$ and let (in both cases) $\lambda' \Meg \dd_M([\mi/\dd_0]+1)$.
		
		There are constants $\rho, C>0$ such that for every $f\in h^p_{\lambda'}(G;\Banach)$ there are a sequence $(f_j)_{j\in\N}$ of $(p,q,\rho,\lambda)$-atoms and $(a_j)\in \ell^p(\N)$ such that
        \[
        f=\sum_{j\in \N} a_j f_j
        \]
        with convergence in $h^p_{\lambda'}(G;\Banach)$, and
		\[
		\sum_{j\in \N} a_j^p \meg C \norm{f}_{h^p_{\lambda'}(G;\Banach)}^p.
		\]
	\end{teo}
 
	\begin{proof}
		\textsc{Step I} Observe first that we may reduce to the case in which $f\in h^1_{\lambda'}(G;\Banach)\cap h^p_{\lambda'}(G;\Banach)$. Indeed, by Corollary~\ref{cor:24}, we may find a sequence $(g_j)_{j\in \N}$ of elements of $h^1_{\lambda'}(G;\Banach)\cap h^p_{\lambda'}(G;\Banach)$ which converges to $f$ in $h^p_{\lambda'}(G;\Banach)$. We may also assume that $g_0=0$ and that 
        \[
        \norm{g_{j+1}-g_j}_{h^p_{\lambda'}(G;\Banach)}\meg 2^{-j/p} \norm{f}_{h^p_{\lambda'}(G;\Banach)}
        \]
        for every $j\in \N$ (if $f=0$, one may simply take $g_j=0$ for every $j\in\N$) so that
        \[
        \sum_{j\in \N} \norm{g_{j+1}-g_j}_{h^p_{\lambda'}(G;\Banach)}^p\meg 2 \norm{f}_{h^p_{\lambda'}(G;\Banach)}^p
        \]
        and that the assertion follows decomposing each $g_{j+1}-g_j$ into $(p,q,\rho,\lambda)$-atoms.
		
		Moreover, notice that by H\"older's inequality it is sufficient to prove the assertion for $q=\infty$.
		
		\textsc{Step II} Take   $\kappa_1,\kappa_2,\kappa_3,\kappa$ as in Theorem~\ref{prop:50}, but subject to the additional condition  $\kappa_2(\kappa_2-(3/2)\kappa_1)\Meg (3/2) \kappa_1(3\kappa_3+(3/2) \kappa_1)$. Fix $\rho'>0$ so that $\overline B(e,\rho')\subseteq \mathbf U$.  
		Let
        \[
        K=\Bigl\{k\in \Z\colon 2^k> \norm{\widetilde \Mc^{(\lambda',\kappa)}f}_{L^p(G)}/\beta(G)^{1/p}\Bigr\}
        \]
        so that $K=\Z$ unless $G$ is compact. For every $k\in K$, take $\Omega_k$, $(x_{k,j},r_{k,j})_{j\in \N}$, $\zeta_{k,j}$, $b_{k,j}$, $P_{k,j}$, and $g_k$ as in Theorem~\ref{prop:50}, with the choice $\alpha=2^k$, $\rho=\rho'$,  and $\lambda=\mi$.
		Given $k\in K$ and $j'\in\N$, define 
        \[
        J_{k,j'}\coloneqq \Set{j\in \N\colon B(x_{k+1,j'},(3/2)\kappa_1 r_{k+1,j'})\cap B(x_{k,j},(3/2)\kappa_1r_{k,j})\neq \emptyset}.
        \]
        Observe that, if $j\in J_{k,j'}$, then $d(x_{k+1,j'},x_{k,j})\meg (3/2)\kappa_1 (r_{k+1,j'}+r_{k,j})$ and, since $\Omega_{k+1}\subseteq \Omega_k$,
		\[
		\begin{split}
		\kappa_2 r_{k+1,j'}&\meg d(x_{k+1,j'},G\setminus \Omega_{k+1})\\
			&\meg d(x_{k+1,j'},G\setminus \Omega_k)\\
			&\meg d(x_{k+1,j'},x_{k,j})+ d(x_{k,j},G\setminus \Omega_k)\\
			&\meg (3/2)\kappa_1 (r_{k+1,j'}+r_{k,j})+\kappa_3 r_{k,j},
		\end{split}	
		\]
		so that
		\[
		r_{k+1,j'}\meg c' r_{k,j} 
		\]
		with $c'\coloneqq \frac{\kappa_3+(3/2)\kappa_1}{\kappa_2-(3/2)\kappa_1}$. In particular, $B(x_{k+1,j'},(3/2)\kappa_1 r_{k+1,j'})\subseteq B(x_{k,j},\kappa_2 r_{k,j})$ by the additional condition on $\kappa_2$. Consequently, if we take $N$ as in Lemma~\ref{lem:81}, then $\card(J_{k,j'})\meg N$ for every $j'\in \N$ and for every $k\in K$.

		\textsc{Step III} Take $k\in K$ and $j\in \N$. For every $j'\in \N$, define $P_{k,j,j'}\in \Pc^{(x_{k+1,j'})}_{\mi^-}\otimes \Banach$ as follows. Set $P_{k,j,j'}\coloneqq0 $ when $j\not \in J_{k,j'}$ or $\kappa_2 r_{k+1,j'}>\rho'$. If, otherwise, $j\in J_{k,j'}$ and $\kappa_2 r_{k+1,j'}\meg \rho'$,  then take $P_{k,j,j'}$ as the unique element of $\Pc^{(x_{k+1,j'})}_{\mi^-}\otimes \Banach$ such that
		\[
		\langle (f-P_{k+1,j'})\zeta_{k,j}, Q \zeta_{k+1,j'}\rangle= \langle P_{k,j,j'}, Q \zeta_{k+1,j'}\rangle
		\]
		for every $Q\in \Pc^{(x_{k+1,j'})}_{\mi^-}$ (see~Corollary~\ref{cor:22} for the existence). 
		By Corollary~\ref{cor:22}, Lemma~\ref{lem:82}, and Corollary~\ref{cor:23}, we then see that there is a constant $C_1>0$ such that 
		\[
		\begin{split}
        \norm{P_{k,j,j'}\zeta_{k+1,j'}}_{L^\infty(G;\Banach)}&\meg C_1  \widetilde \Mc^{(\lambda',\kappa)} ((f-P_{k+1,j'})\zeta_{k+1,j'}\zeta_{k,j})(x )\\
				&\meg C_1^2 \widetilde \Mc^{(\lambda',\kappa)} ((f-P_{k+1,j'} )\zeta_{k+1,j'})(x )\\
				&\meg C_1^3 ( \widetilde \Mc^{(\lambda',\kappa)} f)(x)\\
				&\meg C_1^3 2^{k+1}
		\end{split}
		\]
		for every $x\in B(x_{k+1,j'}, \kappa_3 r_{k+1,j'})\setminus \Omega_{k+1}$ (since $\kappa\Meg 2\kappa_3\Meg \kappa_2+\kappa_3$).
		
		\textsc{Step IV} Let us now prove that, for every $k\in K$,   
		\[
		\sum_{j,j'\in\N}   P_{k,j,j'}\zeta_{k+1,j'}=0
		\]
		pointwise and in $\Sr'(G;\Banach)$. Observe first that, for every $x\in G$, there are at most $N$ elements $j'\in \N$ such that $\zeta_{k+1,j'}(x)\neq 0$, and that for every such $j'$ there at most $N$ elements $j\in\N$ such that $P_{k,j,j'}\neq 0$, by~\textsc{step III}. Consequently, the sum is pointwise finite. Furthermore, again by~\textsc{step III},
		\[
		\sum_{j,j'\in \N} \abs{P_{k,j,j'}(x)\zeta_{k+1,j'}(x)}\meg N^2 C_1^3 2^k,
		\]
		so that the sum converges also in $\Sr'(G;\Banach)$. Therefore, 
        \[
        \sum_{j,j'} P_{k,j,j'}\zeta_{k+1,j'}=\sum_{j'} \bigg(\sum_{j\in J_{k,j'}} P_{k,j,j'}\bigg)\zeta_{k+1,j'}.
        \]
        Note that, for every $Q\in \Pc^{(x_{k+1,j'})}_{\mi^-}$,
		\[
		\begin{split}
				\bigg\langle  \sum_{j\in J_{k,j'}} P_{k,j,j'}, Q\zeta_{k+1,j'}\bigg\rangle&= \sum_{j\in J_{k,j'}} \langle P_{k,j,j'} , Q\zeta_{k+1,j'} \rangle\\
					&=\sum_{j\in J_{k,j'}}\langle (f-P_{k+1,j'})\zeta_{k,j}, Q \zeta_{k+1,j'}\rangle\\
					&=\bigg\langle (f-P_{k+1,j'})   \sum_{j\in J_{k,j'}} \zeta_{k,j}, Q \zeta_{k+1,j'}\bigg\rangle\\
					&=\langle f-P_{k+1,j'}, Q \zeta_{k+1,j'}\rangle\\
					&=0
		\end{split}
		\]
		since $\sum_{j\in J_{k,j'}} \zeta_{k,j}=1$ on the support of $\zeta_{k+1,j'}$.  Consequently, $\sum_{j\in J_{k,j'}} P_{k,j,j'}=0$. 
		
		\textsc{Step V} Observe that, by Theorem~\ref{prop:50}, there is a constant $C_2>0$ such that
		\[
		\norm{f-g_k}_{h^p_{\lambda'}(G;\Banach)}^p\meg \sum_{j\in \N} \norm{b_{k,j}}_{h^p_{\lambda'}(G;\Banach)}^p \meg C_2 \norm{\chi_{\Omega_k} \widetilde \Mc^{(\lambda',\kappa)}f}_{L^p(G)}^p,
		\]
		so that $g_k\to f$ in $h^p_{\lambda'}(G;\Banach)$ for $k\to +\infty$. In addition, if $G$ is not compact, then by Theorem~\ref{lem:83} we see that $g_k \to 0 $ in $\Lp^\infty(G;\Banach)$ for $k\to -\infty$. If, otherwise, $G$ is compact and $k_0\coloneqq \min K$, then $2^{k_0-1}\meg \norm{\widetilde \Mc^{(\lambda',\kappa)} f}_{L^p(G;\Banach)}/\beta(G)^{1/p}$, so that by means of Theorem~\ref{lem:83} we see that we may find a constant $C_3>0$ such that
		\[
		\norm{g_{k_0}}_{\Lp^\infty(G;\Banach)}\meg C_3 2^{k_0}\meg 2C_3\frac{\norm{\widetilde \Mc^{(\lambda',\kappa)} f}_{L^p(G;\Banach)}}{\beta(G)^{1/p}},
		\]
		so that $  g_{k_0}= 2C_3 \norm{\widetilde \Mc^{(\lambda',\kappa)} f}_{L^p_{\lambda'}(G;\Banach)} \widetilde f$ for some $(p,\infty,\rho,\lambda)$-atom $\widetilde f$, where $\rho\coloneqq \rho'/c'$. Up to replacing $f$ with $f-g_{k_0}$, we may therefore assume that $f=\sum_{k\in K} (g_{k+1}-g_k)$ in $\Sr'(G;\Banach)$. Now, observe that
		\[
		\begin{split}
			g_{k+1}-g_k&= f-\sum_{j\in\N} b_{k+1,j}- f+\sum_{j\in \N} b_{k,j}\\
				&=\sum_{j\in \N}b_{k,j}-\sum_{j\in \N } b_{k+1,j}\\
				&=\sum_{j\in \N}b_{k,j}-\sum_{j,j'\in \N } b_{k+1,j'} \zeta_{k,j}+\sum_{j,j'\in\N} P_{k,j,j'}\zeta_{k+1,j'}\\
				&=\sum_{j\in \N} h_{k,j},
		\end{split}
		\]
		where
		\[
		h_{k,j}\coloneqq b_{k,j}- \sum_{j'\in \N } (b_{k+1,j'} \zeta_{k,j}- P_{k,j,j'}\zeta_{k+1,j'})
		\]
		for every $k\in K$ and for every $j\in\N$,
		since 
        \[
        \sum_{ j\in\N }\zeta_{k,j}=1 \quad \text{on} \quad \Omega_k\supseteq \Omega_{k+1} \supseteq B(x_{k+1,j'}, (3/2)\kappa_1 r_{k+1,j'})
        \]
        for every $j'\in\N$.
		Observe that the final sum converges (at least) in $\Sr'(G;\Banach)$, and that $h_{k,j}$ is concentrated in $B(x_{k,j}, \kappa_2 r_{k,j})$ for every $k\in K$ and for every $j\in\N$, thanks to~\textsc{step II}.
		Next, observe that
		\[
		\begin{split}
		h_{k,j}&= f \zeta_{k,j}- P_{k,j}\zeta_{k,j}- \sum_{j'\in \N}(f \zeta_{k+1,j'} \zeta_{k,j} - P_{k+1,j'}\zeta_{k+1,j'} \zeta_{k,j}- P_{k,j,j'} \zeta_{k+1,j'})\\
			&= f\chi_{G\setminus \Omega_{k+1}} \zeta_{k,j} -P_{k,j}\zeta_{k,j}-\sum_{j'\in \N}( P_{k+1,j'}\zeta_{k+1,j'} \zeta_{k,j}- P_{k,j,j'} \zeta_{k+1,j'})
		\end{split}
		\]
		and that, by Corollary~\ref{cor:22} and~\textsc{step III}, we see that there is a constant $C_4>0$ such that
		\[
		\norm{h_{k,j}-f\chi_{G\setminus \Omega_{k+1}} \zeta_{k,j}}_{L^\infty(G;\Banach)}\meg C_4 2^{k+1}.
		\]
		In addition, by Proposition~\ref{prop:42} there is a constant $C_5>0$ such that
		\[
		\abs{f}(x)\meg C_5 (\widetilde \Mc^{(\lambda',\kappa)} f)(x)\meg C_5 2^{k+1}
		\]
		for $\beta$-almost every $x\in G\setminus \Omega_{k+1}$, so that
		\[
		\norm{h_{k,j}}_{\Lp^\infty(G;\Banach)}\meg (C_4+C_5)2^{k+1}
		\]
		for every $k\in K$ and for every $j\in\N$. 
		
		Now, let us prove that, if $\kappa_2 r_{k,j}<\rho$, then $\langle h_{k,j}, Q\rangle=0$ for every $Q\in \Pc^{(x_{k,j})}_{\lambda^-}$. Observe that, since
        \[
        \langle f, Q \zeta_{k,j}\rangle=\langle P_{k,j},Q\zeta_{k,j}\rangle
        \]
        for every $Q\in \Pc^{(x_{k,j})}_{\mi^-}\supseteq \Pc^{(x_{k,j})}_{\lambda^-}$, it will suffice to show that
        \[
        \langle (f-P_{k+1,j'})\zeta_{k,j} , Q\zeta_{k+1,j'}\rangle =\langle P_{k,j,j'},Q\zeta_{k+1,j'}\rangle
        \]
        for every $Q\in \Pc^{(x_{k,j})}_{\mi^-}$. 
        If  $\lambda \meg \dd_0$, then we have $\Pc^{(x)}_{\lambda^-}=\C \chi_G$ for every $x\in G$, so that  it suffices to show that $\langle h_{k,j},1\rangle=0$. Then, notice that 
		\[
		\langle h_{k,j}, 1\rangle=\langle f , \zeta_{k,j}\rangle-\langle P_{k,j},\zeta_{k,j}\rangle-\sum_{j'\in \N}(\langle (f-P_{k+1,j'})\zeta_{k,j} , \zeta_{k+1,j'}\rangle -\langle P_{k,j,j'},\zeta_{k+1,j'}\rangle)=0
		\]
		by the definition of $P_{k,j,j'}$.
		 
		If, otherwise, $G$ is nilpotent and $\lambda > \dd_0$, then by the Baker--Campbell--Hausdorff formula, the left translations on $G$ are polynomials of degree at most $\min(1,\ell-1)$ in the Euclidean sense (in exponential coordinates). It follows that, if $Q\in \Pc^{(x_{k,j})}_{\lambda^-}$, then there is   $Q'\in \Pc^{(x_{k+1,j'})}_{\mi^{-}}$ such that  $Q=Q'$ on 
        \[
        x_{k,j}\mathbf U\cap x_{k+1,j'}\mathbf U\supseteq B(x_{k,j}, \kappa_2 r_{k,j})\cap B(x_{k+1,j'}, \kappa_2 r_{k+1,j'}),
        \]
        noticing that, if $\kappa_2 r_{j,k}< \rho$, then $\kappa_2r_{k+1,j'}\meg c' \kappa_2 r_{k+1,j'}< \rho $ by~\textsc{step II}. Therefore,
		\[
		\langle (f-P_{k+1,j'})\zeta_{k,j} , Q\zeta_{k+1,j'}\rangle= \langle (f-P_{k+1,j'})\zeta_{k,j} , Q'\zeta_{k+1,j'}\rangle =\langle P_{k,j,j'},Q'\zeta_{k+1,j'}\rangle=\langle P_{k,j,j'},Q\zeta_{k+1,j'}\rangle 
		\]
		by the definition of $P_{k,j,j'}$, whence our assertion also in this case.
	
		Thus, if we set $a_{k,j}\coloneqq (C_4+C_5)2^{k+1} \beta(B(e, \kappa_2 r_{k,j}))^{1/p}$ and $f_{k,j}\coloneqq h_{k,j}/a_{k,j}$, then $f_{k,j}$ is a $(p,\infty,\rho,\lambda)$-atom for every $k\in K$ and $j\in \N$, 
		\[
		f=\sum_{k\in K, j\in\N} a_{k,j} f_{k,j}
		\]
		in $\Sr'(G;\Banach)$, and
		\[
		\begin{split}
		\sum_{k\in K, j\in\N} a_{k,j}^p&\meg (C_4+C_5)^p \sum_{k\in K, j \in\N} 2^{(k+1)p}\beta(B(x_{k,j},\kappa_2 r_{k,j}))\\
			&\meg (C_4+C_5)^p N \sum_{k\in K} 2^{(k+1)p}\beta(\Omega_k)\\
			&\meg (C_4+C_5)^p N \frac{4^p}{\log 2} \int_0^{+\infty} \alpha^p \beta\Big( \Set{x\in G\colon (\widetilde \Mc^{(\lambda',\kappa)} f)(x)>\alpha}\Big) \frac{\dd \alpha}{\alpha}\\
			&= (C_4+C_5)^p N \frac{4^p}{p\log 2} \norm{\widetilde \Mc^{(\lambda',\kappa)} f}_{L^p(G)}^p.
		\end{split}
		\]
		Consequently,
        \[
        f=\sum_{k\in K, j\in \N} a_{k,j} f_{k,j}
        \]
        in $h^p_{\lambda'}(G;\Banach)$ by Proposition~\ref{prop:41} and  Corollary~\ref{cor:26}, whence the result.
	\end{proof}
	
	\begin{cor}\label{cor:27}
		Take $p\in (0,1]$, a \BanachSpace\ $\Banach$ such that $\Lp^1_\loc(G;\Banach)=L^1_\loc(G;\Banach)\cdot \beta$, $\lambda>Q_*(1/p-1)$ so that either $\lambda \meg \dd_0$ or $G$ is nilpotent. Define $\mi$ and take $ \lambda'$ as in Theorem~\ref{teo:18}. Then, $C^\infty_c(G;\Banach)$ is dense in $h^p_{\lambda'}(G;\Banach)$.
	\end{cor}
	
	The assumption on $\Banach$ is satisfied when $\Banach$ has the Radon--Nikod\'ym property, in particular when $\Banach$ is reflexive.
	
	\begin{proof}
		Take $\rho,\lambda$ as in Theorem~\ref{teo:18}.
		By Theorem~\ref{teo:18}, we may reduce to proving that each $(p,\infty,\rho,\lambda)$-atom $f$ may be approximated in $h^p_{\lambda'}(G;\Banach)$ by a sequence of elements of $C^\infty_c(G;\Banach)$. Take $x\in G$ and $r>0$ so that $f$ is supported in $B(x,r)$. Observe that, by the assumption on $\Banach$, we have $f\in L^\infty(G;\Banach)$. Then, there is a sequence $(f_j)$ of elements of $C^\infty_c(G;\Banach)$, supported in $B(x,r)$, which converges to $f$ in (say) $L^2(G;\Banach)$. Since, by Proposition~\ref{prop:25},  $\chi_{B(x,r)}\cdot L^2(G;\Banach)\subseteq h^p_{\lambda'}(G;\Banach)$ continuously, the assertion follows.
	\end{proof}

	\begin{cor}\label{cor:28}
		Take $p\in (0,1]$, $\kappa>0$, $a>D/p$, $\lambda>Q_*(1/p-1)$  such that either $\lambda \meg \dd_0$ or $G$ is nilpotent, and define $\mi$ as in Theorem~\ref{teo:18}. Take $\lambda',\lambda''\Meg \dd_M([\mi/\dd_0]+1), \dd([a/\dd]+1)$. There exists a constant $C>1$ such that
		\[
		\frac 1 C\norm{\widetilde \Mc^{(\lambda',\kappa)}f}_{L^p(G)}\meg   \norm{\Mc^{(\lambda',a)}f}_{L^p(G)}\meg C  \norm{\widetilde \Mc^{(\lambda'',\kappa)}f}_{L^p(G)}
		\] 
		for every \BanachSpace\ $\Banach$ and for every $f\in \Sr'(G;\Banach)$.
	\end{cor}
	
	\begin{proof}
		The first inequality follows from Lemmas~\ref{lem:85} and~\ref{lem:71}, while the second one follows from Theorem~\ref{teo:18} and Corollary~\ref{cor:26}. 
	\end{proof}
	
Now notice that if we take $p\in (0,1]$ and assume that either $p>Q_*/(Q_*+\dd_0)$ or $G$ is nilpotent, then Corollary~\ref{cor:28} shows that $h_{\lambda'}^p(G;\Banach)=h^p_{\lambda''}(G;\Banach)$ for every  $\lambda',\lambda''\Meg \dd_M([\mi/\dd_0]+1), \dd([a/\dd]+1)$ (for some $a>D/p$) and for every \BanachSpace\ $\Banach$. Therefore, the following definition is well-posed.
	\begin{deff}
		 Let $p\in (0,1]$ be such that either $p>Q_*/(Q_*+\dd_0)$ or $G$ is nilpotent. We define $h^p(G;\Banach)$ as $h^p_{\lambda'}(G;\Banach)$ for any $\lambda'\Meg \dd_M([\mi/\dd_0]+1), \dd([a/\dd]+1)$, with $a$ and $\mi$ as in Corollary~\ref{cor:28}. 
	\end{deff}

\section{Identification with the Triebel--Lizorkin space}\label{Section:Identification}

In this section, we provide an alternative characterization of the local Hardy space introduced above by identifying it with the Triebel--Lizorkin space (see Subsection~\ref{subsec:TLspaces}). This characterization is motivated by the classical theory, where local Hardy spaces admit equivalent descriptions in terms of Triebel--Lizorkin spaces.

    The following result combines Proposition~\ref{prop:31} with the atomic decomposition established in Theorem~\ref{teo:18} to show that Calder\'on–Zygmund operators remain bounded on Hardy spaces in the present setting.

	\begin{teo}\label{teo:22bis}
		Take two \BanachSpace s $\Banach_1,\Banach_2$, $p \in (0,\infty)$, $q\in (1,\infty]$, $\lambda>Q_*(1/p-1)_+, D/\min(1,p)-Q_*$, $a>D/\min(1,p)$. Assume that either $\lambda \meg \dd_0$ or $G$ is nilpotent.
		Then,  there is a constant $C>0$ such that
		\[
		\norm{f*K}_{h^p(G;\Banach_2)}\meg C   \norm{f}_{h^p(G;\Banach_1)}
		\]
		for every $f\in \Sr(G;\Banach_1)$ and for every $(\lambda,a,q, \Banach_1,\Banach_2)$-standard kernel $K$.
	\end{teo}

	\begin{proof}
		We start with the case $p\in(0,1]$. Take $(\tau_t)_{t>0}$ as in Corollary~\ref{cor:21} (with the choice $\kappa_1=1$ and $\kappa_2=2$), and let us prove that there is a constant $c>0$ such that $c\tau_t K$ is a $(\lambda,a,q, \Banach_1,\Banach_2)$-standard kernel for every $t\Meg 1$. On the one hand, it is clear that there is a constant $C_1>0$ such that
		\[
		\abs{X(\tau_t K)}(x)\meg \abs{X K}(x)+C_1\chi_{B(e,2t)\setminus B(e,t)}(x)\sup_{Y\in U_\lambda, \abs{Y}\meg 1} t^{\deg(Y)-\deg(X)}\abs{YK}(x)
		\]
		for every $X\in U_\lambda$ with $\abs{X}\meg 1$, whence the required estimates. In addition, since $\norm{(1-\tau_t) K}_{\Mc^1(G;\Lf(\Banach_1;\Banach_2))}$ is bounded for $t\Meg 1$, by Young's inequality there is a constant $C_2>0$ such that
        \[
        \norm{f*[(1-\tau_t) K]}_{L^q(G;\Banach_2)}\meg C_2 \norm{f}_{L^q(G;\Banach_1)}
        \]
        for every $f\in \Sr(G;\Banach_1)$ and for every $t\Meg 1$. Therefore, by Proposition~\ref{prop:31} and Theorem~\ref{teo:18} we see that there is a constant $C_3>0$ such that
		\[
		\norm{f*(\tau_t K)}_{h^p(G;\Banach_2)}\meg C_3   \norm{f}_{h^p(G;\Banach_1)}
		\]
		 for every $f\in \Sr(G;\Banach_1)$. Notice that we used the fact that $\tau_t K$ is compactly supported in order to ensure that $f*(\tau_t K)=\sum_j \lambda_j [f_j *(\tau_t K)]$ in $\Sr'(G;\Banach_2)$ for every atomic decomposition $f=\sum_j \lambda_j f_j$.
         Since $f*(\tau_t K)$ converges to $f*K$ in $\Sr'(G;\Banach_2)$ for $t\to +\infty$, by the lower semicontinuity of the quasi-norm of $ h^p(G;\Banach_2)$ for the topology of $\Sr'(G;\Banach_2)$ (see~the proof of Proposition~\ref{prop:41}) the assertion follows in this case.

		 It remains to consider the case $p\in (1,\infty)$. 
		Take $(\psi_j)$ as in the proof of Proposition~\ref{prop:31} and observe that the proof of Proposition~\ref{prop:31} shows that there is a constant $c'>0$ such that $c' K*\psi_j$ is a $(\lambda,a,q, \Banach_1,\Banach_2)$-standard kernel for every  $j\in\N$. Since
        \[
        f*(K*\psi_j)\to f*K\text{ in }\Sr'(G;\Banach_2),
        \]
        for $j\to \infty$ and for every $f\in \Sr(G;\Banach_1)$
        we may reduce to the case in which $K$ is in $C^\infty$.
		 
		We then take $\eta\in C^\infty_c(G)$ so that $\chi_{B(e,1)}\meg \eta \meg \chi_{B(e,2)}$, we split $K$ as the sum of $\eta K$ and $(1-\eta)K$ and we check the usual Calder\'on--Zygmund conditions  for $\eta K$ (see, e.g.,~\cite[Corollary 17]{Singular}). Observe first that, since $(1-\eta)K\in L^1(G;\Lf(\Banach_1;\Banach_2))$ by Lemma~\ref{lem:37}, by Young's inequality there is a constant $C_1>0$ such that
        \[
        \norm{f*[(1-\eta )K]}_{L^p(G;\Banach_2)}\meg C_1\norm{f}_{L^p(G;\Banach_1)}
        \]
        for every $p\in [1,\infty]$. In particular
        \[
        \norm{f*(\eta K)}_{L^q(G;\Banach_2)}\meg (1+C_1)\norm{f}_{L^q(G;\Banach_1)}.
        \]
        Then, observe that the estimates on the right-invariant derivatives of $K$ (combined with Lemma~\ref{lem:72} to obtain analogous estimates for the left-invariant derivatives of $\eta K$) and~\cite[Lemma 7.5]{Calzi2}, and Lemmas~\ref{lem:37}, and~\ref{lem:37b} show that there is a constant $C_2>0$ such that
		\[
		\int_{G\setminus B(x,2 \abs{x}_*)} \abs{(\eta K)(y)-(\eta K)(y x)}\,\dd \beta(y), \int_{G\setminus B(x,2 \abs{x}_*)} \abs{(\eta K)(y)-(\eta K)(x y)}\,\dd \beta(y)\meg C_2
		\]
		for every $x\in G\setminus \Set{e}$. By the Calder\'on--Zygmund theory on (locally) doubling spaces, the assertion follows.
	\end{proof}

    We may now apply the previous result to provide the identification of the local Hardy spaces $h^p(G;H)$ as Triebel--Lizorkin spaces, when $H$ is a Hilbert space. 
    
	\begin{teo}
		Take $p\in (0,\infty)$ and   assume that either $p>Q_*/(Q_*+\dd_0)$ or $G$ is nilpotent. Then,
	\[ 
	h^p(G;H)=F^{p,2}_0(G;H).
	\]
	for every Hilbert space $H$.
	\end{teo}
	
	\begin{proof}
		When $p\in(1,\infty)$, the assertion follows from~\cite[Theorem 8.2]{BCP} and Proposition~\ref{prop:42}.\footnote{The former result has to be extended to the case $H\neq \C$. When $p=2$, the assertion follows trivially by tensorization. The general case then follows using the same arguments.}  Then, assume that $p\in(0,1]$. Take $(\psi_j),(\psi'_j)\in \Psi_{2^{-\grado}}(\Lc)$, in such a way that
        \[
        \sum_{j\in\N} \psi_j=\delta_e\text{ and }\psi_j=\psi_j*\psi'_j\text{ for every }j\in \N.
        \]
        Let $K\in \Sr'(G;\Lf(\C;\ell^2(\N)))\cong \Sr'(G;\ell^2(\N))$ and $K'\in \Sr'(G;\Lf(\ell^2(\N);\C))\cong \Sr'(G;\ell^2(\N)) $ be so that
        \[
        \langle K,\phi\rangle \coloneqq (\langle \psi_j, \phi\rangle)_{j\in\N},\,\,\langle K',\phi\rangle \coloneqq (\langle \psi'_j, \phi\rangle)_{j\in \N}
        \]
        for every $\phi\in \Sr(G)$. Observe first that $K$ and $K'$ are well defined: by Proposition~\ref{cor:12} there is a constant $C_1>0$ such that
		\[
		\norm{\Lc^{-1}\psi_j }_{L^1(G)}, \norm{\Lc^{-1}\psi'_j}_{L^1(G)}\meg C_1 2^{-j\grado}
		\]
		for every $j\Meg 1$, and such that $\norm{\psi_0}_{L^1(G)},\norm{\psi'_0}_{L^1(G)}\meg C_1$. Therefore, since $\Lc$ is formally self-adjoint,
		\[
		\begin{split}
		\abs{\langle K,\phi\rangle}^2, \abs{\langle K', \phi\rangle}^2&\meg C_1^2 \norm{\phi}_{L^\infty(G)}+ C_1^2\norm{\Lc \overline\phi}_{L^\infty(G)}\sum_{j\Meg 1} 2^{-j\grado}
		\end{split}
		\]
		for every $\phi\in \Sr(G)$. Moreover, given $\lambda>0$ and $a>Q_*+\lambda$, by Proposition~\ref{cor:12} again we may find a constant $C_2>0$ such that
		\[
		\begin{split}
			\abs{ X^R K(x)}, \abs{X^R K'(x)}&\meg C_2\norm*{ \frac{2^{j(Q_*+\deg(X))}}{(1+2^j\abs{x}_*)^a}}_{\ell^2_j(\N)}
		\end{split} 
		\]
		for every $X\in U_\lambda$ with $\abs{X}\meg 1$ and for every $x\in G\setminus \Set{e}$. Here and below we write $\norm{a_j}_{\ell^2_j(\N)}$ instead of $\norm{(a_j)_j}_{\ell^2(\N)}$, for simplicity. Now, if $\abs{x}_*\Meg 1$, then
		\[
		\begin{split}
			\norm*{ \frac{2^{j(Q_*+\deg(X))}}{(1+2^j\abs{x}_*)^a}}_{\ell^2_j(\N)}&\meg \frac{1}{\abs{x}_*^a} \norm{2^{j(Q_*+\lambda-a)}}_{\ell^2_j(\N)},
		\end{split}
		\]
		where $\norm{2^{j(Q_*+\lambda-a)}}_{\ell^2_j(\N)}$ is finite since $a>Q_*+\lambda$. If, otherwise, $\abs{x}_*<1$, then
		\[
		\begin{split}
			\norm*{ \frac{2^{j(Q_*+\deg(X))}}{(1+2^j\abs{x}_*)^a}}_{\ell^2_j(\N)}^2&=\sum_{j\in\N}\frac{2^{2j(Q_*+\deg(X)-a)}}{(2^{-j}+\abs{x}_*)^{2a}}\\
				&\meg \sum_{j\meg [\log_{1/2} \abs{x}_*]}2^{2j(Q_*+\deg(X))}+\sum_{j> [\log_{1/2} \abs{x}_*]} \frac{2^{2j(Q_*+\deg(X)-a)}}{ \abs{x}_*^{2a}}\\
				&\meg \frac{\abs{x}_*^{-2(Q_*+\deg(X))}  }{1-2^{-2 (Q_*+\deg(X))}} +\frac{ \abs{x}_*^{2(a-Q_*-\deg(X))} }{(1-2^{2(Q_*+\deg(X)-a)})\abs{x}_*^{2 a}},
		\end{split}
		\]
		so that there is a constant $C_3>0$ such that
		\[
		\norm*{ \frac{2^{j(Q_*+\deg(X))}}{(1+2^j\abs{x}_*)^a}}_{\ell^2_j(\N)}\meg \frac{C_3}{\abs{x}_*^{Q_*+\deg(X)}}.
		\]
        Observe that we may also consider $K$ and $K'$ as elements of $\Sr'(G; \Lf(H;\ell^2(\N;H)))$ and $\Sr'(G; \Lf(\ell^2(\N;H);H))$, respectively, simply tensorizing with the identity of $H$. The previous estimates then extend trivially.
		By the case $p=2$ there is a constant $C_4>0$ such that
		\[
		\norm{f*K}_{L^2(G;\ell^2(\N;H))}\meg C_4 \norm{f}_{L^2(G;H)}
		\]
		for every $f\in \Sr(G;H )$ and such that
		\[
		\begin{split}
			\norm{f*K'}_{L^2(G;H)}&= \sup_{\norm{g}_{L^2(G;H)}\meg 1}\abs{\langle f*K'\vert g \rangle}\\
				&\meg \sup_{\norm{g}_{L^2(G;H)}\meg 1}\abs{\langle f\vert g*K'^* \rangle}\\
				&\meg C_4 \sup_{\norm{g}_{L^2(G;H)}\meg 1}\norm{f}_{L^2(G;\ell^2(\N;H))} \norm{g}_{L^2(G;H)}\\
				&\meg C_4\norm{f}_{L^2(G;\ell^2(\N;H))}
		\end{split}
		\]
		for every $f\in \Sr(G;\ell^2(\N;H))$ since $(\psi_j^*)\in \Psi_{2^{-\grado}}(\Lc)$. Since we may take $\lambda,a$ as large as we please, by Theorem~\ref{teo:22bis} we see that there is a constant $C_5>0$ such that
		\[
		\norm{f*K}_{h^p(G;\ell^2(\N;H))}\meg C_5 \norm{f}_{h^p(G;H)}
		\]
		for every $f\in \Sr(G;H)$, and such that
		\[
		\norm{f*K'}_{h^p(G;H)}\meg C_5 \norm{f}_{h^p(G;\ell^2(\N;H))}
		\]
		for every $f\in \Sr(G;\ell^2(\N;H))$. Now, take $\tau\in C^\infty_c(\R)$ such that $\tau=1$ on a neighbourhood of $0$ and such that, if we set $\psi''_j\coloneqq \Kc_{2^{-j\grado} \Lc}(\tau)$, then $\psi'_{j'}=\psi'_{j'}*\psi''_j$ (so that $\psi_{j'}=\psi_{j'}*\psi''_j$) for every $j,j'\in\N$ with $j'\meg j$. Then, 
		\[
		\abs{f*K*\psi''_j}\Meg \norm{f*\psi_{j'}}_{\ell^2_{j'}(\Set{1,\dots,j})}
		\]
		so that by Lemma~\ref{lem:71} there is a constant $C_6>0$ such that
		\[
		\norm{(f*\psi_j)(x)}_{L^{2,p}_{j,x}(\N,G;H)}\meg \Bigl\lVert\sup_{j\in \N} \abs{f*K*\psi''_j}  \Bigr\rVert_{L^p(G)}\meg C_6 \norm{f*K}_{h^p(G;\ell^2(\N;H))}\meg C_5 C_6 \norm{f}_{h^p(G;H)}
		\]
		for every $f\in \Sr(G;H)$. By the density of $\Sr(G;H)$ in $h^p(G;H)$ (see~Corollary~\ref{cor:27}), this proves the inclusion $h^p(G;H)\subseteq F^{p,2}_0(G;H)$.
		Conversely, take $f\in \Sr(G;H)$ and define  $\widetilde f_k\in \Sr(G;\ell^2(\N;H))$ so that $\widetilde f_k(x)= (\chi_{\Set{0,\dots, k}}(j)(f*\psi_j)(x))_{j\in \N}$ for every $x\in G$ and for every $k\in \N$. Observe that there is $N\in\N$ such that $\psi_{j'}*\psi''_j=0$ for $j\Meg j'+N$, and that $\widetilde f_k*K'= \sum_{j\meg k} f*\psi_j$ converges to $f$ in $\Sr'(G;H)$ for $k\to \infty$. Take $a>D/p$. By Lemma~\ref{lem:71}, Proposition~\ref{cor:12}, and~\cite[Lemma 2.31 and Corollary 2.30]{Calzi4},  there is a constant $C_7>0$ such that
		\[
		\begin{split}
			\norm{f}_{h^p(G;H)}&\meg \liminf_{k\to \infty} \norm{\widetilde f_k *K'}_{h^p(G;H)}\\
				&\meg C_5 \liminf_{k\to \infty} \norm{\widetilde f_k  }_{h^p(G;\ell^2(\N;H))}\\
				&\meg C_5 C_7\liminf_{k\to \infty}\Bigl\lVert\sup_{j\in \N} \norm{\widetilde f_k*\psi''_j}_{\ell^2(\N;H)}  \Bigr\rVert_{L^p(G)}\\
                &\meg C_5 C_7 \Bigl\lVert\sup_{j\in \N} \norm{f*\psi_{j'}*\psi''_j}_{\ell^2_{j'}(\N;H)}  \Bigr\rVert_{L^p(G)}\\
                &\meg C_5 C_7^2 \Bigl\lVert  \norm{\Nc_{\infty,a,2^{-j\grado}}(f*\psi_{j})}_{\ell^2_{j}(\N;H)}  \Bigr\rVert_{L^p(G)}\\
                &\meg C_5 C_7^3 \Bigl\lVert  \norm{\Nc_{p/2,a,2^{-j\grado}}(f*\psi_{j})}_{\ell^2_{j}(\N;H)}  \Bigr\rVert_{L^p(G)}\\
				&\meg C_5 C_7^4 \norm{(f*\psi_j)(x)}_{L^{2,p}_{j,x}(\N,G;H)},
		\end{split}
		\]
        where in order to get the fifth inequality, we use the estimates \[
        \abs{f*\psi_{j'}*\psi''_j}=\abs{f*\psi_{j'}}\meg C_7\Nc_{\infty,a,2^{-j'\grado}}(f*\psi_{j'})
        \]
        when $j'\meg j$, and
        \[
        \abs{f*\psi_{j'}*\psi''_j}\meg C_7\Nc_{\infty,a,2^{-j'\grado}}(f*\psi_{j'})
        \]
        when $j<j'<j+N$.
		By the density of $\Sr(G;H)$ in $F^{p,2}_0(G;H)$ (see~\cite[Proposition 3.16]{Calzi4}), this proves the inclusion $F^{p,2}_0(G;H)\subseteq h^p(G;H) $.
	\end{proof}

\section{The $\bmo$ Space}\label{Section:bmoSpaces}

In this section, we introduce local $\bmo$ spaces on groups of polynomial growth and establish their duality with the local Hardy space introduced above. We also provide an equivalent characterization in terms of the Triebel–Lizorkin space $F_0^{\infty,2}$.
\begin{deff}
	Take $q\in [1,\infty)$, $\rho>0$, and a \BanachSpace\ $\Banach$, and define $\bmo^{q,\rho}(G;\Banach)$ as the space of $f\in \Lp^q_\loc(G;\Banach)$ such that, with some abuse of notation, 
	\[
	\sup_{t>0} \sup_{x\in G} \left( \dashint_{B(x,t)} \abs*{f-m_{\rho}(t,f)}^q \,\dd \beta\right) ^{1/q}
	\]
	is finite, where $m_\rho(t,f)=\dashint_{B(x,t)} f\,\dd \beta$ if $t\meg \rho$, while $m_\rho(t,f)\coloneqq 0$ if $t>\rho$. We endow $\bmo^{q,\rho}(G;\Banach)$ with the corresponding norm.
	
	In addition, define $\mathring h^1(G;\Banach)$ as the closure of $\Sr(G;\Banach)$ in $h^1(G;\Banach)$.
\end{deff}

By inspection of the proof of Theorem~\ref{teo:18}, one may show that, if $f\in \mathring h^1(G;\Banach)$, then there is $(\lambda_j)\in \ell^1(\N)$ and a family $(f_j)$ of $(1,\infty,\dd_0,\rho)$-atoms such that $f=\sum_j \lambda_j f_j$ and such that $f_j\in L^\infty(G;\Banach)$ for every $j\in\N$ (for $\rho$ sufficiently small).

\begin{lem}\label{lem:100}
Take $q\in [1,\infty)$, $\rho_1,\rho_2>0$, and a \BanachSpace\ $\Banach$. Then, $\bmo^{q,\rho_1}(G;\Banach)=\bmo^{q,\rho_2}(G;\Banach)$ with equivalence of the norms.
\end{lem}
\begin{proof}
Assume $0<\rho_1<\rho_2$ and fix $x\in G$.
For $t\in (0,\rho_1]\cup (\rho_2,+\infty)$ we have
$m_{\rho_1}(t,f)=m_{\rho_2}(t,f)$,
so that there is nothing to prove.
Now, let $t\in (\rho_1,\rho_2]$, in this case $m_{\rho_1}(t,f)=0$ and $m_{\rho_2}(t,f)=\dashint_{B(x,t)} f\,\dd \beta$.
Using Minkowski's inequality and H\"older's inequality we get
\begin{align*}
\Bigg(
\frac{1}{\beta(B(x,t))}&\int_{B(x,t)}
|f-\textstyle\dashint_{B(x,t)} f\,\dd \beta|^q \dd\beta
\Bigg)^{1/q}\\
\meg &
\left(
\frac{1}{\beta(B(x,t))}\int_{B(x,t)} |f|^q\dd\beta
\right)^{1/q}
+ \left(
\frac{1}{\beta(B(x,t))}\int_{B(x,t)} |\textstyle\dashint_{B(x,t)} f\,\dd \beta|^q\dd\beta
\right)^{1/q}\\
\meg& 
2\left(
\frac{1}{\beta(B(x,t))}\int_B |f|^q\dd\beta
\right)^{1/q}.
\end{align*}
Therefore
\[
\|f\|_{\bmo^{q,\rho_2}(G;\Banach)}
\meg 2
\|f\|_{\bmo^{q,\rho_1}(G;\Banach)}.
\]
Conversely, observe that 
\[
\dashint_{B(x,t)} \abs{f}^q\,\dd \beta \meg C\dashint_{B(x,\rho_2)} \abs{f}^q\,\dd \beta\meg \|f\|_{\bmo^{q,\rho_2}(G;\Banach)}^q
\]
where $C=\frac{\beta(B(e,\rho_2))}{\beta(B(e,\rho_1))}$. Consequently, 
\[
\|f\|_{\bmo^{q,\rho_1}(G;\Banach)}\meg C^{1/q}\|f\|_{\bmo^{q,\rho_2}(G;\Banach)}
\]
whence the conclusion.

\end{proof}

\begin{deff}
		 Take $q\in [1,\infty]$. We define $\bmo^q(G;\Banach)$ as $\bmo^{q,\rho}(G;\Banach)$ for any $\rho>0$.
	\end{deff}

\begin{teo}\label{teo:26}
	Take $q\in [1,\infty)$ and a \BanachSpace\ $\Banach$. Then, there is a unique continuous bilinear mapping
	\[
	\mathring h^1(G;\Banach)\times \bmo^{q}(G;Z')\to \C
	\]
	which extends the canonical continuous bilinear mapping on $\Sr(G;\Banach)\times \Sr'(G;Z')$.\footnote{Notice that  we defined $\Sr'(G;Z')$ as $\Lf(\Sr(G);Z')$ and that $\Sr(G)$ is a nuclear Fréchet space (see~\cite[Theorem 2.15]{Calzi4}), so that one has $\Sr'(G;Z')\cong (\Sr(G)\widehat \otimes Z)'$, thanks to~\cite[Propositions 50.4, 50.6 and 50.7]{Treves}. In addition, $\Sr(G)\widehat \otimes Z\cong \Sr(G;Z)$ (see~\cite[Exercise 44.6]{Treves} for the case $G=\R^n$).} In addition, this mapping induces an isomorphism of $ \bmo^{q}(G;Z')$ onto $\mathring h^1(G;\Banach)'$.
\end{teo}

\begin{proof}
	Observe that, by Theorem~\ref{teo:18} and the above remarks, there are $\rho,C>0$ such that, if $f\in \Sr(G;\Banach)$, then there is   $(\lambda_j)\in \ell^1(\N)$ and a family $(f_j)$ of $(1,q',\dd_0,\rho)$-atoms such that
    \[
    f=\sum_j \lambda_j f_j, \, \sum_j \abs{\lambda_j}\meg C \norm{f}_{h^1(G;\Banach)} \text{ and }f_j\in L^\infty(G;\Banach) \text{ for every }j\in\N.
    \]
    In addition, take $g\in \bmo^{q}(G;Z')$. If $f_j$ is supported in a ball $B(x,t)$ with $t\meg \rho$, then
	\[
	\abs{\langle f_j,g\rangle}= \abs{\langle f_j, g- m_\rho(t,g) \rangle} \meg \beta(B(e,t))^{1/q'-1}\beta(B(e,t))^{1/q} \norm{g}_{\bmo^{q,\rho}(G;Z')}=\norm{g}_{\bmo^{q,\rho}(G;Z')}
	\]
	since $\int_G f_j\,\dd \beta=0$. If, otherwise, $f_j$ is supported in a ball $B(x,t)$ with $t>\rho$, then
	\[
	\abs{\langle f_j,g\rangle}\meg  \beta(B(e,t))^{1/q'-1}\beta(B(e,t))^{1/q} \norm{g}_{\bmo^{q,\rho}(G;Z')}=\norm{g}_{\bmo^{q,\rho}(G;Z')}.
	\]
	It then follows that 
	\[
	\abs{\langle f,g\rangle}\meg C  \norm{f}_{h^1(G;\Banach)}\norm{g}_{\bmo^{q,\rho}(G;Z')},
	\]
	whence the first assertion. 
	In order to conclude, we need to show that, if $v\in \mathring h^1(G;\Banach)'$, then there is a (necessarily unique) $g\in \bmo^{q }(G;Z')$ such that $v(f)=\langle f, g \rangle$ for every $f\in \mathring h^1(G;\Banach)$. Observe first that, since by Proposition~\ref{prop:25} we have $(1+\abs{\,\cdot\,}_*)^{-a} L^{q'}(G;\Banach)\subseteq \mathring h^1(G;\Banach)$ continuously (for $a>Q_*/q'$), we may find a unique $g\in (1+\abs{\,\cdot\,}_*)^{a} \Lp^{q}(G;Z')$ such that $v(f)=\langle f, g \rangle$ for every $f\in (1+\abs{\,\cdot\,}_*)^{-a} L^{q'}(G;\Banach)$. In addition, we may take a constant $C'>0$ such that $\abs{v(f)}\meg C'$ for every $(1,q',\dd_0,\rho)$-atom $f\in L^{q'}(G;Z)$. Then, take $x\in G$ and $t>0$. If $t>\rho$, then taking $\phi\in C_c(G;\Banach)$ with $0\meg \abs{\phi}\meg \chi_{B(e,t)}$ we see that
	\[
	\abs{\langle \phi, g \rangle}=\abs{v(\phi)}\meg C'  \beta(B(e,t))^{1/q} \norm{\phi}_{L^{q'}(G;\Banach)}
	\]
	since $ \phi/( \beta(B(e,t))^{1/q} \norm{\phi}_{L^{q'}(G;\Banach)}) $ is a $(1,q',\dd_0,\rho)$-atom. By the arbitrariness of $\phi$, this proves that  
	\[
	\int_{B(x,t)} \abs{g}^q\,\dd \beta\meg C'  \beta(B(e,t))^{1/q}.
	\]
	If, otherwise, $t\meg \rho$, then taking $\phi\in L^{q'}(G;\Banach)$ with $0\meg \abs{\phi}\meg \chi_{B(e,t)}$   we see that
	\[
    \begin{split}
	\abs{\langle \phi, g -m_{\rho}(t,g)\rangle}&=\abs{\langle \phi-m_{\rho}(t,\phi), g -m_{\rho}(t,g)\rangle}\\
        &=\abs{v(\phi-m_\rho(t,\phi))}\\
        &\meg 2C'  \beta(B(e,t))^{1/q} \norm{\phi}_{L^{q'}(G;\Banach)}
    \end{split}
	\]
	since $(\phi-m_\rho(t,\phi))/( 2\beta(B(e,t))^{1/q} \norm{\phi}_{L^{q'}(G;\Banach)}) $ is a $(1,q',\dd_0,\rho)$-atom. By the arbitrariness of $\phi$, this proves that
	\[
	\int_{B(x,t)} \abs{g-m_\rho(t,g)}^q\,\dd \beta\meg C'  \beta(B(e,t))^{1/q}.
	\]
	We have therefore proved that
	\[
	\norm{g}_{\bmo^{q,\rho}(G;Z')}\meg C',
	\]
	whence the conclusion.
\end{proof}

By combining Lemma \ref{lem:100} and Theorem \ref{teo:26}, we get the following result. The assertion is clear if $\Banach$ is the dual of some Banach space, hence also if we replace $\Banach$ with $\Banach''$. Since the embedding $\Banach\subseteq \Banach''$ is isometric, the assertion follows also in the general case.

\begin{cor}
Take $q_1,q_2\in [1,\infty)$, $\rho_1,\rho_2>0$, and a \BanachSpace\ $\Banach$. Then, $\bmo^{q_1,\rho_1}(G;\Banach)$ and $\bmo^{q_2,\rho_2}(G;\Banach)$ are isomorphic.
\end{cor}

\begin{deff}
		We define $\bmo(G;\Banach)$ as $\bmo^{q}(G;\Banach)$ for any $q\in[1,\infty)$. 
	\end{deff}

As a consequence of Theorem~\ref{teo:26} and of the definition of $F^{\infty,2}_0(G;H)$ as the dual of $F^{1,2}_0(G;H)$ (canonically embedded in $\Sr'(G;H)$), we obtain the following result. 

\begin{cor}
	Let $H$ be a Hilbert space. Then, $\bmo(G;H)= F^{\infty,2}_0(G;H)$ with equivalence of norms.
\end{cor}

\section{Proof of Lemma~\ref{lem:73}}\label{sec:app2}
		\textsc{Step I} Take $(\rho_j)\in (0,+\infty)^n$ and $g\in \Lp^p(\R^n;\Banach)$ supported in  $R=\prod_{j=1}^n [-\rho_j,\rho_j]$, and assume that
		\[
		\langle g, 1\rangle=0.
		\]
		Define  $h_1\in \Mc(\R^{n-1};\Banach)$ setting
		\[
		\langle h_1, \phi\rangle \coloneqq \frac{1}{2 \rho_1}\langle g, 1\otimes \phi\rangle
		\]
		for every $\phi\in C_c(\R^{n-1})$, so that $h_1$ is supported in $R_1\coloneqq \prod_{j=2}^n [-\rho_j,\rho_j]$ and
		\[
		\abs{h_1}\meg  \frac{1}{2 \rho_1}\int_{\R} \abs{g}(x_1,\,\cdot\,)\,\dd x_1.
		\]  
		Consequently, $\abs{h_1}\in L^1_\loc(\R^{n-1})$ and, setting $\widetilde h_1\coloneqq  (\chi_{[-\rho_1,\rho_1]}\cdot \Hc^1) \otimes h_{1} $, 
		\[
		\begin{split}
			\norm{\widetilde h_1 }_{\Lp^p (\R^n;\Banach)}&\meg (2\rho_1)^{1/p}\norm{h_{1}}_{\Lp^p(\R^{n-1};\Banach)}\\
			&\meg (2\rho_1)^{1/p-1}\norm*{\norm{\abs{g}(x_1,x')}_{L^1_{x_1}(\R)}}_{L^{p}_{x'}(\R^{n-1} )}\\
			&\meg   \norm{g}_{L^p(\R^n;\Banach)}.
		\end{split}
		\]
		In addition,
		\[
		\langle g-\widetilde h_{1}, 1\otimes \phi\rangle=0
		\] 
		for every $\phi \in C_c(\R^{n-1})$. Consequently, if we define
		\[
		g_1\coloneqq (g-\widetilde h_1) *[(\chi_{[0,+\infty)}\cdot \Hc^1)\otimes \delta_0],
		\]
		then $g_1\in \Mc(\R^n;\Banach)$ and
		\[
		\abs{g_1}\meg \abs{\widetilde h_1}*[(\chi_{[0,+\infty)}\cdot \Hc^1)\otimes \delta_0],
		\]
		so that $\abs{g_1}\in L^1_\loc(\R^n)$. In addition, $\supp g_1\subseteq R$,\footnote{This is related to the fact that, if $\phi\in C_c(\R^{n-1})$ and we consider the measures $g^{(\phi)}$, $h^{(\phi)}_{1 }$ and $g_{1 }^{(\phi)}$  defined so that $\langle g^{(\phi)}, \psi\rangle\coloneqq \langle g, \psi\otimes \phi\rangle$ for every $\psi\in C_c(\R)$ (the other being defined analogously), then $g_1^{(\phi)}\in C(\R;\Banach) $ and $g_1^{(\phi)}(x)=\int_{(-\infty,x]}\,\dd  (g^{(\phi)}-h_1^{(\phi)})$ for every $x\in \R$, the integral being a well defined element of $\Banach$ thanks to~\cite[Proposition 6 of Chapter VI, \S\ 2, No.\ 3]{BourbakiInt1}.}
		\[
		g =  \partial_1 g_1 +\widetilde h_1 
		\] 
		and, arguing as above,
		\[
		\begin{split}
			\norm{g_1}_{\Lp^p(\R^n;\Banach)}&\meg 2 \rho_1 \norm{ g-\widetilde h_{1} }_{\Lp^{p}(\R^n;\Banach)}\meg 2(2\rho_1) \norm{g}_{\Lp^p(\R^n;\Banach)}.
		\end{split}
		\]
		Now, observe that $\langle \widetilde h_1, 1\rangle=\langle h_1,1\rangle=\frac{1}{2\rho_1}\langle g,1\rangle=0$. Consequently, arguing inductively one may find $g_2,\dots, g_n\in \Lp^p(\R^n;\Banach)$, supported in $R$, such that
		\[
		g=\sum_{j=1}^n \partial_j g_j
		\]
		and such that
		\[
		\norm{g_j}_{\Lp^p(\R^n;\Banach)}\meg 4\rho_j \norm{g}_{\Lp^p(\R^n;\Banach)}.
		\]
		In fact, $ g_j$ only depends on $x_{j},\dots, x_n$ on $R$ (but not on $\R^n$).
		
		\textsc{Step II} Recall that  $(Y_j)$ denotes the   homogeneous basis of $\gf_*$ induced by $(X_j)$,  that $d_j=\deg(Y_j)=\deg(X_j)$ for every $j\in J$, and that $\dd_M= \max_{j\in J} d_j$, while $\dd_0=\min_{j\in J} d_j$. Endow $\gf_*$ with the homogeneous norm $x\mapsto \max_{j=1}^n \abs{\langle  x, Y_j\rangle}^{1/d_j}$, and denote with $\overline B_\rho$ the corresponding closed ball of center $0$ and radius $\rho$.   Take $\rho>0$ and $g\in \Lp^p(\gf_*;\Banach)$ supported in $\overline B_\rho$ such that
		\[
		\langle g, 1\rangle=0
		\]
		Let $(\partial_j)$ be the partial derivatives on $\gf_*$ corresponding to the identification of $\gf_*$ with $\R^J$ by means of the homogeneous basis $(Y_j)$, and observe that~\textsc{step I} shows that there are   $\widetilde g_j\in \Lp^p(\gf_*;\Banach)$, $j\in J$, such that $\supp \widetilde g_j\subseteq\overline B_\rho$, such that
		\[
		g=\sum_{j=1}^n \partial_j\widetilde g_j,
		\]
		and such that
		\[
		\norm{\widetilde g_j}_{\Lp^p(\gf_*;\Banach)}\meg 4\rho^{d_j} \norm{g}_{\Lp^p(\gf_*;\Banach)}.
		\]
		Observe that there is a family $(P_{j,k})_{j,k\in J}$ of homogeneous polynomials on $\gf_*$ such that $\deg P_{j,k}= d_j-d_k$ (so that $P_{j,k}=0$ if $d_j<d_k$) for every $j,k\in J$, and such that
		\[
		\partial_k =\sum_{j\in J} P_{j,k}X_j
		\]
		for every $k\in J$.
		Since Lebesgue measure is invariant both under the translations of the vector space $\gf_*$ and of the group $G_*$ (identified with $\gf_*$ by means of the exponential map), by transposition we get
		\[
		\partial_k h= \sum_{j\in J} X_j(P_{j,k} h)
		\]
		for every $h\in \Dc'(\gf_*;\Banach)$ and for every $k\in J$. In particular, if we set
		\[
		g_j\coloneqq \sum_{k\in J} P_{j,k} \widetilde g_k ,
		\]
		then
		\[
		g=\sum_{j\in J} X_j g_j
		\]
		and there is a constant $C_1>0$ such that
		\[
		\norm{ g_j}_{\Lp^p(\gf_*;\Banach)}\meg C_1\rho^{d_j} \norm{g}_{\Lp^p(\gf_*;\Banach)}
		\]
		for every $j\in J$, using that $P_{j,k}=0$ if $d_j<d_k$, while $\max_{\overline B_\rho} \abs{P_{j,k}}= \rho^{d_j-d_k} \max_{z\in \partial \overline B_1} \abs{P_{j,k}}$ if $d_j\Meg d_k$.
		
		\textsc{Step III} Keep the notation of~\textsc{step II} and set $\Dc=\Set{\lambda \Meg0\colon U_\lambda\neq U_{\lambda^-}}$, but this time assume that $\lambda\in \Dc$, with $\lambda>0$, and that
		\[
		\langle g, P\rangle=0
		\]
		for every   $P\in \Pc_\lambda$.  Set   $\lambda'\coloneqq \max \Dc \cap [0,\lambda)$. Recall that $q(t)=[t/\dd_0]+1$ for every $t\Meg 0$. Assume, by induction (on $\Dc$) that there are a constant $C_2>0$ and a family $(g_\alpha)_{\lambda'<d_\alpha \meg  \lambda'+\dd_M}$ of elements of $\Lp^p(\gf_*;\Banach)$ supported in $\overline B_\rho$ such that $g_\alpha=0$ if $\abs{\alpha}\Meg \min(q(d_\alpha),q(\lambda')+1)$, such that
		\[
		g=\sum_{\lambda'< d_\alpha\meg \lambda'+\dd_M } \vect Y^\alpha g_\alpha= \sum_{\lambda\meg d_\alpha\meg \lambda'+\dd_M } \vect Y^\alpha g_\alpha
		\]
		and such that
		\[
		\norm{g_\alpha}_{\Lp^p(\gf_*;\Banach)}\meg C_2\rho^{d_\alpha} \norm{g}_{\Lp^p(\gf_*;\Banach)}
		\]
		for every $\alpha$ with $\lambda\meg d_\alpha\meg \lambda'+\dd_M $ (the case $\lambda'=0$ follows from \textsc{step I}).
		Take $\alpha$ with $d_\alpha=\lambda$ and $g_\alpha \neq 0$ and observe that there is a homogeneous polynomial $P$ on $\gf_*$ with (homogeneous) degree $\lambda$ such that $(\vect Y^{\alpha} P)(0)=1$ while $(\vect Y^{\alpha'}P)(0)=0$ for every $\alpha'\neq \alpha$ with $d_{\alpha'}=\lambda$. 
		Since, however, $ \vect Y^{\alpha} P $ and  $ \vect Y^{\alpha'}P $ are constant functions by homogeneity, this shows that  $ \vect Y^{\alpha} P =1$ and  $ \vect Y^{\alpha'}P =0$ for  every $\alpha'$ as above, while clearly $\vect Y^{\alpha''}P=0$ for every $\alpha''$ with $d_{\alpha''}>\lambda$. Consequently,
		\[
		0=\langle g, P\rangle= \sum_{\lambda\meg d_{\alpha'}\meg \lambda'+\dd_M } \langle  \vect Y^{\alpha'} g_{\alpha'}, P\rangle= (-1)^{\abs{\alpha}}\langle g_{\alpha}, 1 \rangle
		\]
		so that by~\textsc{step II} there are $g_{\alpha,j}\in \Lp^p(\gf_*;\Banach)$ ($j\in J$), supported in $\overline B_\rho$, such that
		\[
		g_\alpha= \sum_{j\in J} Y_j g_{\alpha,j}
		\]
		and
		\[
		\norm{g_{\alpha,j}}_{\Lp^p(\gf_*;\Banach)}\meg C_1\rho^{d_j} \norm{g_\alpha}_{\Lp^p(\gf_*;\Banach)}\meg C_1 C_2\rho^{d_\alpha+d_j} \norm{g}_{\Lp^p(\gf_*;\Banach)}
		\]
		for every $j\in J$. Observe that $\abs{\alpha}+1 \meg q(d_\alpha)=q(\lambda)<q(\lambda+\dd_0)$ since $d_j\Meg \dd_0$ for every $j\in J$.
		After expressing $\vect X^\alpha X_j$ as a finite linear combination of the $\vect X^{\alpha'}$ with $\lambda+\dd_0\meg\dd_{\alpha'}\meg \lambda+\dd_M$ and $\abs{\alpha'}\meg \abs{\alpha}+1$, and applying the previous arguments to every $\alpha$ such that $d_\alpha=\lambda$ and $g_\alpha\neq 0$, one can see that there are $C_3>0$ and a family $(g'_\alpha)_{\lambda < d_\alpha \meg \lambda+ \dd_M}$ such that $g'_\alpha =0$ if $\abs{\alpha}\Meg \min(q(d_\alpha), q(\lambda)+1)$, such that
		\[
		g= \sum_{\lambda< d_\alpha\meg \lambda+\dd_M } \vect Y^\alpha g'_\alpha
		\]
		and such that
		\[
		\norm{g'_\alpha}_{\Lp^p(\gf_*;\Banach)}\meg C_3\rho^{d_\alpha} \norm{g}_{\Lp^p(\gf_*;\Banach)}
		\]
		for every $\alpha$ with $\lambda< d_\alpha\meg \lambda+\dd_M $. 
		
		\textsc{Step IV} Let us now prove the statement. Identify $\gf$ and $\gf_*$, as vector spaces, by means of the bases $(X_j)$ and $(Y_j)$. By~\cite[Proposition 6.1]{ElstRobinson}, we may find $C_4>1$ and  $\eps'\in (0,\mathrm{inj}(G)]$ so that  
		\[
		\frac{1}{C_4}\abs{\exp_G(y)}_*\meg \max_{j\in J} \abs{\langle y, Y_j\rangle}^{1/d_j}\meg C_4 \abs{\exp_G(y)}_*
		\]
		for every $y\in B_\gf(0,\eps')$. Take $\eps\in (0,1)$ so that $B(e,\eps)\subseteq \exp_G(B_\gf(0,\eps'))$ and $\overline  B_{C_4 \eps}\subseteq B_\gf(0,\eps')$.
		Let $b$ be the nowhere-vanishing real analytic function on $B_\gf(0,\eps')$ such that  
\[(\exp_G)_*(b\cdot \beta_{\gf})=\chi_{\exp_G(B_\gf(0,\eps'))}\cdot\beta,\]
 where $\beta_\gf$ denotes the Lebesgue measure on $\gf$. 
		Denote with 
\[\theta\colon B_\gf(0,\eps')\to x\exp_G(B_\gf(0,\eps'))\]
 the diffeomorphism induced by $x\exp_G$, and set $g\coloneqq \theta^*(f)$, extended by $0$ on $\gf$. Then, $g$  is supported in $\overline B_{C_4\rho}$ and
		\[
		\abs{g}(y)= \abs{f}(x\exp_G(y))b(y)
		\]
		for $\beta_\gf$-almost every $y\in B_\gf(0,\eps')$. Notice that we are identifying $\abs{g}$ with its density with respect to $\beta_\gf$, and $\abs{f}$ with its density with respect to $\beta$. This ensures that $g\in \Lp^p(\gf;\Banach)$ and that
		\[
		\norm{g b^{-1/p'}}_{\Lp^p(\gf;\Banach)}=\norm{f}_{\Lp^p(G;\Banach)}.
		\]
		Notice that the assumptions on $f$ mean that
		\[
		\langle g, P \rangle=0
		\]
		for every polynomial $P$ on $\gf$ with $\deg P<\lambda$, that is, $\deg P\meg \lambda'$ with the notation of~\textsc{step III}. 
		Then, by~\textsc{step III}, we may find a family $(g_\alpha)_{\lambda \meg d_\alpha \meg \lambda' +\dd_M}$ of elements of $\Lp^p(\gf;\Banach)$ supported in $B_{C_4\rho} $ such that $g_\alpha=0$ if $\abs{\alpha}\Meg \min(q(d_\alpha),q(\lambda')+1 )$, such that
		\[
		g= \sum_{ \lambda \meg d_\alpha \meg \lambda'+\dd_M} \vect Y^\alpha g_\alpha
		\]
		and such that
		\[
		\norm{g_\alpha}_{\Lp^p(\gf;\Banach)}\meg C_3\rho^{d_\alpha} \norm{g }_{\Lp^p(\gf;\Banach)}\meg C_3\rho^{d_\alpha}\norm{b}_{L^\infty(\overline B_{C_4\eps})}^{1/p'}\norm{f}_{\Lp^p(G;\Banach)}
		\]
		for every $\alpha$ such that  $\lambda \meg d_\alpha \meg \lambda'+d_M$.
		Observe that, by Lemma~\ref{lem:72}, we may find a constant $C_5>0$ and a family of real analytic functions $(a_{\alpha,\gamma})$ on $B_\gf(0,\eps')$ such that
		\[
		\vect Y^\alpha= \sum_{\abs{\gamma}\meg \abs{\alpha}} a_{\alpha, \gamma} \vect X^\gamma,
		\]
		with some abuse of notation, and such that
		\[
		\abs{a_{\alpha, \gamma}(x)}\meg C_5 \abs{\exp_G(x)}_*^{(d_\gamma-d_\alpha)_+}
		\]
		for every $x\in B_\gf(0,\eps')$. In addition, $a_{\alpha,0}=0$ for every $\alpha\neq 0$. 
		If we define $f_{\alpha,\gamma}\coloneqq \theta_*(g_\alpha a_{\alpha,\gamma})$ (extended by $0$ on $G$), then, arguing as above we see that $f_{\alpha,\gamma}\in \Lp^p(G;\Banach)$ and that  
		\[
		\norm{f_{\alpha,\gamma}}_{\Lp^p(G;\Banach)}\meg C_3C_5\rho^{\max(d_\alpha,d_\gamma)}\norm{b}_{L^\infty(\overline B_{C_4\eps})}^{1/p'}\norm{1/b}_{L^\infty(\overline B_{C_4\eps})}^{1/p'} \norm{f}_{\Lp^p(G;\Banach)}.
		\]
		In addition,\footnote{Notice that, for a distribution $u$ on $\gf$ and a differential operator $X$ of class $C^\infty$ on $\gf$, we define $Xu$ so that $\langle Xu, \phi\rangle= \langle u, X^\dag \phi\rangle$ where $X^\dag$ denotes the formal transpose of $X$ with respect to $\beta_\gf$, whereas for distributions and differential operators on $G$ we use the formal transpose with respect to $\beta$. One therefore has to be quite careful in handling these computations. }
		\[
		\begin{split}
			\langle f,\phi\rangle&= \langle g, \phi\circ \theta \rangle\\
				&=\sum_{ \lambda \meg d_\alpha \meg \lambda'+\dd_M} (-1)^{\abs{\alpha}}\langle g_\alpha, \vect Y^\alpha(\phi\circ \theta)\rangle\\
				&=\sum_{ \lambda \meg d_\alpha \meg \lambda'+\dd_M}\sum_{0<\abs{\gamma}\meg \abs{\alpha}} (-1)^{\abs{\alpha}}\langle a_{\alpha,\gamma} g_\alpha, (\vect X^\gamma \phi)\circ \theta\rangle\\
				&=\sum_{ \lambda \meg d_\alpha \meg \lambda'+\dd_M}\sum_{0<\abs{\gamma}\meg \abs{\alpha}}(-1)^{\abs{\alpha}} \langle f_{\alpha,\gamma} ,  \vect X^\gamma \phi \rangle\\
				&=\sum_{ \lambda \meg d_\alpha \meg \lambda'+\dd_M}\sum_{0<\abs{\gamma}\meg \abs{\alpha}}(-1)^{\abs{\alpha}+\abs{\gamma}} \langle \vect X^\gamma f_{\alpha,\gamma} ,   \phi \rangle
		\end{split}
		\]
		for every $\phi\in C^\infty_c(x\exp_G(B_\gf(0,\eps')))$, so that 
		\[
		f=  \sum_{ \lambda \meg d_\alpha \meg \lambda'+\dd_M}\sum_{0<\abs{\gamma}\meg \abs{\alpha}} (-1)^{\abs{\alpha}+\abs{\gamma}} \vect X^\gamma f_{\alpha,\gamma}.
		\]
		The assertion follows easily.

\section{Proof of Proposition~\ref{prop:31}}\label{sec:app}

Before we pass to the proof, we need a couple of auxiliary results. Recall that $(X_j)_{j\in J}$ denotes an \emph{ordered} basis of $\gf$ which induces a homogeneous basis $(Y_j)$ of $\gf_*$, and that $d_\alpha= \sum_{j\in J} \alpha_j \deg(X_j)$ for every $\alpha\in \N^J$.

\begin{lem}\label{lem:76}
	Take $q\in [1,\infty]$ and $k\in\N$ such that $k>\dim G/q$ if $q<\infty$.  Define  $\Wp^{k,q}_t(G;\Banach)$, for every $t\in (0,1]$ and for every \BanachSpace\ $\Banach$, as the space of $f\in \Lp^q(G;\Banach)$ such that
    \[
    \max_{\abs{\alpha}\meg k} t^{d_\alpha-Q_*/q}\norm{\vect X^\alpha f}_{\Lp^q(G;\Banach)}<\infty,
    \]
    endowed with the corresponding norm. There exists a constant $C>0$ such that
	\[
	\norm{f}_{\Lp^\infty(G;\Banach)}\meg C   \norm{f}_{\Wp^{k,q}_t(G;\Banach)}
	\]
	for every \BanachSpace\ $\Banach$, $f\in \Wp^{k,q}_t(G;\Banach)$,  and $t\in (0,1]$.
\end{lem}

Note that $\Wp^{k,q}_t(G;\Banach)=\Wp_1^{k,q}(G;\Banach)$ for every $t\in (0,1]$ as topological vector spaces, only the norm depends on $t$. 
We also observe that the statement appears to be somewhat  peculiar, since a `Riemannian' condition (the control on the order of $\vect X^\alpha$ instead of its degree) is mixed with a condition which is based on the filtration of $\gf$. This is related to the structure of the proof, but is also better adapted to the application of this result in the proof of Proposition~\ref{prop:31}. In any case, this reflects the behaviour of the usual Sobolev embeddings in the Euclidean setting when perturbed by anysotropic dilations.

\begin{proof}
	\textsc{Step I} Assume first that $\Banach=\R$ or $\C$. Take $\rho\in (0,1]$ so that $\overline B(e,\rho)\subseteq \mathbf U$, and let us prove the statement for $t\in (0,\rho]$ and for $f$ supported in $B(e,t)$. 
	We may therefore transfer the problem to $\gf$, and replace $\beta$ with the Lebesgue measure $\beta_\gf$ on $\gf$, since the density of $\beta$ with respect $\beta_\gf$ is a positive and nowhere vanishing real analytic function (and is therefore bounded from above and from below on $\overline B(e,\rho)$). 
	Now, by the Euclidean Sobolev embeddings we see that, identifying $\gf$ with $\R^J$ by means of the basis $(X_j)$,  there is a constant $C_1>0$ such that
	\[
	\norm{f}_{L^\infty(\gf)}\meg C_1 \max_{\abs{\alpha}\meg k}\norm{\partial^\alpha f }_{L^q(\gf)}
	\]
	for every $f\in \Wp^{k,q}(\gf)$. Applying the previous inequality to $f(t\,\cdot\,)$ (using the dilations induced by $\gf_*$ and \emph{not} the Euclidean ones), we then see that
	\[
	\norm{f}_{L^\infty(\gf)}\meg C_1 \max_{\abs{\alpha}\meg k}t^{d_\alpha}\norm{(\partial^\alpha f)(t\,\cdot\,) }_{L^q(\gf)}=C_1 \max_{\abs{\alpha}\meg k}t^{d_\alpha-Q_*/q}\norm{\partial^\alpha f }_{L^q(\gf)}
	\]
	for every $t>0$ and for every $f\in \Wp^{k,q}(\gf)$. Now, observe that, by Lemma~\ref{lem:72} (together with~\cite[Proposition 1.29]{FollandStein}), we may find a constant $C_2>0$ and a family $(a_{\alpha,\gamma})_{\abs{\gamma}\meg \abs{\alpha}\meg k}$ of real analytic functions on $\mathbf U$ such that 
    \[
    \partial^\alpha=\sum_{\abs{\gamma}\meg \abs{\alpha}} a_{\alpha,\gamma} \vect X^\gamma \,\text{  and  }\,\abs{a_{\alpha,\gamma}(x)}\meg C_2\abs{x}_*^{(d_\gamma-d_\alpha)_+}
    \]
    for every $x\in \overline B(e,\rho)$ and for every $\alpha,\gamma$ with $\abs{\gamma}\meg \abs{\alpha}\meg k$. Then,
	\[
	t^{d_\alpha}\norm{\partial^\alpha f }_{L^q(\gf)}\meg C_2 t^{d_\alpha} \sum_{\abs{\gamma}\meg k} \norm*{\abs{\,\cdot\,}_*^{(d_\gamma-d_\alpha)_+}( \vect X^\gamma f)}_{L^q(\gf)}\meg C_2 t^{d_\gamma}\sum_{\abs{\gamma}\meg k} \norm{  \vect X^\gamma f}_{L^q(\gf)}
	\]
	for every $t\in (0,\rho]$ and for every $f\in \Wp^{k,q}(\gf)$ supported in $B(e,t)$. The assertion follows.
	
	\textsc{Step II} Assume  that $\Banach=\R$ or $\C$. We shall now prove the assertion for $t\in (0,\rho]$ and for general $f$, with $\norm{f}_{L^\infty(G)}$ replaced by $\norm{\chi_{B(e,t/2)}f}_{L^\infty(G)}$. To this aim, take $(\tau_t)_{t>0}$ as in Corollary~\ref{cor:21} with $\kappa_1=1/2$ and $\kappa_2=1$, so that $\chi_{B(e,t/2)}\meg \tau_t \meg \chi_{B(e,t)}$ and there is a constant $C_3>0$ such that
	\[
	\norm{\vect X^\alpha\tau_t}_{L^\infty(G)}\meg C_3 t^{-d_\alpha} 
	\]
	for every $t\in (0,1]$ and for every $\alpha$ with length $\meg k$. Then, by~\textsc{step I} there is a constant $C_4>0$ such that
	\[
	\begin{split}
		\norm{f}_{L^\infty(B(e,t/2),\beta)}&\meg \norm{f \tau_t}_{L^\infty(G)}\\
		&\meg C_4 t^{-Q_*/q}\max_{\abs{\alpha}\meg k} t^{d_\alpha} \norm{\vect X^\alpha (\tau_t f)}_{L^q(G)}\\
		&\meg C_3C_4 t^{-Q_*/q}\max_{\abs{\alpha}\meg k} \sum_{\alpha_1+\alpha_2=\alpha} \frac{\alpha!}{\alpha_1!\alpha_2!} t^{d_\alpha-d_{\alpha_1}}\norm{\vect X^{\alpha_2} f}_{L^q(G)}\\
		&\meg 2^k C_3 C_4\max_{\abs{\alpha}\meg k} t^{d_\alpha-Q_*/q}\norm{\vect X^\alpha f}_{L^q(G)}
	\end{split}
	\]
	for every $t\in (0,\rho]$ and for every $f\in \Wp^{k,q}_t(G)$.
	
	\textsc{Step III} We now consider the general case. Applying~\textsc{step II} to $\langle f(x^{-1}\,\cdot\,), v'\rangle$\footnote{Notice that $\langle f, v'\rangle\in L^1_\loc(G)$ since $\abs{\langle f,v'\rangle}\meg \abs{f}$, so that we may adopt the usual notation for functions.} for every $x\in G$ and for every $v'\in \Banach'$ with $\norm{v'}_{\Banach'}\meg 1$, we then see that
	\[
	\norm{\langle f,v'\rangle}_{L^\infty(B(x,t/2))}\meg 2^k C_3 C_4 \norm{\langle f,v'\rangle}_{\Wp^{k,q}_t(G)}\meg 2^k C_3 C_4 \norm{ f}_{\Wp^{k,q}_t(G;\Banach)}
	\]
	for every $t\in (0,\rho]$ and for every $f\in \Wp^{k,q}_t(G)$, so that the assertion follows for $t\in (0,\rho]$. The case of $t\in [\rho,1]$ follows from the fact that $\norm{f}_{\Wp^{k,q}_\rho(G;\Banach)}\meg \rho^{-Q_*/q} \norm{f}_{\Wp^{k,q}_t(G;\Banach)}$ for every $f\in \Wp^{k,q}_t(G;\Banach)$.
\end{proof}

\begin{prop}\label{prop:43}
		Take two \BanachSpace s  $\Banach_1, \Banach_2 $, $p,q\in [1,\infty]$, and
		 $K\in \Dc'(G;\Lf(\Banach_1;\Banach_2))$. If there is a constant  $C>0$ such that
		\[
		\norm{f*K}_{L^q(G;\Banach_2)}\meg C \norm{f}_{L^p(G;\Banach_1)}
		\]
		for every $f\in C^\infty_c(G;\Banach_1)$, then  
		\[
		\norm{f*K}_{\Lp^q(G;\Banach_2)}\meg  C \norm{f}_{\Lp^p(G;\Banach_1)}
		\]
		for every $f\in \Lp^p(G;\Banach_1)$ with compact support.
	\end{prop}
	
	Notice that we consider only compactly supported functions only to provide simple conditions so that the convolution $f*K$ be defined. In addition, observe that $C^\infty_c(G;\Banach_1)$ need \emph{not} be dense in $\Lp^p(G;\Banach_1)$ even when $p<\infty$.
	
	\begin{proof}
		Take $(\phi_t)_{t\in (0,1]}$ as in Lemma~\ref{lem:55} and a compactly supported $f\in \Lp^p(G;\Banach_1)$, and observe that $\phi_t*f\in C^\infty_c(G;\Banach_1)$ and that $\abs{\phi_t*f}\meg \phi_t*\abs{f}$, so that 
		\[
		\norm{\phi_t*f}_{L^p(G;\Banach_1)}\meg \norm{f}_{\Lp^p(G;\Banach_1)}
		\]
		for every $t>0$. In addition, $(\phi_t*f)*K=\phi_t*(f*K)$ converges to $f*K$ in $\Dc'(G;\Banach_2)$ for $t\to 0^+$ and
\[
\norm{(\phi_t*f)*K}_{L^q(G;\Banach_2)}\meg C \norm{\phi_t*f}_{L^p(G;\Banach_1)}\meg C \norm{f}_{\Lp^p(G;\Banach_1)}
\]
 for every $t\in (0,1]$. This proves that $f*K\in \Lp^q(G;\Banach_2)$ and that \[
 \norm{f*K}_{\Lp^q(G;\Banach_2)}\meg  C \norm{f}_{\Lp^p(G;\Banach_1)}, 
 \]
 whence the result.
	\end{proof}

\begin{proof}[Proof of Proposition~\ref{prop:31}.]
	\textsc{Step I} Fix $\phi\in C^\infty_c(\R)$ so that $\chi_{[-1,1]}\meg \phi\meg \chi_{[-2,2]}$ and set $\psi_j\coloneqq \Kc_{2^{-j\grado }\Lc}(\phi )$ for every $j\in\N$. Let us prove that there is a constant $C_1>0$ such that, for every $j\in\N$ and for every $X\in U(\gf)$ with order $\meg q(\lambda)$,
	\[
	\abs{(X^R K*\psi_j)(x)} \meg \frac{ C_1\abs{X}}{\abs{x}_*^{Q_*+\deg(X)} }
	\]
	for every $x\in B(e,1)$, while
	\[
	\abs{(X^R K*\psi_j)(x)}\meg \frac{ C_1\abs{X}}{(1+\abs{x}_*)^{a} }
	\]
	for every $x\in G\setminus B(e,1)$.  
	To this aim, take $(\tau_t)_{t>0}$ as in Corollary~\ref{cor:21} with $\kappa_1=1$ and $\kappa_2=2$, and define $\eta_j\coloneqq\tau_{2^{-j}}$, so that $\chi_{B(e,2^{-j})}\meg \eta_j\meg \chi_{B(e,2^{1-j})}$ and there is a constant $C_2\Meg 1$ such that
	\[
	\abs{X Y^R \eta_j}\meg C_2\abs{X}\abs{Y} 2^{j(\deg(X)+\deg(Y))}
	\] 
	for every $X,Y\in U(\gf)$ with order $\meg q(\lambda)$, and for every $j\in\N$.  
	Now, take $x\in G\setminus B(e,1)$, and observe that, by Lemmas~\ref{cor:12} and~\ref{lem:37}, there is a constant $C_3>0$ such that
	\[
	\begin{split}
		\abs{[((1-\eta_2) X^R K)*\psi_j](x)}	&\meg \abs{X} \int_{G\setminus B(e,1/4)}\frac{1}{\abs{y }_*^a} \frac{C_3 2^{jQ_*}}{(1+2^j\abs{y^{-1}x}_*)^{2a }} \,\dd \beta(y)\\
		&\meg \abs{X} \int_{G\setminus B(e,1/4)}\frac{ 2^{(3+j)a}}{(1+2^j\abs{y }_*)^a} \frac{C_3 2^{jQ_*}}{(1+2^j\abs{y^{-1}x}_*)^{2a }} \,\dd \beta(y)\\
		&\meg  \abs{X} \frac{2^{(3+j)a}  C_3}{(1+2^j\abs{x}_*)^a}\int_{G }  \frac{2^{jQ_*}}{ (1+2^j\abs{y^{-1}x}_*)^{a }} \,\dd \beta(y)\\
		&\meg \abs{X} \frac{ C_3^2 }{\abs{x}_*^a}\\
		&\meg \abs{X} \frac{2^{a} C_3^2 }{(1+\abs{x}_*)^a}
	\end{split}
	\]
	for every $j\in\N$. 
	Now, observe that $\abs{\vect X^\alpha}\meg 1$  and that
	\[
	\begin{split}
		[( \eta_2 (\vect{X}^\alpha)^R K)*\psi_j](x)&=\langle \eta_2 (\vect{X}^\alpha)^R K, \psi_j( \,\cdot\,^{-1}x)\rangle \\
        &=[[\check \eta_2\psi_j(\,\cdot\,x)]* (\vect{X}^\alpha)^R K](e)\\
        &=[[\vect{X}^\alpha(\check \eta_2\psi_j(\,\cdot\,x))]* K](e)
	\end{split}
	\]
    for every  $\alpha\in\N^J$.
	Keep the notation of Lemma~\ref{lem:76}, and consider the spaces $\Wp^{k,q}_t(G^\circ;\Banach)$, where $G^\circ$ and $\Banach$ denote the opposite group of $G$ and a \BanachSpace, respectively. In other words, 
	\[
	\norm{f}_{\Wp^{k,q}_t(G^\circ;\Banach)}= \max_{\abs{\alpha}\meg k} t^{d_\alpha-Q_*/q}\norm{(\vect X^\alpha)^R f}_{\Lp^q(G;\Banach)}
	\]
	for every $f\in \Sr'(G;\Banach)$.
	By Lemma~\ref{lem:76}, if we take an integer $k>\dim G/q,q(\lambda)$, then we may find a constant $C_4>0$ such that
    \[
    \norm{f}_{\Lp^\infty(G;\Banach_2)}\meg C_4  \norm{f}_{\Wp^{k,q}_t(G^\circ;\Banach_2)}
    \]
    for every $t\in (0,1]$ and for every $f\in \Wp^{k,q}_t(G^\circ;\Banach_2)$. 
	In addition, since right-invariant differential operators commute with convolution on the right with $K$,  it is also clear that  \[
    \norm{f*K}_{\Wp^{k,q}_t(G^\circ;\Banach_2)}\meg   \norm{f}_{\Wp^{k,q}_t(G^\circ;\Banach_1)}
    \]
    for every compactly supported $f\in \Wp^{k,q}_t(G^\circ;\Banach_1)$ and for every $t\in (0,1]$ (see~Proposition~\ref{prop:43}).	
	Observe that there is a family $(a_{\alpha ,\gamma})_{\abs{\alpha}\Meg\abs{\gamma}}$ of real analytic functions on $G$ such that $\vect X^\alpha=\sum_{\abs{\gamma}\meg \abs{\alpha}} a_{\alpha,\gamma} (\vect X^\gamma)^R$ for every $\alpha\in\N^J$.
	
	The previous remarks then show that
	\[
	\begin{split}
		\abs{ [( \eta_2 (\vect{X}^\alpha)^R K )*(\psi_j v_1)](x)  } &\meg \sum_{\abs{\gamma}\meg \abs{\alpha}} 
		\abs{  [[a_{\alpha,\gamma}(\vect{X}^\gamma)^R(\check \eta_2\psi_j(\,\cdot\,x))v_1]* K ](e) }\\
		&\meg C_4 \sum_{\abs{\gamma}\meg \abs{\alpha}}  \norm{  [a_{\alpha,\gamma}(\vect{X}^\gamma)^R(\check \eta_2\psi_j( \,\cdot\,x))v_1]* K   }_{\Wp^{k,q}_1(G^\circ;\Banach_2)}\\
		&\meg C_4 \sum_{\abs{\gamma}\meg \abs{\alpha}}    \norm{  a_{\alpha,\gamma}(\vect{X}^\gamma)^R (\check \eta_2\psi_j( \,\cdot\,x)) v_1 }_{\Wp^{k,q}_1(G^\circ;\Banach_1)}\\
		&\meg  C_4 \sum_{\abs{\gamma}\meg \abs{\alpha}}    \norm{  a_{\alpha,\gamma}(\vect{X}^\gamma)^R (\check \eta_2\psi_j( \,\cdot\,x)) }_{\Wp^{k,q}_1(G^\circ)}
	\end{split}
	\]
	for every $v_1\in \Banach_1$ with $\norm{v_1}_{\Banach_1} \meg 1$, and for every $j\in\N$.  Consequently, using Proposition~\ref{cor:12} we see that there is a constant $C_5>0$ such that
	\[
	\begin{split}
		\abs{    [( \eta_2 (\vect{X}^\alpha)^R  K)*\psi_j](x) }&\meg C_5 \max_{\abs{\gamma}\meg k} \norm{ \chi_{B(e,1/2)} [(\vect{X}^\gamma)^R \psi_j]( \,\cdot\,x)}_{L^q(G)}\\
		&\meg  C_5^2 \max_{\abs{\gamma}\meg k} \sup_{y\in B(e,1/2)} \frac{2^{j(Q_*+d_\gamma) }}{(1+2^j\abs{yx}_*)^{a+d_\gamma+Q_* }}\\
		&\meg  C_5^2 \max_{\abs{\gamma}\meg k} \frac{2^{j(Q_* +d_\gamma)}}{(2^{j-1}\abs{x}_*)^{a+d_\gamma+Q_*}}\\
		&\meg  \frac{ C_5^2    }{ \abs{x}_* ^{a}}\max_{\abs{\gamma}\meg k}  2^{a+d_\gamma+Q_*}\\
		&\meg \frac{ C_5^2    }{ (1+\abs{x}_*) ^{a}}\max_{\abs{\gamma}\meg k}  2^{3a+d_\gamma+Q_*}
	\end{split}
	\]
	for every  $j\in\N$, since $\abs{y x}_*\Meg \abs{x}_*-1/2\Meg \abs{x}_*/2$ for every $y\in B(e,1/2)$. This concludes the proof of the first assertion.

	Now, take $x\in B(e,1 )\setminus \Set{e}$. Observe that, for $\alpha$ with length $\meg q(\lambda)$, 
	\[
	\begin{split}
		\abs{(\vect X^\alpha)^R [((1-\eta_{j+1})K)(y)}&\meg \sum_{\alpha_1+\alpha_2=\alpha}\frac{\alpha!}{\alpha_1!\alpha_2!} \abs{  ((\vect{X}^{\alpha_1})^R(1-\eta_{j+1}) (\vect{X}^{\alpha_2} )^R K)(y)}\\
		&\meg C_2\chi_{B(e,2^{-j})\setminus B(e, 2^{-j-1})}(y)\sum_{\alpha_1+\alpha_2=\alpha, \alpha_1\neq 0}\frac{\alpha!}{\alpha_1!\alpha_2!} \frac{2^{j d_{\alpha_1}}}{\abs{y}_*^{Q_*+d_{\alpha_2}}}\\
		&\qquad+ \chi_{G\setminus B(e, 2^{-j-1})}(y)\frac{1}{\abs{y}_*^{Q_*+d_{\alpha}}}\\
		&\meg C_2\chi_{G\setminus B(e, 2^{-j-1})}(y)2^{\abs{\alpha}} \frac{1}{\abs{y}_*^{Q_*+d_{\alpha}}}
	\end{split}
	\]
	for every $y\in B(e,1)$, while 
	\[
	\begin{split}
		\abs{(\vect X^\alpha)^R ((1-\eta_{j+1})K)(y)}&=  \abs{(\vect{X}^{\alpha} )^R K)(y)} \meg    \frac{1}{\abs{y}_*^{a}}
	\end{split}
	\]
	for every $y\in G\setminus B(e,1)$. Therefore, by means of Proposition~\ref{cor:12} we see that there is a constant $C_6>0$ such that, for $\alpha$ with  $\abs{\alpha} \meg q(\lambda)$,
	\[
	\begin{split}
		&\abs{(\vect X^\alpha)^R [((1-\eta_{j+1})K)*\psi_j](x)}\\
        &\qquad\meg C_6 \int_{B(e,1)\setminus B(e,2^{-j-1})} \frac{1}{\abs{y}_*^{Q_*+d_\alpha}} \frac{2^{jQ_*}}{(1+2^j\abs{y^{-1}x}_*)^{Q_*+d_\alpha+a}}\,\dd \beta(y)  +C_6\int_{G\setminus B(e,1)} \frac{\abs{\psi_j(y)}}{\abs{y}_*^{a}}  \,\dd \beta(y)\\
		&\qquad\meg C_6 \int_{B(e,1)\setminus B(e,2^{-j-1})} \frac{2^{(j+2)(Q_*+d_\alpha)}}{  (1+ 2^j\abs{y}_*)^{Q_*+d_\alpha}} \frac{2^{jQ_*}}{(1+2^j\abs{y^{-1}x}_*)^{Q_*+d_\alpha+a}} \,\dd \beta(y) +C_6^2\\
		&\qquad\meg C_6  \frac{2^{(j+2)(Q_*+d_\alpha)}}{  (1+ 2^j\abs{x}_*)^{Q_*+d_\alpha}}\int_G\frac{2^{jQ_*}}{(1+2^j\abs{y^{-1}x}_*)^{a}} \,\dd \beta(y)   +C_6^2\\
		&\qquad\meg C_6^2  \frac{2^{2(Q_*+d_\alpha)}}{  \abs{x}_*^{Q_*+d_\alpha}}+C_6^2\\
	\end{split}
	\]
	thanks to Lemma~\ref{lem:37}. Next, take a family $(a'_{\alpha,\gamma})_{\abs{\alpha}\Meg \abs{\gamma}}$ of real analytic functions on $G$ such that $(\vect X^\alpha)^R =\sum_{\abs{\gamma}\meg \abs{\alpha}} a'_{\alpha,\gamma} \vect X^\gamma$, and observe that, by Lemma~\ref{lem:72}, there is a constant $C_7>0$ such that $\abs{a'_{\alpha,\gamma}(y)}\meg C_7\abs{y}_*^{(d_\gamma-d_\alpha)_+}$ for every $y\in B(e,1)$ and for every  $\alpha,\gamma$ with length $\meg k$. Then,  
	\[
	\begin{split}
		[   (\vect{X}^\alpha)^R ( \eta_{j+1} K)*\psi_j](x)&=\sum_{\abs{\gamma}\meg \abs{\alpha}} a'_{\alpha,\gamma}(x) [    ( \eta_{j+1} K)*\vect X^{\gamma}\psi_j](x)\\
		&= 	\sum_{\abs{\gamma}\meg \abs{\alpha}} a'_{\alpha,\gamma}(x)\langle \eta_{j+1}  K, \eta_j(\vect{X}^\gamma\psi_j)( \,\cdot\,^{-1}x)\rangle \\
		&=\sum_{\abs{\gamma}\meg \abs{\alpha}} a'_{\alpha,\gamma}(x)[[\check\eta_j(\vect{X}^\gamma\psi_j)(\,\cdot\,x)]* ( \eta_{j+1}  K)](e) .
	\end{split}
	\]
	Arguing as above, we then see that
	\[
	\begin{split}
		&\abs{[[\check\eta_j(\vect{X}^\gamma\psi_j)(\,\cdot\,x)v_1]* ( \eta_{j+1}  K)](e)} \meg C_4  \norm{[\check\eta_j(\vect{X}^\gamma\psi_j)(\,\cdot\,x)v_1]* ( \eta_{j+1}  K)}_{\Wp^{k,q}_{2^{-j}}(G^\circ;\Banach_2)} \\
		&\qquad\meg C_4  \norm{[\check\eta_j(\vect{X}^\gamma\psi_j)(\,\cdot\,x)v_1]*   K}_{\Wp^{k,q}_{2^{-j}}(G^\circ;\Banach_2)}		\\
		&\qquad\qquad + C_4  \norm{[\check\eta_j(\vect{X}^\gamma\psi_j)(\,\cdot\,x) v_1]*  ((1- \eta_{j+1} ) K)}_{\Wp^{k,q}_{2^{-j}}(G^\circ;\Banach_2)}		\\
		&\qquad\meg C_4   \norm{ \check\eta_j(\vect{X}^\gamma\psi_j)(\,\cdot\,x) }_{\Wp^{k,q}_{2^{-j}}(G^\circ)}\\
		&\qquad\qquad+C_4  2^{jQ_*/q}\norm{ \check\eta_j(\vect{X}^\gamma\psi_j)(\,\cdot\,x)]}_{\Wp^{k,1}_{2^{-j}}(G^\circ) } \norm{(1- \eta_{j+1} )\abs{ K}}_{L^q(G)}
	\end{split}
	\]
	for every $v_1\in \Banach_1$ with $\norm{v_1}_{\Banach_1} \meg 1$.
	Now, using Proposition~\ref{cor:12} we see that there is a constant $C_8>0$ such that 
	\[
	\begin{split}
		\norm{ \check\eta_j(\vect{X}^\gamma\psi_j)(\,\cdot\,x) }_{\Wp^{k,r}_{2^{-j}}(G^\circ)}&\meg C_8\max_{\abs{\alpha'}\meg k} 2^{ -j (d_{\alpha'}-Q_*/r)}\norm{(\vect{X}^\gamma(\vect X^{\alpha'})^R\psi_j)(\,\cdot\,x) }_{L^r(B(e,2^{1-j}))}\\
		&\meg C_8^2 2^{-j d_{\alpha'}}\sup_{\abs{y}_* \meg 2^{1-j}}\frac{ 2^{j(d_{\alpha'}+d_\gamma+Q_*)} }{(1+2^j\abs{y x}_*)^{Q_*+d_\gamma}}\\
		&\meg C_8^2 2^{(j+1)(d_{\gamma}+Q_*)} \frac{ 1 }{(2^j\abs{x}_*)^{Q_*+d_\gamma}}\\
		&=C_8^2 \frac{ 2^{ d_{\gamma}+Q_* } }{ \abs{x}_*^{Q_*+d_\gamma}}
	\end{split}
	\]
	for every $r\in [1,\infty]$. In addition,   Lemmas~\ref{lem:37} and~\ref{lem:37b} show that there is a constant $C_9>0$ such that
	\[
	\norm{(1- \eta_{j+1} ) \abs{K}}_{L^q(G)}\meg C_9 2^{-jQ_*/q}
	\]
	since $q>1$.  Therefore,
	\[
	\begin{split}
		\abs{[   (\vect{X}^\alpha)^R ( \eta_{j+1} K)*\psi_j](x)}&\meg C_7 \sum_{\abs{\gamma}\meg \abs{\alpha}} \abs{x}_*^{(d_\gamma-d_\alpha)_+} \frac{ C_4   C_8^2 2^{d_\gamma+Q_*}+ C_4 C_8^2 C_9}{ \abs{x}_*^{Q_*+d_\gamma}}\\
		&\meg C_7 \sum_{\abs{\gamma}\meg \abs{\alpha}}  \frac{ C_4   C_8^2 2^{d_\gamma+Q_*}+ C_4 C_8^2 C_9}{ \abs{x}_*^{Q_*+d_\alpha}}.
	\end{split}
	\]
	This completes the proof of our assertions.  
	
	\textsc{Step II}  We now prove the statement.    
    By Lemma~\ref{lem:71}, it will suffice to show that there is a constant $C_{10}>0$ such that
	\[
	\Bigl\lVert \sup_{j\in \N} \abs{f*K*\psi_j}\Bigr\rVert_{L^p(G)}\meg C_{10}
	\]
	for every $\Banach_1$-valued  $(p,q, \rho,\lambda)$-atom $f$.
    Take $c\Meg 1$ and $\eps>0$ as in  Lemma~\ref{lem:73}, and observe that we may assume that $\rho \meg \eps$.
	Observe that $f$ is concentrated on $B(x,r)$ for some $x\in G$ and for some $r>0$, and that $\norm{f}_{\Lp^q(G;\Banach_1)}\meg \beta(B(x,r))^{1/q-1/p}$, so that by Propositions~\ref{prop:42} and~\ref{prop:43}, there is a constant $C_{11}>1$ such that
	\[
	\begin{split}
		\Bigl\lVert\chi_{B(x,2cr)}\sup_{j\in \N} \abs{f*K*\psi_j}\Bigr\rVert_{L^p(G)} &\meg \Bigl\lVert\sup_{j\in \N} \abs{f*K*\psi_j}\Bigr\rVert_{L^q(G)} \beta(B(x,2cr))^{1/p-1/q}\\
		&\meg C_{11} \norm{f*K}_{\Lp^q(G;\Banach_2)} \beta(B(x,2cr))^{1/p-1/q}\\
		&\meg   C_{11} \norm{f}_{\Lp^q(G;\Banach_1)} \beta(B(x,2cr))^{1/p-1/q}\\
		&\meg C_{11}^2
	\end{split}
	\]
	since $q>1$ and $\beta $ is doubling.  
	
	In order to deal with $\norm*{(1-\chi_{B(x,2cr)})\sup_{j\in \N} \abs{f*K*\psi_j}}_{L^p(G)}$, assume first that $r>\rho$. Then, there is a constant $C_{12}>0$ such that, for every $y\in G\setminus B(x,2cr)$ and for every $j\in\N$,
	\[
	\begin{split}
		\abs{(f*K*\psi_j)(y)}&\meg C_1\int_{B(x,r)}\frac{ \abs{f(z)}}{\abs{z^{-1}y}_*^a}\,\dd \beta(z)\\
		&\meg 3^a C_1\int_{B(x,r)}\frac{ \abs{f(z)}}{(2r+\abs{z^{-1}y}_*)^a}\,\dd \beta(z)\\
		&\meg 3^a C_1\beta(B(x,r))^{1/q-1/p}\frac{\beta(B(x,r))^{1/q'} }{(r+\abs{y^{-1}x}_*)^a}\\
		&\meg \frac{C_{12} r^{Q_G(1-1/p)}}{(\rho+\abs{y^{-1}x}_*)^a},
	\end{split}
	\]
	so that by Lemma~\ref{lem:37} there is a constant $C_{13}>0$ such that
	\[
	\Bigl\lVert(1-\chi_{B(x,2cr)}) \sup_{j\in \N} \abs{(f*K*\psi_j)(y)}\Bigr\rVert_{L^p(G)}\meg C_{13}
	\]
	since $p\meg 1$.
	Then, assume that $r\meg \rho$. 
	Then, by Lemma~\ref{lem:73}  there are a constant $C_{14}>1$ and a family $(f_\alpha)_{0<\abs{\alpha}\meg q(\lambda)}$ of elements of $\Lp^q(G;\Banach_1)$ such that $f=\sum_{0<\abs{\alpha}\meg q(\lambda)} \vect X^\alpha f_\alpha$, such that $\supp f_\alpha \subseteq B(x,cr)$, and such that 
	\[
	\norm{f_\alpha}_{\Lp^q(G;\Banach_1)}\meg C_{14} r^{\max(\lambda, d_\alpha)} \norm{f}_{\Lp^q(G;\Banach_1)}\meg C_{14} r^{\max(\lambda, d_\alpha)} \beta(B(x,r))^{1/q-1/p}
	\]
	for every $\alpha$.  
	Consequently, there is a constant $C_{15}>0$ such that,   for every $j\in\N$ and for every $y\in B(x,2c\rho)\setminus B(x,2cr)$ (so that $\abs{y^{-1}x}_*/2\meg \abs{z^{-1}y}_*\meg 3c\rho$ for every $z\in B(x,cr)$),  
	\[
	\begin{split}
		&\abs{(f*K*\psi_j)(y)}\meg  \sum_{0<\abs{\alpha}\meg q(\lambda)} \abs{[f_\alpha*(\vect X^\alpha)^R(K*\psi_j)](y)}\\
		&\qquad\meg C_1 \sum_{0<\abs{\alpha}\meg q(\lambda)}\max(3c\rho,1)^{Q_*+d_\alpha} \int_{B(x,cr)}\frac{ \abs{f_\alpha(z)}}{ \abs{z^{-1}y}_*^{Q_*+d_\alpha}}\,\dd \beta(z)\\
		&\qquad\meg C_1  C_{14}  \sum_{0<\abs{\alpha}\meg q(\lambda)}\max(3c\rho,1)^{Q_*+d_\alpha} r^{\max(\lambda, d_\alpha)} \beta(B(x,r))^{1/q-1/p}\frac{2^{Q_*+d_\alpha}\beta(B(x,cr))^{1/q'} }{ \abs{y^{-1}x}_*^{Q_*+d_\alpha}}\\
		&\qquad\meg C_{15}  \sum_{0<\abs{\alpha}\meg q(\lambda)}   \frac{ r^{Q_*(1-1/p)+\max(\lambda,d_\alpha )}}{\abs{y^{-1}x}_*^{Q_*+d_\alpha}}.
	\end{split}
	\]
	Then, by Lemma~\ref{lem:37b} there is a constant $C_{16}>0$ such that\footnote{Notice that here we use the assumption $\lambda>0$: if $\lambda=0$ (which implies $p=1$ in view of the condition $\lambda \Meg Q_*(1/p-1)$), then there would be no decomposition $f=\sum \vect X^\alpha f_\alpha$ and here we would need to estimate $\norm{(\chi_{B(x,2\rho)}-\chi_{B(x,2r)}) \abs{\,\cdot\,}_*^{-Q_*}}_{L^1(G)}$, which behaves like $\log (1/r)$ for $r\to 0^+$ and cannot therefore stay bounded.}
	\[
	\begin{split}
		\Bigl\lVert(\chi_{B(x,2c\rho)}-\chi_{B(x,2cr)}) \sup_{j\in \N} \abs{ f*K*\psi_j }\Bigr\rVert_{L^p(G)}&\meg C_{16} \sum_{0<\abs{\alpha}\meg q(\lambda)} r^{Q_*(1-1/p)+\max(\lambda,d_\alpha )-Q_*(1-1/p)-d_\alpha}\\
		&\meg C_{16}^2.
	\end{split}
	\]
	Analogously, we may find a constant $C_{17}>0$ such that,   for every $j\in\N$ and for every $y\in G\setminus B(x,2c\rho)$,  
	\[
	\begin{split}
		\abs{(f*K*\psi_j)(y)}&\meg  \sum_{0<\abs{\alpha}\meg q(\lambda)} \abs{[f_\alpha*(\vect X^\alpha)^R(K*\psi_j)](y)}\\
		&\meg C_{17}  \sum_{0<\abs{\alpha}\meg q(\lambda)}\int_{B(x,cr)}\frac{ \abs{f_\alpha(z)}}{ \abs{z^{-1}y}_*^{a}}\,\dd \beta(z)\\
		&\meg C_{17}  C_{14}  \sum_{0<\abs{\alpha}\meg q(\lambda)}r^{\max(\lambda, d_\alpha)} \beta(B(x,r))^{1/q-1/p}\frac{2^{a}\beta(B(x,cr))^{1/q'} }{ \abs{y^{-1}x}_*^{a}}\\
		&\meg C_{17}^2 \sum_{0<\abs{\alpha}\meg q(\lambda)} \frac{   r^{Q_*(1-1/p)+\max(\lambda,d_\alpha )}}{\abs{y^{-1}x}_*^{a}}.
	\end{split}
	\]
	By means of Lemma~\ref{lem:37} we may then find a constant $C_{18}>0$ such that
	\[
	\Bigl\lVert(1-\chi_{B(x,2c\rho)}) \sup_{j\in \N} \abs{ f*K*\psi_j }\Bigr\rVert_{L^p(G)}\meg C_{18} \sum_{0<\abs{\alpha}\meg q(\lambda)} r^{Q_*(1-1/p)+\max(\lambda,d_\alpha )}\meg C_{18}^2
	\]
	since $\lambda \Meg Q_*(1/p-1)$.		
	The assertion follows. 
\end{proof}

\end{document}